\documentclass[11pt]{article}
\usepackage{microtype} 
\usepackage{url}
\usepackage{hyperref}
\hypersetup{colorlinks, linkcolor=blue}
\usepackage{amsfonts}
\usepackage{amsthm}
\usepackage{amsmath}
\usepackage{amscd}
\usepackage{amssymb}
\usepackage{caption,subcaption}
\usepackage[shortlabels]{enumitem}
\usepackage{mathtools}
\usepackage{tikz}
\usepackage{tikz-cd}
\usetikzlibrary{arrows,arrows.meta}
\tikzset{>=latex}
\tikzcdset{arrow style=tikz, diagrams={>=latex}}
\usepackage{xcolor}

\def\AA {{\mathbb A}}     
\def\CC {{\mathbb C}}     
\def\LL {{\mathbb L}}     
\def\PP {{\mathbb P}}     
\def\QQ {{\mathbb Q}}     
\def\RR {{\mathbb R}}     
\def\XX {{\mathbb X}}     
\def\ZZ {{\mathbb Z}}     

\def\t  {\mathfrak{tr}}   

\def\ring#1{\ifmmode \mathaccent'027 #1\else \rm\accent'027 #1\fi}

\newcommand{\asm}{{\mathfrak m}} 	

\def\ol  {\overline}

\def\mc {\mathcal}
\def\mk {\mathfrak}

\def\Ob {\mathrm{Ob}}
\def\Hom {\mathrm{Hom}}

\def\Ext {\mathrm{Ext}}

\mathchardef\mhyphen="2D

\newtheorem{theorem}{Theorem}[section]

\newtheorem{lemma}[theorem]{Lemma}
\newtheorem{prop}[theorem]{Proposition}
\newtheorem{coro}[theorem]{Corollary}
\theoremstyle{definition}
\newtheorem{remark}[theorem]{Remark}
\newtheorem{df}[theorem]{Definition}
\newtheorem{ex}[theorem]{Example}
\theoremstyle{plain}

\makeatletter
\def\smallunderbrace#1{\mathop{\vtop{\m@th\ialign{##\crcr
   $\hfil\displaystyle{#1}\hfil$\crcr
   \noalign{\kern3\p@\nointerlineskip}%
   \tiny\upbracefill\crcr\noalign{\kern3\p@}}}}\limits}
\makeatother

\numberwithin{equation}{section}

\newcommand{\Addresses}{{
  \bigskip
  \footnotesize

	(F.~Haiden) \textsc{Syddansk Universitet, Campusvej 55, 5230 Odense, Denmark} \par\nopagebreak
	\textit{E-mail:} \texttt{fab@sdu.dk}
	\medskip
	
}}

\begin{document}

\title{3-d Calabi--Yau categories for Teichm\"uller theory}
\author{Fabian Haiden}
\date{}

\maketitle

\begin{abstract}
For $g,n\geq 0$ a 3-dimensional Calabi-Yau $A_\infty$-category $\mathcal C_{g,n}$ is constructed such that a component of the space of Bridgeland stability conditions, $\mathrm{Stab}(\mathcal C_{g,n})$, is a moduli space of quadratic differentials on a genus $g$ surface with simple zeros and $n$ simple poles.
For a generic point in the moduli space the corresponding quantum/refined Donaldson--Thomas invariants are computed in terms of counts of finite-length geodesics on the flat surface determined by the quadratic differential.
As a consequence, these counts satisfy wall-crossing formulas.
\end{abstract}

\setcounter{tocdepth}{2}
\tableofcontents

\section{Introduction}

\subsection{Stability conditions on 3CY categories}

The Fukaya category, $\mc F(M)$, is a rich invariant of the symplectic topology of a compact Calabi--Yau manifold, $M$, but is independent of its complex structure. 
In an intricate conjectural picture developed by Thomas~\cite{thomas01}, Thomas--Yau~\cite{thomas_yau02}, and in particular Joyce~\cite{joyce_conj} the complex structure on $M$, together with a holomorphic volume form $\Omega^{n,0}$, gives rise to a Bridgeland stability condition on $\mc F(M)$ whose central charge records the periods of $\Omega^{n,0}$ and whose semistable objects are, roughly speaking, the special Lagrangian submanifolds.
Variants of this conjecture for certain non-compact Calabi--Yau spaces have been proposed in~\cite[Section 7.3]{ks_wcs}. 
This includes, in particular, the 3-folds considered in~\cite{dddhp} in the context of large $N$ duality.
These have the structure of affine quadric bundles, $X$, over a Riemann surface $C$ with nodal singularities over finitely many points in $C$.
The holomorphic volume form $\Omega^{3,0}$ on $X$ then gives rise to, by fiberwise integration, a holomorphic quadratic differential $\varphi$ (section of the square of the canonical bundle) on $C$ with simple zeros.
Conversely, much of the geometry of $X$ can be recovered from the pair $(C,\varphi)$.
This suggests that $\mc F(X)$, together with its stability condition, can be constructed starting from $(C,\varphi)$ instead of $X$.
This is essentially the strategy of this paper, except that we do not prove here that the 3-d Calabi--Yau (3CY) categories which we construct are equivalent to Fukaya categories of 3-folds, though we expect this to be true.
Our first main result is the following (Theorem~\ref{thm_stab_C} in the main text).

\begin{theorem}
\label{thm1_intro}
Fix $g,n\geq 0$, where $n\geq 4$ if $g=0$ and $n\geq 2$ if $g=1$, and let $\mc Q_{g,n}$ be the moduli space of Riemann surfaces of genus $g$ together with a quadratic differential with simple zeros and $n$ simple poles and a marking (see below). 
Then there is a 3-d Calabi--Yau category $\mc C_{g,n}$ such that a component of the space of stability conditions, $\mathrm{Stab}(\mc C_{g,n})$, is identified with $\mc Q_{g,n}$.
\end{theorem}

The spaces $\mc Q_{g,n}$ which appear in the above theorem are special cases of the spaces of quadratic differentials considered in~\cite{hkk} and are defined as follows.
Choose a surface $S$ of genus $g$, a set $M\subset S$ of $4g-4+2n$ marked points, and a foliation $\nu\in\Gamma(S\setminus M;\PP(TS))$ which is homotopic to the horizontal foliation $\mathrm{hor}(\varphi):=\{v\mid \varphi(v,v)\in\RR_{\geq 0}\}$ of a (hence any) quadratic differential with $4g-4+n$ simple zeros and $n$ simple poles at $M$.
An element of $\mc Q_{g,n}$ is represented by a triple $(J,\varphi,h)$ where $J$ is a complex structure on $S$, $\varphi$ is a quadratic differential with $\mathrm{Zeros}(\varphi)\cup\mathrm{Poles}(\varphi)=M$ and $h$ is a homotopy (path in the space of sections) from $\mathrm{hor}(\varphi)$ to $\nu$.
Two triples $(J_i,\varphi_i,h_i)$, $i=1,2$ are equivalent if there is an orientation preserving diffeomorphism $f:S\to S$ with $f_*J_1=J_2$, $f_*\varphi_1=\varphi_2$ and such that if $\tilde{f}$ is the homotopy from $f^*\nu$ to $\nu$ which is the concatenation of the inverse of $f^*h_2$ and $h_1$, then the pair $(f,\tilde f)$ is isotopic to $(\mathrm{id},\mathrm{id})$.

A result similar to Theorem~\ref{thm1_intro}, but with $\mc C_{g,n}$ replaced by the wrapped Fukaya category of the punctured surface, $\mc F(S\setminus M,\nu)$, (and also for more general types of quadratic differentials), is proven in~\cite{hkk}. 
Our proof of Theorem~\ref{thm1_intro} combines this result with a method of transferring stability conditions from $\mc F(S\setminus M,\nu)$ to the 3CY category $\mc C_{g,n}$.

\subsubsection{Construction of the 3CY categories $\mc C_{g,n}$.}
We give a brief sketch. 
Fix the surface $S$, marked points $M\subset S$, and the foliation $\nu$ as before.
The $A_\infty$ category $\mc F(S\setminus M,\nu)$ is defined as in~\cite{hkk} and admits an additional $\ZZ/2$-grading by equipping curves $c:I\to S$, which are objects in the category, with a choice of orientation of $\nu$ (= choice of $\sqrt{\varphi}$ if $\nu=\mathrm{hor}(\varphi)$) along $c$.
The $\ZZ\times \ZZ/2$-graded version of $\mc F(S\setminus M,\nu)$ admits a deformation, $\mc A(S,\nu)$, over $\mathbf k[[t]]$, the $\ZZ\times\ZZ/2$-graded $\mathbf k$-algebra, commutative in the $\ZZ\times\ZZ/2$-graded sense, freely generated by a variable $t$ of degree $(-1,\mathrm{odd})$.
This deformation is defined geometrically by counting immersed disks in $M$ with coefficient $t^k$, where $k$ is the number of times the disk covers a point in $M$.
The category $\mc C_{g,n}$ is then obtained as torsion modules over the category $\mc A(S,\nu)$, i.e. those supported on the ``central fiber'' $\mc F(S\setminus M,\nu)$.
By construction, there is a base-change adjunction between the wrapped Fukaya category and $\mc C_{g,n}$, and we use it to transfer stability conditions from the former to the latter.

\textit{To our knowledge, Theorem~\ref{thm1_intro} presents the first instance where a component of the space of stability conditions is determined for a 3CY category not coming from a quiver with potential.}

\subsubsection{Relation to work of Bridgeland--Smith.}
Our result complements celebrated work of Bridgeland--Smith~\cite{bs} who prove a result similar to Theorem~\ref{thm1_intro}, but for moduli spaces of quadratic differentials with at least one higher order pole (although quadratic differentials without higher order poles are included in the moduli space as non-generic strata). 
The 3CY categories in~\cite{bs} are constructed from certain quivers with potential corresponding to ideal triangulations of the surface. 
\textit{In contrast, our categories $\mc C_{g,n}$ do not come from quivers with potential, and therefor require completely different methods.}
We also note that quadratic differentials without higher order poles, the case considered here, are more interesting from the point of view of flat geometry/ergodic theory.
For example, a generic geodesic with respect to the (flat) metric $|\varphi|$ converges to a higher order pole of $\varphi$, should one exist, but winds around densely on the surface if there are none.

Smith furthermore shows in~\cite{smith15} that the 3CY categories considered in~\cite{bs} are full subcategories of Fukaya categories of non-compact Calabi--Yau threefolds, essentially variants of those considered in~\cite{dddhp} adapted to the case of higher order poles.
As stated above, we expect a similar result for the categories $\mc C_{g,n}$ constructed in this paper.
Such an equivalence, together with the results of this paper, would confirm an expectation in~\cite[Subsection 3.4]{smith_stabsymp}.

\subsection{DT invariants}

In seminal work~\cite{ks}, Kontsevich--Soibelman devised a way of ``counting'' semistable objects in 3CY categories with stability condition, the so-called \textit{motivic Donaldson--Thomas invariants}. 
These greatly expand the scope of the invariants originally proposed by Donaldson--Thomas~\cite{donaldson_thomas} and rigorously defined by Thomas~\cite{thomas00}, building on the contributions of many researchers.

Having constructed a 3CY category with stability condition for any quadratic differential $\varphi$ with simple zeros and/or simple poles on a Riemann surface $C$, one can ask what its DT invariants are. 
We solve this problem for generic $\varphi$, expressing the answer in terms of numbers of finite-length geodesics with respect to the flat metric $|\varphi|$.

\subsubsection{Saddle connections and geodesic loops.}
The metric $|\varphi|$ has the following local structure: Away from the zeros and poles, it is locally isometric to the flat Euclidean plane, with local isometric charts given by $\int\sqrt{\varphi}$.
At each of the simple zeros and simple poles, the metric has a conical singularity with cone angle $3\pi$ and $\pi$, respectively.

There are two types of finite length (unbroken) geodesics on a flat surface:
\begin{enumerate}
\item 
\textbf{Saddle connections} start and end at conical points (possibly the same) and do not meet any conical points along the way.
\item
\textbf{Closed geodesics} are periodic trajectories which do not meet any conical points. 
These always appear in one-parameter families foliating a maximal \textbf{flat cylinder} (unless $g(C)=1$ and $\varphi$ is constant) whose boundary circles consist of one or more saddle connections.
\end{enumerate}
For any $L>0$ the number of saddle connections of length $\leq L$ or flat cylinders of circumference $\leq L$ is finite, and grows quadratically in $L$ as shown by Masur~\cite{masur88,masur90}.
Determining the number of finite-length geodesics (as a function of $L$) more precisely is one of the main themes in the subject of flat surfaces.

Instead of counting finite-length geodesics shorter than a given length, we consider more refined counts which keep track of homology classes.
Consider the lattice 
\[
\Gamma:=H_1(C,M;\ZZ\sqrt{\varphi})
\]
where $M$ is the set of zeros and poles of $\varphi$ and $\ZZ\sqrt{\varphi}$ is the local system of abelian groups of $\ZZ$-multiples of choices of $\pm\sqrt{\varphi}$.
Then $\Gamma$ depends, up to isomorphism, only on the genus, $g$, of $C$ and the number, $n$, of poles of $\varphi$, and has a natural non-degenerate anti-symmetric pairing $\Gamma\times\Gamma\to\ZZ$ (see Subsection~\ref{subsec_pairing} for details).
The periods of $\sqrt{\varphi}$ give a map $Z:\Gamma\to \CC$, and the moduli space $\mc Q_{g,n}$ is locally modeled on $\Hom(\Gamma,\CC)$.
Any saddle connection or closed geodesic defines a class in $\Gamma$, up to sign.
Our genericity condition on $(C,\varphi)$ is then the following:

\begin{df}
A pair $(C,\varphi)$ is \textit{generic} if any $\gamma_1,\gamma_2\in\Gamma$ supporting finite length geodesics of the same slope, i.e. with $\mathrm{Arg}(Z(\gamma_1))=\mathrm{Arg}(Z(\gamma_2))$, are $\QQ$-linearly dependent.
\end{df}

The second main result of this paper is then the following (Theorem~\ref{thm2} in the main text).

\begin{theorem}
\label{thm2_intro}
Fix a generic quadratic differential in $\mc Q_{g,n}$,
then the refined DT invariants of the corresponding stability condition on $\mc C_{g,n}$ are given by
\[
\Omega(\gamma)=N_{++}(\gamma)+2N_{+-}(\gamma)+4N_{--}(\gamma)+\left(q^{1/2}+q^{-1/2}\right)N_c(\gamma)
\]
where $\gamma\in\Gamma$ and $N_{++}(\gamma)$ (resp. $N_{+-}(\gamma)$, resp. $N_{--}(\gamma)$) is the number of saddle connections with class $\gamma$ between distinct zeros (resp. a zero and a pole, resp. distinct poles) of $\varphi$ and $N_c(\gamma)$ is the number of flat cylinders with class $\gamma$.
\end{theorem}

The numerical DT invariants are obtained by setting $q^{1/2}=-1$:
\[
\Omega(\gamma)_{\mathrm{num}}=N_{++}(\gamma)+2N_{+-}(\gamma)+4N_{--}(\gamma)-2N_c(\gamma).
\]
As a consequence of Theorem~\ref{thm2_intro}, \textit{the numbers $\Omega(\gamma)$ and $\Omega(\gamma)_{\mathrm{num}}$ satisfy the Kontsevich--Soibelman wall-crossing formula~\cite{ks} as one moves in the moduli space $\mc Q_{g,n}$}.
This answer a question of Iwaki--Kidwai~\cite[Remark 3.13]{iwaki_kidwai}.

It is a consequence of a central ergodicity result in the theory of flat surfaces due to H.~Masur and W.~A.~Veech, that the generic orbit of the action of $\mathrm{GL}^+(2,\RR)$ on the quotient of $\mc Q_{g,n}$ by the mapping class group is \textit{dense}.
By Theorem~\ref{thm1_intro}, this can be stated entirely in terms of stability conditions on $\mc C_{g,n}$, and appears to be a rare phenomenon among triangulated categories, c.f.~\cite{h_conjstab}.
Thus we have both:
\begin{enumerate}
\item The $GL^+(2,\RR)$ action relating $\Omega(\gamma)$ for points on the same orbit in $\mc Q_{g,n}$.
\item The wall-crossing formula relating $\Omega(\gamma)$ for ``nearby'' points in $\mc Q_{g,n}$.
\end{enumerate}
This should impose strong constraints on the DT invariants $\Omega(\gamma)$, which are worth exploring further.

\subsection{Outline}

The text is organized as follows.
In Section~\ref{sec_ainfty} we develop the theory of curved $A_\infty$-categories over the algebra $\mathbf k[[t]]$ where $t$ has (bi-)degree of the form $(d,d\mod 2)\in\ZZ\times\ZZ/2$. 
We define our main examples of such categories in Section~\ref{sec_3cysurf} starting from a surface with some additional decorations.
This is an intermediate step in the construction of the proper 3CY categories $\mc C_{g,n}$, completed in the same section.
Our first main result, identifying a component of $\mathrm{Stab}(\mc C_{g,n})$ with a module space of quadratic differentials, is deduced in Section~\ref{sec_stability} from a more general result allowing the transfer of stability conditions along certain adjunctions.
The second main result, expressing DT invariants in terms of counts of saddle connections and flat cylinders, is established in Section~\ref{sec_dt}, which also includes a very brief introduction to motivic DT theory. 
Finally, Appendix~\ref{sec_qdilog} collects some facts around the quantum dilogarithm needed in Section~\ref{sec_dt}.

\vspace{\baselineskip}

\subsubsection{Acknowledgements} We thank Tom Bridgeland, Alexander Goncharov, and Andy Neitzke for encouragement and interest in this work, Dylan Allegretti for discussing his paper~\cite{allegretti}, Merlin Christ and Tobias Dyckerhoff for interest in this work and discussions on perverse schobers, Ben Davison for suggestions on how to compute DT invariants, Kohei Iwaki for bringing his work with Omar Kidwai to our attention, Ivan Smith for explaining his related work on 3CY categories, and Jon Woolf for discussing his paper~\cite{woolf12}.
Finally, we thank the anonymous referees for carefully reviewing the manuscript and valuable comments.
The author was supported by a Titchmarsh Research Fellowship during the writing of this paper.

\section{$A_\infty$-categories over a formal disk}
\label{sec_ainfty}

The first step in the construction of the 3CY categories $\mc C_{g,n}$ is to consider a deformation of the wrapped Fukaya category of the punctured surface over a polynomial ring $\mathbf k\langle t\rangle$, where $t$ has degree $-1$.
Two features of this deformation are:
\begin{enumerate}
\item It is necessary to consider an \textit{additional $\mathbb Z/2$ grading}, i.e. work with $\ZZ\times\ZZ/2$-graded vector spaces.
\item The deformation is a \textit{curved} $A_\infty$-category, i.e. has not only structure maps $\mk m_1$ (differential), $\mk m_2$ (product), $\mk m_3$ (homotopy up to which $\mk m_2$ is associative), $\mk m_4$, \ldots, but also $\mk m_0:\mathbf k\to \Hom^2(X,X)$ for every object $X$.
\end{enumerate}
The first has the effect of making all signs a bit more complicated, while the second is standard in the Fukaya category literature.
For the benefit of the non-expert, we will briefly explain the role of both before discussing details in the rest of this section.

\subsubsection{Additional $\ZZ/2$ grading}
$A_\infty$-categories can be defined over fields and more generally \textit{commutative} rings.
The Koszul sign rule $a\otimes b=(-1)^{|a||b|}b\otimes a$ dictates that in the free (super-) commutative algebra generated by an odd variable $t$, $\mathbf k[t]$, we have $t^2=0$. (Assume $\mathrm{char}(\mathbf k)\neq 2$.)
In other words, the polynomial algebra $\mathbf k\langle t\rangle$ is not commutative in the $\ZZ$-graded sense for $|t|$ odd.
This can be fixed by working instead with $\ZZ\times \ZZ/2$-graded objects and the sign rule 
\begin{equation}\label{ZZ2_braiding}
a\otimes b=(-1)^{|a||b|+\pi(a)\pi(b)}b\otimes a
\end{equation}
where $(|a|,\pi(a))\in\ZZ\times\ZZ/2$ is the bi-degree of $a$ (c.f. ``Point of View II'' in \cite{strings_course}).
If we give $t$ the $\ZZ/2$-degree $\pi(t)=|t|\mod 2$, then the free commutative algebra generated by $t$ now has underlying ungraded algebra the usual algebra of polynomials.

There is also a more geometric explanation in the context of the construction in this paper.
We have a punctured surface $S$ with a $T^*S^2$ (affine quadric surface) bundle over it.
If the bundle were trivial, then we would expect the (wrapped) Fukaya category of the total space to be the tensor product of the Fukaya category of the surface and the Fukaya category of $T^*S^2$, which happens to be $\mathrm{Perf}(\mathbf k\langle t\rangle)$, $|t|=-1$.
In general however, the bundle will have monodromy which reverses the orientation of $S^2$ and thus sends $t\mapsto -t$ in $\mathbf k\langle t\rangle$. In this case we expect the Fukaya category of the total space to be a twisted tensor product, i.e. a tensor product of $\ZZ\times\ZZ/2$-graded categories.

\subsubsection{Curved categories}

A general feature of families of categories which appear in geometry is that an object in a fiber of the family can have any number of extensions to the total family.
The following example illustrates this phenomenon.
For the benefit of the algebro-geometrically minded reader, we switch to the other side of homological mirror symmetry. 

Let $R:=\mathbf k[x,y]/(xy)$ be the coordinate ring of the nodal plane conic and $S:=\mathbf k[x,y,t]/(xy-t)$ its deformation to a smooth conic.
Suppose we want to extend the $R$-module $M:=R/(x)$ to an $S$-module in the derived sense, i.e. find $\widetilde{M}$ with $\widetilde{M}\otimes^{\mathbf L}_SR\cong M$.
The module $M$ has a 2-periodic resolution
\[
0\leftarrow N\leftarrow R\xleftarrow{x} R\xleftarrow{y} R\xleftarrow{x} R\xleftarrow{y} R \leftarrow \ldots
\]
in terms of free modules.
As ansatz for $\widetilde{M}$ we replace $R$ in the above resolution by $S$ and perturb the differential by terms in $tS$:
\[
S\xleftarrow{x+t\epsilon_1} S\xleftarrow{y+t\epsilon_2} S\xleftarrow{x+t\epsilon_3} S\xleftarrow{y+t\epsilon_4} S \leftarrow \ldots
\]
However, this is not a chain complex for any choice of $\epsilon_i\in S$, but instead a kind of ``curved object''.
Moreover, $M$ has indeed no extension to $S$ and is therefore \textit{obstructed}.

The $R$-module $N:=R/(x-1)$, on the other hand, has a resolution 
\[
0\leftarrow N\leftarrow R\xleftarrow{x-1}R\leftarrow 0
\]
which extends to
\[
0\leftarrow \widetilde{N}\leftarrow S\xleftarrow{x-1+t\epsilon}S\leftarrow 0
\]
corresponding to structure sheaves of plane curves which intersect the $x$-axis exactly in $x=1$ and do not intersect the $y$-axis.

In the framework of \textit{curved $A_\infty$-categories} the equation which one solves to find an extension of an object, $E$, to the total family is the $A_\infty$ Maurer--Cartan equation
\begin{equation*}
\sum_{n=0}^\infty\mk m_n(\delta,\ldots,\delta)=0
\end{equation*}
for $\delta\in\Hom^1(E,E)$.
If $\mk m_0=0$ then one has at least the trivial solution $\delta=0$, but if $\mk m_0\neq 0$ then there are possibly no solutions (the obstructed case).

\subsection{$A_\infty$-categories and functors}

$A_\infty$-categories are a generalization of differential graded (DG-) categories in which composition is only required to be associative up to homotopy, and this homotopy needs to satisfy a coherence condition, again up to homotopy, and so on.
We start with the ``vanilla'' definition without the extensions discussed in the introduction to this section.
Throughout, we fix a ground field $\mathbf k$.

\begin{df}
An \textbf{$A_\infty$-category} over $\mathbf k$, $\mc A$, is given by a collection of \textit{objects}, $\Ob(\mc A)$, for each pair $X$, $Y$ of objects a $\ZZ$-graded vector space $\Hom(X,Y)$ of \textit{morphisms}, and for $n\geq 1$ and $X_0,\ldots,X_n\in\Ob(\mc A)$ \textit{structure maps}
\[
\mk m_n:\Hom(X_{n-1},X_{n})\otimes\cdots\otimes\Hom(X_0,X_1)\to\Hom(X_0,X_n)
\] 
of degree $2-n$, satisfying the \textit{$A_\infty$-equations}
\begin{equation} \label{ainfty_eq}
\sum_{i+j+k=n}(-1)^{\|a_1\|+\ldots+\|a_i\|}\mk m_{i+1+k}(a_n,\ldots,\mk m_{j}(a_{i+j},\ldots,a_{i+1}),\ldots,a_1)=0
\end{equation}
where $\|a\|:=|a|-1$ is the \textit{reduced degree}.
Furthermore, we require the existence of \textit{strict units} $1_X\in\Hom^0(X,X)$ with
\begin{gather*}
\mk m_2(a,1_X)=a=(-1)^{|a|}\mk m_2(1_Y,a) \\
\mk m_n(\ldots,1_X,\ldots)=0,\qquad  n\neq 2
\end{gather*}
where $a\in\Hom(X,Y)$ is homogeneous.
\end{df}

$A_\infty$-categories with $\mk m_n=0$ for $n\geq 3$ correspond to DG-categories up to a sign change: The differential is $da=(-1)^{|a|}\mk m_1(a)$ and the composition is $ab=(-1)^{|b|}\mk m_2(a,b)$.

\begin{remark}
A note on signs: We follow the conventions as in~\cite{seidel08}, which are common in the Fukaya category literature. They are based on the degree in the bar complex, $\|a\|$, and have the advantage of leading to overall simpler signs, but are somewhat unnatural, as seen in the comparison with DG-categories above.
\end{remark}

Next, we define $A_\infty$-categories over $\mathbf k[[t]]$, where $t$ has bi-degree $(|t|,\pi(t))=(d,d\mod 2)\in\ZZ\times\ZZ/2$.
We regard $\mathbf k[[t]]$ as a commutative algebra in the category of $\ZZ\times\ZZ/2$-graded vector spaces over $\mathbf k$ with braiding isomorphism~\eqref{ZZ2_braiding}.
The underlying ungraded algebra is the algebra of power series $\mathbf k[[t]]$ if $|t|=0$  and the algebra of polynomials $\mathbf k[t]$ if $|t|\neq 0$. 
The use of double brackets in the notation is justified in both cases though, since, if $|t|\neq 0$, $\mathbf k[t]$ is complete as a graded algebra (with the usual $t$-adic filtration). Indeed, power series in $t$ are \textit{not} a graded algebra if $|t|\neq 0$. 

As morphism spaces in $A_\infty$-categories over $\mathbf k[[t]]$ we will only allow modules which are \textit{topologically free} in the following sense:
If $M$ is a $\ZZ\times\ZZ/2$-graded vector space then the module $M\otimes_{\mathbf k}\mathbf k[[t]]$ has a completion with respect to its $t$-adic filtration which we denote by $M[[t]]$.
Concretely, an element of $M[[t]]^{m,n}$ is given by a formal power series $\sum_{k=0}^\infty a_kt^k$ with $a_k\in M^{m-k|t|,n-k\pi(t)}$.
In particular, $M\otimes_{\mathbf k}\mathbf k[[t]]$ is already complete whenever $M$ is finite-dimensional.
A $\mathbf k[[t]]$-module is \textbf{topologically free} if it is isomorphic to one of the form $M[[t]]$.

\begin{df}
An \textbf{$A_\infty$-category over $\mathbf k[[t]]$}, $\mc A$, is given by a collection of objects $\Ob(\mc A)$, for each pair $X$, $Y$ of objects a topologically free $\mathbf k[[t]]$-module $\Hom(X,Y)$ of morphisms, and for $n\geq 0$ and $X_0,\ldots,X_n\in\Ob(\mc A)$ structure maps
\[
\mk m_n:\Hom(X_{n-1},X_{n})\otimes_{\mathbf k}\cdots\otimes_{\mathbf k}\Hom(X_0,X_1)\to\Hom(X_0,X_n)
\]
of bi-degree $(2-n,0)$ such that $\mk m_0\in\Hom^{2,0}(X,X)t$, the $A_\infty$-equations \eqref{ainfty_eq} hold, now allowing $j=0$, i.e. terms with $\mk m_0$. 
Furthermore, the $\mk m_n$ should be $\mathbf k[[t]]$-multilinear in the sense that
\begin{gather*}
\mk m_n(a_n,\ldots,a_it,\ldots,a_1)=(-1)^{*}\mk m_n(a_n,\ldots,a_i,\ldots,a_1)t \\
*=|t|(\|a_1\|+\ldots+\|a_{i-1}\|+1)+\pi(t)(\pi(a_1)+\ldots+\pi(a_{i-1}))
\end{gather*}
(The $+1$ in the sign comes from thinking of $\mk m_n$ as acting on the right and having reduced degree $1$.)
Finally, strict units are required to have bi-degree $(0,0)$.
\end{df}

Note that $\mk m_0$ is really a map $\mathbf k\to\Hom^{2,0}(X,X)$ for each $X\in\Ob$, which we will usually identify with the element $\mk m_0(1)\in\Hom^{2,0}(X,X)$.
According to our definition, it should vanish at $t=0$. We refer the interested reader to the introduction of~\cite{positselski} for a discussion on this requirement.

It is useful to reformulate the above definition in terms of Taylor coefficients.
Suppose we have chosen isomorphisms of $\mathbf k[[t]]$-modules
\[
\Hom(X,Y)\cong\ol{\Hom}(X,Y)[[t]]
\]
for any pair of objects $X,Y$, where $\ol{\Hom}(X,Y):=\Hom(X,Y)/(\Hom(X,Y)t)$.
Then we can expand the structure maps $\asm_n$ in terms of powers of $t$ as
\[
\mk m_n(a_n,\ldots,a_1)=:\sum_{k\geq 0}\mk m_{n,k}(a_n,\ldots,a_1)t^k
\]
where
\[
\mk m_{n,k}:\ol{\Hom}(X_{n-1},X_n)\otimes\cdots\otimes\ol{\Hom}(X_0,X_1)\to\ol{\Hom}(X_0,X_n)
\]
are $\mathbf k$-linear maps of degree $2-n-k|t|$ and parity $k\pi(t)$.

\begin{lemma}
The components $\mk m_{n,k}$, as above, define an $A_\infty$-category over $\mathbf k[[t]]$ if and only if
\begin{equation}
\label{a_infty_eta}
\sum_{\substack{r+s=d \\ i+j+k=n}}(-1)^{*}\mk m_{i+1+k,r}(a_n,\ldots,\mk m_{j,s}(a_{i+j},\ldots,a_{i+1}),\ldots,a_1)=0
\end{equation}
for $n,d\geq 0$, where the sign is given by
\[
*=(s|t|+1)(\|a_1\|+\ldots+\|a_i\|)+s|t|(\pi(a_1)+\ldots+\pi(a_i)+1),
\]
as well as $\mk m_{0,0}=0$, $\mk m_{2,0}(a,1_X)=a=(-1)^{|a|}\mk m_{0,2}(1_Y,a)$, and $\mk m_{n,k}(\ldots,1_X,\ldots)=0$ for $n\neq 2$ or $k>0$.
\end{lemma}

\begin{proof}
Given an $A_\infty$-category over $\mathbf k[[t]]$, expand $\mk m_n$ in terms of $\mk m_{n,k}$ in the $A_\infty$-equation for the $\mk m_n$, then \eqref{a_infty_eta} is the coefficient of $t^d$.
From the $A_\infty$-equations we get a sign $\|a_1\|+\ldots+\|a_i\|$ and moving $t^s$ past $a_1,\ldots,a_i$ and $\mk m_n$ gives an additional sign $s|t|(\|a_1\|+\ldots+\|a_i\|+\pi(a_1)+\ldots+\pi(a_i)+1)$.

Conversely, any $\mk m_{n,k}$ extend uniquely to $\mathbf k[[t]]$-multilinear maps
\begin{gather*}
\left(\ol{\Hom}(X_{n-1},X_{n})\otimes_{\mathbf k}\mathbf k[[t]]\right)\otimes_{\mathbf k}\cdots\otimes_{\mathbf k}\left(\ol{\Hom}(X_0,X_1)\otimes_{\mathbf k}\mathbf k[[t]]\right) \to \\
\to\ol{\Hom}(X_0,X_n)\otimes_{\mathbf k}\mathbf k[[t]]
\end{gather*}
and to continuous maps
\[
\ol{\Hom}(X_{n-1},X_{n})[[t]]\otimes_{\mathbf k}\cdots\otimes_{\mathbf k}\ol{\Hom}(X_0,X_1)[[t]]\to\ol{\Hom}(X_0,X_n)[[t]]
\]
by topological freeness. 
The equation \eqref{a_infty_eta} is then equivalent to \eqref{ainfty_eq} and similarly for strict unitality.
\end{proof}

So far we have considered $A_\infty$-categories over $\mathbf k$ and (curved) $A_\infty$-categories over $\mathbf k[[t]]$. 
Neither is a special case of the other, but they both fit within the larger framework of \textit{filtered} $A_\infty$-categories in the following sense (c.f.~\cite{h_skein}).

\begin{df}
A \textbf{filtered $A_\infty$-category over $\mathbf k$}, $\mc A$, is given by a collection of objects $\Ob(\mc A)$, for each pair $X$, $Y$ of objects a vector space $\Hom(X,Y)$ with decreasing filtration
\[
\Hom(X,Y)\supseteq \Hom(X,Y)_{\geq 0}\supseteq \Hom(X,Y)_{\geq 1} \supseteq \Hom(X,Y)_{\geq 2} \supseteq \ldots
\]
which is complete and Hausdorff, and for $n\geq 0$ and $X_0,\ldots,X_n\in\Ob(\mc A)$ structure maps
\[
\mk m_n:\Hom(X_{n-1},X_{n})\otimes_{\mathbf k}\cdots\otimes_{\mathbf k}\Hom(X_0,X_1)\to\Hom(X_0,X_n)
\]
of degree $2-n$ which are filtration preserving (with respect to the natural filtration on the tensor product) and such that $\mk m_0\in\Hom^2(X,X)_{\geq 1}$, the $A_\infty$-equations \eqref{ainfty_eq} hold. 
We also require the existence of strict units as before.
\end{df}

An $A_\infty$-category over $\mathbf k$ is filtered by $\Hom(X,Y)_{\geq 0}:=\Hom(X,Y)$ and $\Hom(X,Y)_{\geq 1}=0$.
An $A_\infty$-category over $\mathbf k[[t]]$ is filtered by the $t$-adic filtration, $\Hom(X,Y)_{\geq k}=\Hom(X,Y)t^k$.
We can also consider $\ZZ\times\ZZ/2$-graded filtered $A_\infty$-categories.

\begin{df}
An \textbf{$A_\infty$-functor} $F$ between filtered $A_\infty$-categories $\mc A$ and $\mc B$  is given by a map $F:\Ob(\mc A)\to\Ob(\mc B)$ and a collection of filtration preserving maps
\[
F_n:\Hom_{\mc A}(X_{n-1},X_n)\otimes\cdots\otimes \Hom_{\mc A}(X_0,X_1)\to\Hom_{\mc B}(FX_0,FX_n)
\]
of degree $1-n$, $n\geq 1$, such that the $A_\infty$-equations
\begin{equation}\label{ainfty_mor}
\begin{gathered}
\sum_{i_1+\ldots+i_k=n}\mk m_k(F_{i_k}(a_n,\ldots,a_{n-i_k+1}),\ldots,F_{i_1}(a_{i_1},\ldots,a_1)) \\
=\sum_{i+j+k=n}(-1)^{\Vert a_k\Vert+\ldots+\Vert a_1\Vert}F_{i+1+k}(a_n,\ldots,\mk m_j(a_{n-i},\ldots,a_{k+1}),a_k,\ldots,a_1).
\end{gathered}
\end{equation}
hold.
We further require $A_\infty$-functors to be \textit{strictly unital}, i.e.
\[
F_1(1_X)=1_{FX},\qquad F_n(\ldots,1_X,\ldots)=0\qquad\text{for }n>1.
\]
\end{df}

\begin{remark}
One can also consider \textit{curved} $A_\infty$-functors, which are defined similarly but allow non-zero $F_0=F_{0,X}:\mathbf k\to\Hom^1_{\mc B}(FX,FX)$ for every $X\in\mc A$.
The left-hand side of~\eqref{ainfty_mor} then has potentially infinitely many terms, so some condition on $F_0$ is needed to ensure convergence, e.g. $F_{0,X}\in\Hom^1(X,X)_{\geq 1}$.
However, one can replace such a curved functor by an uncurved one by sending $X$ to the \textit{twisted complex} (deformed object) $(X,F_{0,X}(1))$ instead.
Twisted complexes are defined in the next subsection. 
\end{remark}

$A_\infty$-functors are composed as follows:
\begin{align*}
(F\circ G)_n(a_n,\ldots,a_1)&:= \\ \sum_{i_1+\ldots+i_k=n}& F_k(G_{i_k}\left(a_n,\ldots,a_{n-i_k+1}),\ldots,G_{i_1}(a_{i_1},\ldots,a_1)\right)
\end{align*}
Finally, a \textbf{natural transformation}, $\lambda$, from $F:\mc A\to\mc B$ to $G:\mc A\to\mc B$ is given by a collection of maps
\[
\lambda_n:\Hom_{\mc A}(X_{n-1},X_n)\otimes\cdots\otimes \Hom_{\mc A}(X_0,X_1)\to\Hom_{\mc B}(FX_0,GX_n)
\]
of degree $-n$, $n\geq 0$, such that
\begin{gather*}
\sum_{\substack{i_1+\ldots+i_k=n \\ 1\leq j\leq k}}(-1)^{\Vert a_{i_1+\ldots+i_{j-1}}\Vert+\ldots+\Vert a_1\Vert}\mk m_k(G_{i_k}(a_n,\ldots,a_{n-i_k+1}),\ldots, \\
G_{i_{j+1}}(\ldots,a_{i_1+\ldots+i_j+1}),
\lambda_{i_j}(a_{i_1+\ldots+i_j},\ldots,a_{i_1+\ldots+i_{j-1}+1}), \\
F_{i_{j-1}}(a_{i_1+\ldots+i_{j-1}},\ldots),\ldots,F_{i_1}(a_{i_1},\ldots,a_1)) \\
=\sum_{i+j+k=n}(-1)^{\Vert a_{k}\Vert+\ldots+\Vert a_1\Vert-1}\lambda_{i+1+k}(a_n,\ldots,\mk m_j(a_{n-i},\ldots,a_{k+1}),a_k,\ldots,a_1).
\end{gather*}
for $n\geq 0$.
We will only need to consider natural transformations $\lambda$ with $\lambda_n=0$ for $n\geq 1$, see Subsection~\ref{subsubsec_basechange}.

\begin{remark}
$A_\infty$-functors between a fixed pair of $A_\infty$-categories form an $A_\infty$-category themselves, whose closed degree zero morphisms are the natural transformations in the above sense, see e.g.~\cite{seidel08}.
\end{remark}

\subsection{Twisted complexes}
\label{subsec_twisted}

\textit{Twisted complexes} can be viewed as a generalization, to $A_\infty$-algebras and -categories, of the notion of a chain complex of free modules (of finite total dimension) over a ring. 
There are different flavors depending on the precise conditions on the differential.
When the differential is represented by a strictly upper-triangular matrix, the twisted complex has an interpretation as an iterated extension, and thus categories of such twisted complexes give a concrete model for the closure under finite direct sums, shifts, and cones.
However, we will also come across twisted complexes where the differential is not strictly upper triangular, which still makes sense in the setting of filtered $A_\infty$-categories.

\subsubsection{Additive closure}
The first step in the construction of categories of twisted complexes is to form the closure under finite direct sums and shifts.
As in the previous section, we start with the definition in the case of $A_\infty$-categories over $\mathbf k$, i.e. uncurved and $\ZZ$-graded, following~\cite[Section 3k]{seidel08}, then discuss modifications.

\begin{df}
Let $\mc A$ be an $A_\infty$-category over $\mathbf k$, then its \textbf{additive closure}, $\mathrm{Add}(\mc A)$, is the $A_\infty$-category defined as follows.
Objects are formal direct sums of tensor products
\[
X=\bigoplus_{i\in I}V_i\otimes X_i
\]
where $I$ is a finite set, $V_i$ are finite-dimensional $\ZZ$-graded vector spaces over $\mathbf k$, and $X_i\in\Ob(\mc A)$.
Morphisms are 
\[
\Hom\left(\bigoplus_{i\in I}V_i\otimes X_i,\bigoplus_{j\in J}W_j\otimes Y_j\right):=\bigoplus_{i,j}\Hom(V_i,W_j)\otimes_{\mathbf k}\Hom(X_i,Y_j)
\]
with the natural $\ZZ$-grading.
Structure maps are given by
\begin{equation}
\label{add_closure_maps}
\mk m_n(\phi_n\otimes a_n,\ldots,\phi_1\otimes a_1):=(-1)^{\sum_{i<j}|\phi_i|\|a_j\|}\phi_n\cdots\phi_1\otimes\mk m_n(a_n,\ldots,a_1)
\end{equation}
extended $\mathbf k$-linearly.
\end{df}

In the case where $\mc A$ has an additional $\ZZ/2$ grading, the $V_i$'s are instead $\ZZ\times\ZZ/2$-graded vector spaces, i.e. we consider formal shifts in arbitrary bi-degree.
The sign in~\eqref{add_closure_maps} should then be
\[
(-1)^{\sum_{i<j}\left(|\phi_i|\|a_j\|+\pi(\phi_i)\pi(a_j)\right)}
\]
to take the additional grading into account.

When $\mc A$ is an $A_\infty$-category over $\mathbf k[[t]]$, $\mathrm{Add}(\mc A)$ is so too.
This uses the fact that topologically free modules are closed under finite direct sums and shifts.
Similarly, when $\mc A$ is a filtered $A_\infty$-category, then so is $\mathrm{Add}(\mc A)$.

\subsubsection{Twisting cochains}
The second step is to add ``formal deformations'' of objects in $\mathrm{Add}(\mc A)$.
Here it makes sense to start with the case of filtered $A_\infty$-categories.

\begin{df}
Let $\mc A$ be a filtered $A_\infty$-category over $\mathbf k$, then its \textbf{category of (two-sided) twisted complexes}, $\mathrm{Tw}(\mc A)$, is the uncurved, filtered $A_\infty$-category defined as follows.
Objects of $\mathrm{Tw}(\mc A)$ are pairs $(X,\delta)$ where $X\in\mathrm{Add}(\mc A)$ and $\delta\in\Hom^1(X,X)$, the \textit{twisting cochain}, so that the constant term, $\delta_0\in\Hom^1(X,X)/\Hom^1(X,X)_{\geq 1}$, is \textit{strictly upper triangular}, i.e. if $X=\bigoplus_{i\in I}V_i\otimes X_i$ then there is a total order on $I$ so that the component of $\delta_0$ in $\Hom(V_i\otimes X_i,V_j\otimes X_j)/\Hom(V_i\otimes X_i,V_j\otimes X_j)_{\geq 1}$ vanishes for $i\leq j$, and furthermore $\delta$ satisfies the \textit{($A_\infty$-)Maurer--Cartan} equation
\begin{equation}
\label{mc_eqn}
\sum_{n=0}^\infty\mk m_n(\delta,\ldots,\delta)=0
\end{equation}
where the series converges since $\delta_i$'s are strictly upper triangular mod higher order terms.
Morphisms are the same as in $\mathrm{Add}(\mc A)$:
\[
\Hom_{\mathrm{Tw}(\mc A)}((X_1,\delta_1),(X_2,\delta_2)):=\Hom_{\mathrm{Add}(\mc A)}(X_1,X_2)
\]
and structure maps $\widetilde{\mk m}_n$ are obtained by ``inserting $\delta$'s everywhere'':
\begin{equation}\label{twisted_structure_maps}
\widetilde{\mk m}_n(a_n,\ldots,a_1):=\sum_{k_0,\ldots,k_n\ge 0}\mk m_{n+k_0+\ldots+k_n}(\underbrace{\delta_n,\ldots,\delta_n}_{k_n\text{ times}},a_n,\ldots,a_1,\underbrace{\delta_0,\ldots,\delta_0}_{k_0\text{ times}})
\end{equation}
where $a_i\in\Hom((X_{i-1},\delta_{i-1}),(X_i,\delta_i))$.
Here the sum is again infinite, but converges by the assumption on $\delta$.
\end{df}

We note that without imposing the Maurer--Cartan equation~\eqref{mc_eqn} on $\delta$, $\mathrm{Tw}(\mc A)$ would be a filtered $A_\infty$-category \textit{with curvature}, i.e. with $\widetilde{\mk m}_0\neq 0$ in general.
The Maurer--Cartan equation is just $\widetilde{\mk m}_0=0$.
Because $\mathrm{Tw}(\mc A)$ is uncurved, we can view it as an $A_\infty$-category over $\mathbf k$ (forgetting the filtration).

Recall that we can consider (uncurved, unfiltered) $A_\infty$-categories over $\mathbf k$ as filtered by the discrete filtration with $\Hom(X,Y)_{\geq 1}=0$.
Then any twisting cochain $\delta$ is strictly upper triangular, not just modulo higher order terms.
More generally, for possibly filtered $\mc A$, we write $\mathrm{Tw}^+(\mc A)$ for the full subcategory of $\mathrm{Tw}(\mc A)$ of pairs $(X,\delta)$ with $\delta$ strictly upper triangular, the \textit{one-sided twisted complexes}.
An uncurved $A_\infty$-category, $\mc A$ is \textbf{triangulated} if the natural inclusion of $\mc A$ into $\mathrm{Tw}^+(\mc A)$ is a quasi-equivalence.
In particular, $\mathrm{Tw}^+(\mc A)$ is triangulated, the triangulated closure of $\mc A$.

\begin{remark}
The notation $\mathrm{Tw}^+(\mc A)$ is not standard and this category is denoted by ``$\mathrm{Tw}(\mc A)$'' in \cite{seidel08}. 
\end{remark}

If $\mc A$ is a filtered $A_\infty$-category with additional $\ZZ/2$ grading, then there is a full subcategory of $\mathrm{Tw}(\mc A)$, denoted $\mathrm{Tw}_{\ZZ/2}(\mc A)$, of pairs $(X,\delta)$ where $\delta$ is \textit{even}, thus has bi-degree $(1,0)$.
Then $\mathrm{Tw}_{\ZZ/2}(\mc A)$ is $\ZZ\times\ZZ/2$-graded since the twisted structure maps $\widetilde{\mk m}_n$ are even.
Finally, if $\mc A$ is an $A_\infty$-category over $\mathbf k[[t]]$, then so is $\mathrm{Tw}_{\ZZ/2}(\mc A)$, and moreover an uncurved one.

\paragraph*{\textbf{Warning: }}
The category $\mathrm{Tw}_{\ZZ/2}(\mc A)$ is generally not triangulated, since only cones over \textit{even} morphisms have been added, but general morphisms have both an even and an odd component. 
The larger category $\mathrm{Tw}(\mc A)$, while not $\ZZ/2$-graded, still has a $\ZZ/2$-action, and subcategory of objects fixed under the action is precisely $\mathrm{Tw}_{\ZZ/2}(\mc A)$.

\subsection{Base-change adjunction}
\label{subsubsec_basechange}

As before, let $\mc A$ be an $A_\infty$-category over $\mathbf k[[t]]$ where $t$ has bi-degree $(d,d\mod 2)$ for some $d\in \ZZ$.
Interpreting $\mc A$ as a family of categories over the formal disk, the ``central fiber'' of that family is the $A_\infty$-category $\mc A_0:=\mc A \otimes_{\mathbf k[[t]]}\mathbf k$ with objects $\Ob(\mc A_0)=\Ob(\mc A)$, morphisms
\[
\Hom_{\mc A_0}(X,Y):=\overline{\Hom}_{\mc A}(X,Y):=\Hom_{\mc A}(X,Y)\otimes_{\mathbf k[[t]]}\mathbf k
\]
and induced structure maps.
There is a functor  $F:\mc A\to \mc A_0$ (of $\ZZ\times\ZZ/2$-graded, filtered $A_\infty$-categories) which is the identity on objects, and the quotient map $\Hom(X,Y)\to \overline{\Hom}(X,Y)$ on morphisms.
The goal of this section is to construct the right adjoint, $G$, of $F$, more precisely a functor
\[
G:\mc A_0\to \mathrm{Tw}_{\ZZ/2}(\mc A)
\]
so that $F$ and $G$ induce an adjunction
\[
F:\mathrm{Tw}_{\ZZ/2}(\mc A) \longleftrightarrow \mathrm{Tw}_{\ZZ/2}(\mc A_0):G
\]
of $A_\infty$-categories.
(Here and elsewhere we use the same symbol for a functor and its extension to twisted complexes.)

To construct $G$ we choose vector space splittings 
\[
\Hom_{\mc A}(X,Y)=\overline{\Hom}_{\mc A}(X,Y)\oplus \left(\Hom_{\mc A}(X,Y)t\right)
\]
and write elements as $a=\bar{a}+\hat{a}t$ with $\bar{a}$ and $\hat{a}t$ in the first and second summand above, respectively.
Define $G$ on objects by 
\begin{equation}
G(X):=\left(X\oplus X[1-d,d],\begin{pmatrix} 0 & t \\ (-1)^d\hat{\mk m}_0 & 0\end{pmatrix}\right)
\end{equation}
and on morphisms by 
\begin{equation}
G_1(a):=\begin{pmatrix} a & 0 \\ (-1)^{d+\delta(a)}\hat{\mk m}_1(a) & a \end{pmatrix}
\end{equation}
where $\delta(a):=(1-d)\|a\|+d\pi(a)$.
There are also higher order terms given by
\begin{equation}
G_n(a_n,\ldots,a_1):=\begin{pmatrix} 0 & 0 \\ (-1)^{d+\delta(a_1)+\ldots+\delta(a_n)}\hat{\mk m}_n(a_n,\ldots,a_1) & 0 \end{pmatrix}
\end{equation}
for $n\geq 2$.

\begin{prop}
$G:\mc A_0\to \mathrm{Tw}_{\ZZ/2}(\mc A)$ as defined above is an $A_\infty$-functor of $\ZZ\times\ZZ/2$-graded $A_\infty$-categories.
\end{prop}

\begin{proof}
For notational purposes, consider the curved $A_\infty$-functor $G'$ with
\[
G'(X):=\left(X\oplus X[1-d,d],\begin{pmatrix} 0 & t \\ 0 & 0\end{pmatrix}\right),\qquad G'_0=\begin{pmatrix} 0 & 0 \\ (-1)^d\hat{\mk m}_0 & 0\end{pmatrix},
\]
and $G'_n=G_n$ for $n\geq 1$.
Then $G'$ satisfies the $A_\infty$-functor equation~\eqref{ainfty_mor} iff $G$ does.

The left-hand side of \eqref{ainfty_mor} is
\begin{equation}
\label{Geq_L}
\sum_{i_1+\ldots+i_k=n}\widetilde{\mk m}_k(G'_{i_k}(a_n,\ldots,a_{n-i_k+1}),\ldots,G'_{i_1}(a_{i_1},\ldots,a_1))
\end{equation}
where only partitions with at most one $i_j\neq 1$ contribute, since $G'_i$ is strictly lower triangular 2-by-2 for $i\neq 1$.
Also, $\widetilde{\mk m}_k=\mk m_k$ for $k\geq 2$, since the twisting cochain of $G'(X)$ involves only $t$, which is a scalar multiple of the identity, and strict unitality of the $\mk m_k$.
There are three types of terms in~\eqref{Geq_L}:
Terms of the form
\begin{equation}
\label{Geq_L1}
\begin{pmatrix} \mk m_n(a_n,\ldots,a_1) & 0 \\ 0 & \mk m_n(a_n,\ldots,a_1) \end{pmatrix}
\end{equation}
coming from $\mk m_n(G'_1(a_n),\ldots,G'_1(a_1))$ (which also contributes to \eqref{Geq_L3} below), terms of the form
\begin{equation}
\label{Geq_L2}
\begin{pmatrix} -\hat{\mk m}_n(a_n,\ldots,a_1)t & 0 \\ 0 & -\hat{\mk m}_n(a_n,\ldots,a_1)t \end{pmatrix}
\end{equation}
coming from 
\[
\mk m_2\left(\begin{pmatrix}0 & t \\ 0 & 0\end{pmatrix},G'_n(a_n,\ldots,a_1)\right)+\mk m_2\left(G'_n(a_n,\ldots,a_1),\begin{pmatrix}0 & t \\ 0 & 0\end{pmatrix}\right)
\]
which is a summand of $\widetilde{\mk m}_1(G'_n(a_n,\ldots,a_1))$, and terms of the form
\begin{equation}
\label{Geq_L3}
(-1)^{d+\delta(a_n)+\ldots+\delta(a_{i+1})}\begin{pmatrix} 0 & 0 \\ \mk m_n(a_n,\ldots,\hat{\mk m}_j(a_{i+j},\ldots,a_{i+1}),\ldots,a_1) & 0 \end{pmatrix}
\end{equation}
coming from $\mk m_n(G'_1(a_n),\ldots,G'_j(a_{i+j},\ldots,a_i),\ldots,G'_1(a_1))$.

On the other hand, the right-hand side of \eqref{ainfty_mor} is 
\begin{equation}
\label{Geq_R}
\sum_{i+j+k=n}(-1)^{\Vert a_k\Vert+\ldots+\Vert a_1\Vert}G'_{i+1+k}(a_n,\ldots,\bar{\mk m}_j(a_{n-i},\ldots,a_{k+1}),\ldots,a_1)
\end{equation}
which is a sum of two types of terms:
Terms of the form
\begin{equation}
\label{Geq_R1}
\begin{pmatrix} \bar{\mk m}_n(a_n,\ldots,a_1) & 0 \\ 0 & \bar{\mk m}_n(a_n,\ldots,a_1) \end{pmatrix}
\end{equation}
and terms of the form
\begin{equation}
\label{Geq_R2}
(-1)^{\|a_1\|+\ldots+\|a_i\|+\delta(a_1)+\ldots+\delta(a_n)+1}\begin{pmatrix}
0 & 0 \\ \hat{\mk m}_n(a_n,\ldots,\bar{\mk m}_j(a_{i+j},\ldots,a_{i+1}),\ldots,a_1) & 0\end{pmatrix}
\end{equation}
But \eqref{Geq_L1}, \eqref{Geq_L2}, and \eqref{Geq_R1} cancel because $\mk m_n=\bar{\mk m}_n+\hat{\mk m}_nt$, and \eqref{Geq_L3} and \eqref{Geq_R2} cancel by looking at terms in positive powers of $t$ in the $A_\infty$-equation for the $\mk m_n$'s.
\end{proof}

The composite functor $FG$ send $X$ to
\begin{equation}
\label{FGobj}
F(G(X))=\left(X\oplus X[1-d,d],\begin{pmatrix} 0 & 0 \\ (-1)^d\bar{\hat{\mk m}}_0 & 0\end{pmatrix}\right)
\end{equation}
(where $\bar{\hat{\mk m}}_0$ is just the linear part of $\mk m_0$)
and we define the counit natural transformation $\varepsilon:FG\to 1_{\mc A_0}$ by projection to the first summand, $\varepsilon_X:X\oplus X[1-d,d]\to X$, and with vanishing higher order terms.
There is also a natural transformation $[1-d,d]\to FG$ by inclusion of the second summand.

\begin{prop}
\label{prop_adjunction}
$G$ is right adjoint to $F$, in the sense that the natural transformation $\varepsilon$ induces quasi-isomorphisms
\begin{equation}
\label{adj_qis}
\Hom_{\mathrm{Tw}_{\ZZ/2}(\mc A)}(X,GY)\longrightarrow\Hom_{\mathrm{Tw}_{\ZZ/2}(\mc A_0)}(FX,Y).
\end{equation}
Moreover, $\varepsilon$ fits into an exact triangle of functors
\begin{equation}
[1-d,d]\longrightarrow FG\longrightarrow 1_{\mc A_0}\longrightarrow [2-d,d].
\end{equation}
\end{prop}

\begin{proof}
The domain of~\eqref{adj_qis} is
\[
\Hom(X,Y\oplus Y[1-d,d])=\Hom(X,Y)\oplus \Hom(X,Y[1-d,d])
\]
and the map induced by $\varepsilon$ sends the first summand to the quotient of $\Hom(X,Y)$ by $\Hom(X,Y)t$ and vanishes on the second summand.
Thus, to show the first claim of the proposition, it suffices to show that $\Hom(X,Y)t\oplus \Hom(X,Y[1-d,d])$ is acyclic.
The differential, $D$, on this complex is
\[
D\begin{pmatrix} a \\ b \end{pmatrix} = \begin{pmatrix} \mk m_1(a)-(-1)^{\delta(b)}bt \\ \mk m_1(b)+(-1)^d\mk m_2(\hat{\mk m}_0,a) \end{pmatrix}
\]
for which 
\[
H\begin{pmatrix}a \\ b\end{pmatrix} := \begin{pmatrix} 0 \\ -(-1)^{d+\delta(a)}at^{-1} \end{pmatrix}
\]
is a contracting homotopy since
\[
(DH+HD)\begin{pmatrix}a \\ b\end{pmatrix}=\begin{pmatrix} a \\ -(-1)^{\delta(a)}\mk m_1(a)t^{-1}\end{pmatrix}+\begin{pmatrix} 0 \\ b+(-1)^{\delta(a)}\mk m_1(a)t^{-1} \end{pmatrix}
\]
thus $DH+HD=\mathrm{Id}$, so the complex in question is acyclic.

For the second claim we note that an exact sequence of functors and natural transformations is  exact (by definition) if it is exact on objects.
But this follows since~\eqref{FGobj} is an extension of the first and third term in the sequence.
\end{proof}

\begin{df}
Define $\mathrm{Tors}(\mc A):=\mathrm{Tw}^+(\mathrm{Im}(G))$, the triangulated closure of the full subcategory of $\mathrm{Tw}_{\ZZ/2}(\mc A)$ of objects in the image of $G$.
\end{df}

\section{3CY categories of surfaces}
\label{sec_3cysurf}

In this section we construct the 3CY categories $\mc C(S,\nu)$ starting from a compact surface $S$ with punctures $M\subset S$ and foliation $\nu$ on $S\setminus M$ which locally looks like the horizontal foliation of a quadratic differential with simple zeros and simple poles.
In particular, as a special case, we can choose $\nu$ to be the horizontal foliation $\mathrm{hor}(\varphi)$ of such a quadratic differential $\varphi$. 
The category $\mc C_{g,n}:=\mc C(S,\mathrm{hor}(\varphi))$, mentioned in the introduction, then depends up to equivalence only on the genus $g$ of $S$ and the number $n$ of poles of $\varphi$.

We start by reviewing the construction of the wrapped Fukaya category from~\cite{hkk} in Subsection~\ref{subsec_wrapped}.
This category admits an additional $\ZZ/2$-grading as explained in Subsection~\ref{subsec_Z2}.
For appropriate choice of $\nu$, the $\ZZ\times\ZZ/2$-graded version of the wrapped Fukaya category of $S\setminus M$ admits a deformation over $\mathbf k[[t]]$, where $|t|=-1$, $\pi(t)=\mathrm{odd}$, defined in Subsection~\ref{subsec_def}.
The category $\mc C(S,\nu)$ is defined as certain torsion modules over the $\mathbf k[[t]]$-linear category.
Up to this point, the construction depends on a choice of arc system $\XX$ for $(S,M)$. The result of Subsection~\ref{subsec_changearc} is that $\mc C(S,\nu)$ is independent, up to equivalence, of $\XX$.
Finally, in Subsection~\ref{subsec_3cy} we show that $\mc C(S,\nu)$ is a proper 3-Calabi--Yau category.

\subsection{Wrapped Fukaya categories of surfaces}
\label{subsec_wrapped}

We begin by reviewing the construction of the \textit{wrapped Fukaya category of a surface} as in~\cite[Section 3.3]{hkk}. 
In fact, the construction there is more general, allowing also \textit{partial wrapping}, but those categories will not be relevant here.

The categories are defined for a compact oriented surface $S$ with finite subset $M\subset S$ of \textit{punctures}, a foliation or \textit{grading structure} $\nu\in\Gamma(S\setminus M,\PP(TS))$, and choice of ground field $\mathbf k$. 
Here, $\nu$ is simply a line in each $T_pS$, $p\in S\setminus M$, varying smoothly with $p$.
The corresponding category is denoted $\mc F(S\setminus M,\nu)$ and is a triangulated $A_\infty$-category over $\mathbf k$.

For us, the main source of $\nu$ are \textit{horizontal foliations} of quadratic differentials: 
If $S$ is a Riemann surface and $\varphi$ a meromorphic quadratic differential on $S$ with set $M\subset S$ of zeros and poles, then the horizontal foliation $\mathrm{hor}(\varphi)$ with $\mathrm{hor}(\varphi)_p:=\{v\in T_pS\mid \varphi(v,v)\in\mathbb R_{\geq 0}\}$ is a grading structure on $S\setminus M$.

\subsubsection*{Step 1: Arc system}
The first step in constructing $\mc F(S\setminus M,\nu)$ is to choose an \textit{arc system} for $(S,M)$. 
\begin{df}
An \textbf{arc system} for a punctured surface $(S,M)$ is a collection, $\XX$, of embedded compact intervals in $S$ such that:
\begin{enumerate}
\item endpoints of arcs belong to $M$,
\item different arcs can intersect only in endpoints and do so transversely,
\item each point of $M$ belongs to some arc,
\item the arcs cut $S$ into a collection of polygons.
\end{enumerate}
\end{df}
The strategy is to define an $A_\infty$-category $\mc F_\XX$ with objects $\Ob(\mc F_\XX)=\XX$ and then show that $\mc F(S\setminus M,\nu):=\mathrm{Tw}^+(\mc F_\XX)$ is, up to quasi-equivalence, independent of $\XX$.
(To be precise, $\mc F_\XX$ depends also on a choice of grading on each arc, see below, and objects should be thought of as arcs together with the chosen grading.)

\subsubsection*{Step 2: Grading}
The next step is to choose a grading for each arc $X\in\XX$.
Let $\widetilde{\PP(TS)}$ be the fiberwise universal covering of $\PP(TS)|_{S\setminus M}$, where we take $\nu(p)$ as a basepoint in $\PP(T_pS)$. (In fact, we could have taken the $\RR$-bundle $\widetilde{\PP(TS)}$ with map to $\PP(TS)|_{S\setminus M}$ as a starting point, instead of $\nu$.)
The definition of grading for general curves is the following.

\begin{df}
Suppose $S$ is a smooth oriented surface with grading $\nu\in\Gamma(S,\PP(TS))$ and $c:I\to S$ an immersed curve, where $I$ is a 1-manifold. 
A \textbf{grading} of $c$ is a section, $\tilde{c}$, of $c^*\widetilde{\PP(TS)}$ which lifts the section of $c^*\PP(TS)$ given by the tangent spaces to $c$.
A \textbf{graded curve} is a curve together with a choice of grading.
\end{df}

The action of $\ZZ$ on each $\widetilde{\PP(T_pS)}$ by deck-transformations gives us an action of $\ZZ$ on the set of gradings of a curve by $\tilde{c}\mapsto\tilde{c}+n$.
If $I$ is an interval, then the set of gradings of a curve $c:I\to S$ is a $\ZZ$-torsor.
Also note that for an arc $X\in\XX$, the grading is really chosen over the interior of $X$ only, since $\nu$ is not defined at the endpoints.

\subsubsection*{Step 3: Real blow-up}
The \textit{real blow-up} of $S$ in $M$ is a compact surface $\widehat{S}$ with boundary, together with a map $\pi:\widehat{S}\to S$ which maps the interior of $\widehat{S}$ diffeomorphically to $S\setminus M$ and collapses each boundary circle (component of $\partial\widehat{S}$) to a point of $M$.
The system of arcs, $\XX$, lifts to a collection of disjoint embedded intervals in $\widehat{S}$ with endpoints on $\partial\widehat{S}$ which cut the surface into polygons with double the number of corners as the corresponding polygons in $S$.
After possibly perturbing $\nu$ near $M$, we can assume that $\nu$ extends to $\widehat{S}$. Each lifted arc is then a graded curve in $\widehat{S}$.

The purpose of passing to the real blow-up is to define \textit{boundary paths}, which are used in the definition of morphisms of $\mc F(S\setminus M,\nu)$.

\begin{df}
A \textbf{boundary path} is an immersed path $a:[0,1]\to \partial\widehat{S}$, up to reparametrization, which follows the boundary in the direction opposite to its natural orientation, i.e. so that the interior of $\widehat{S}$ is to the right of the path as one travels along it.
\end{df} 

By definition, $\Hom_{\mc F_\XX}(X,Y)$ is the free $\mathbf k$-vector space generated by boundary paths from an arc $X$ to an arc $Y$, and, if $X=Y$, the identity morphisms $1_X$.

\subsubsection*{Step 4: Degrees of boundary paths}
Let $a$ be a boundary path starting at an arc $X\in\XX$ and ending at an arc $Y\in\XX$.
As before, we assume that arcs are graded.
Then one can assign a \textit{degree} $|a|\in\ZZ$ to $a$.
This is based on the following more general definition.

\begin{df}
Let $L_0=(I_0,c_0,\tilde{c}_0)$ and $L_1=(I_1,c_1,\tilde{c}_1)$ be graded curves intersecting transversely at $p=c_0(t_0)=c_1(t_1)$. 
The \textbf{intersection index} of $L_0$ and $L_1$ at $p$ is
\begin{equation}
\begin{aligned}
i(L_0,t_0,L_1,t_1) & :=\lceil \tilde{c}_0(t_0)-\tilde{c}_1(t_1)\rceil \\
& :=\text{smallest }n\in\ZZ\text{ with }\tilde{c}_0(t_0)<\tilde{c}_1(t_1)+n
\end{aligned}
\end{equation}
where we use the total order on $\widetilde{\PP(T_pS)}$ coming from the orientation of $S$ (counterclockwise rotation is increasing) and the action of $\ZZ=\pi_1(\widetilde{\PP(T_pS)})$ by deck transformations as before. 
\end{df}

To define the degree of a boundary path, $a$, choose an arbitrary grading on $a$ and let
\[
|a|:=i(X,a(0),a,0)-i(Y,a(1),a,1)
\]
which is independent of the grading on $a$ and gives a $\ZZ$ grading on $\Hom_{\mc F_\XX}(X,Y)$.
Composition in $\mc F_\XX$ is defined on boundary paths by
\[
\mk m_2(a,b):=(-1)^{|b|}ab
\]
where $ab$ is the concatenation of boundary paths if $a$ starts where $b$ ends, or $0$ otherwise. 
If $\XX$ cuts out bigons, i.e. has one or more pairs of isotopic arcs, then $\mk m_2$ has additional terms decribed below.

\subsubsection*{Step 5: Immersed disks}

To complete the definition of the $A_\infty$-category $\mc F_\XX$, we need to define higher order terms $\mk m_{\geq 3}$. These come from certain immersed disks in $S$.

\begin{df}
\label{def_diskseq0}
Let $D$ be the standard $2n$-gon with edges 
\[
e_1, f_1, e_2, f_2,\ldots, e_n, f_n 
\]
in clockwise order. 
Suppose $\psi:D\to \widehat{S}$ is an immersion so that $X_i:=\psi|_{e_i}$ is an arc in $\XX$ and $a_i:=\psi|_{f_i}$ is a boundary path.
Then the sequence of boundary paths $a_1,\ldots,a_n$ is a \textbf{disk sequence}.
\end{df}

If $a_1,\ldots,a_n$ is a disk sequence, then for any boundary paths $b$ which ends where $a_1$ starts we define
\[
\mk m_n(a_n,\ldots,a_1b):=(-1)^{|b|}b
\]
and for any boundary path $c$ which starts where $a_n$ ends we define
\[
\mk m_n(ca_n,\ldots,a_1):=c
\]
and in all other cases the higher composition maps are zero.

\begin{prop}[\cite{hkk}]
$\mc F_\XX$ as defined above is an $A_\infty$-category over $\mathbf k$.
\end{prop}

\begin{ex}
Let $S=S^2$ be the 2-sphere with two punctures $M\subset S$ and arbitrary grading $\nu$.
For $\XX$ choose a single arc $X$ connecting the two points in $M$, cutting $S^2$ into a bigon.
The real blow-up $\widehat{S}$ is a cylinder, so $\partial\widehat{S}$ has two components corresponding to boundary paths $x$ and $y$ which go once around the boundary. 
Then $\Hom(X,X)$ has basis $1$, $x$, $y$, $x^2$, $y^2$, \ldots where the grading depends on $\nu$ and satisfies $|x|=-|y|$. 
The disk sequences are $x^k,y^k$ and $y^k,x^k$, $k\geq 1$, and we get $y=x^{-1}$, i.e. $\Hom(X,X)=\mathbf k\langle x,x^{-1}\rangle$ is just a graded algebra. 
The Fukaya category $\mc F(S\setminus M,\nu)$ is the category of twisted complexes over that algebra.
\end{ex}

\subsubsection*{Step 6: Independence of choice of arc system}

If two arc system $\XX$, $\XX'$ are isotopic, then there is an isomorphism of $A_\infty$-categories $\mc F_\XX\cong \mc F_{\XX'}$.
Less trivially, if $\XX'\subset \XX$ is a sub-arc system, then the inclusion functor $\mc F_{\XX'}\to \mc F_{\XX}$ induces a quasi-equivalence of categories of twisted complexes  $\mathrm{Tw}^+(\mc F_{\XX'})\to \mathrm{Tw}^+(\mc F_{\XX})$, as shown in~\cite{hkk}. 
The same holds if $\XX$ and $\XX'$ just differ in the gradings of the arcs.
This, together with the contractability of the classifying space of arc systems, implies that $\mathrm{Tw}^+(\mc F_{\XX})$ is essentially independent of the choice of $\XX$ and we denote it by $\mc F(S\setminus M,\nu)$.

\subsection{Additional $\ZZ/2$-grading}
\label{subsec_Z2}

Before deforming the wrapped Fukaya category $\mc F(S\setminus M,\nu)$ (see next subsection) it is necessary to upgrade the $\ZZ$-grading on $\Hom$'s to a $\ZZ\times\ZZ/2$-grading.

First, let $\Sigma\to S\setminus M$ be the double cover with fiber over a given point $p\in S\setminus M$ equal to the two-element set of orientations of the line $\nu(p)\subset T_pS$.
(If $\nu$ is the horizontal foliation of a quadratic differential $\varphi$, then $\Sigma$ is the double cover where $\sqrt{\varphi}$ is well-defined.)
For each arc $X\in\XX$ choose not just a lift to $\widetilde{\PP(TS)}$, but also a lift to $\Sigma$.

\begin{df}
The \textbf{parity} of a boundary path $a$ from an arc $X$ to an arc $Y$, denoted $\pi(a)$, is even (resp. odd) if the two lifts of $X$ and $Y$ lie on the same (resp. different) sheet of $\Sigma$ when transported along $a$.
\end{df}

Note that parity is additive with respect to concatenation of boundary paths, and the total parity of a disk sequence is even (since $\Sigma$ has no monodromy over a disk).
Thus $\mc F_\XX$ is a $\ZZ\times\ZZ/2$-graded $A_\infty$-category when equipped with this extra grading.
Passing to twisted complexes, $\mathrm{Tw}_{\ZZ/2}(\mc F_\XX)$ is a $\ZZ\times\ZZ/2$-graded $A_\infty$-category which is essentially independent of $\XX$.
The shift $X[m,n]$ of an arc corresponds to changing the lifts to $\widetilde{\PP(TS)}$ and $\Sigma$ accordingly.

A choice of lift of an arc $X$ to both $\widetilde{\PP(TS)}$ and $\Sigma$ determines an orientation of $X$, since the latter is a choice of orientation of the lines $\nu(p)\subset T_p(S)$ and the former gives a homotopy class of paths in $\PP(T_p(S))$ between $\nu(p)$ and the tangent line to $X$ at $p$.
In fact, we see from this that a choice of lift of $X$ to $\Sigma$ is equivalent to a choice of orientation of $X$. 
Moreover, the induced orientation on $X[m,n]$ is the same (resp. opposite) to the one on $X$ if $m+n$ is even (resp. odd).

Suppose that $a$ is a boundary path from an arc $X$ to an arc $Y$ with $X$ and $Y$ given their induced orientations.
Define $\varepsilon_+(a):=0$ (resp.  $\varepsilon_+(a):=1$) if $Y$ points away from (resp. towards) $a$ at the end of $a$. Similarly, define $\varepsilon_-(a):=0$ (resp.  $\varepsilon_-(a):=1$) if $X$ points away from (resp. towards) $a$ at the start of $a$, then 
\begin{equation}\label{parity_def}
\pi(a)+|a|=\varepsilon_+(a)+\varepsilon_-(a)\qquad \mod 2.
\end{equation}

\subsubsection*{$\ZZ/2$-action on the wrapped Fukaya category}
A-priori, the $A_\infty$-category $\mathrm{Tw}_{\ZZ/2}^+(\mc F_\XX)$ is, after forgetting the $\ZZ/2$-grading, embedded in the larger category $\mathrm{Tw}^+(\mc F_\XX)=\mc F(S\setminus M,\nu)$.
However, it turns out that, up to quasi-equivalence, these are the same category.
Another way of phrasing this is that there is a $\ZZ/2$-action on the wrapped Fukaya category (coming from the additional $\ZZ/2$-grading) which fixes all isomorphism classes of objects.

\begin{lemma}
\label{lem_inv_trivial}
The $\ZZ\times\ZZ/2$-graded category $\mathrm{Tw}_{\ZZ/2}^+(\mc F_{\XX})$ is triangulated when viewed as a $\ZZ$-graded category, i.e. $\mathrm{Tw}_{\ZZ/2}^+(\mc F_{\XX})$ and $\mathrm{Tw}^+(\mc F_{\XX})$ are quasi-equivalent.
Thus, the involutive autoequivalence on the wrapped Fukaya category $\mc F(S\setminus M,\nu)$ induced by change of parity acts trivially on isomorphism classes of objects, i.e. $X\cong X[0,1]$ for $X\in\mc F(S\setminus M,\nu)$.
\end{lemma}

\begin{figure}
\centering
\includegraphics[scale=1.3]{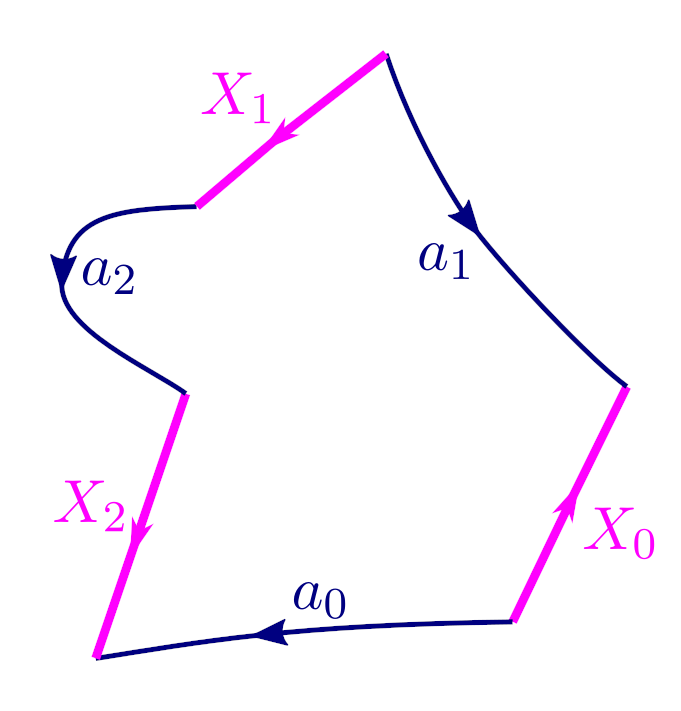}
\caption{Example of a cyclic sequence of arcs and boundary paths from which an indecomposable object is constructed as a twisted complex.}
\label{fig_indobj}
\end{figure}

\begin{proof}
Let $\mc C$ be the closure of $\mc F_\XX$ under shifts, i.e. the category which has for each arc $X\in\XX$ a $\ZZ\times\ZZ/2$-torsor of objects corresponding to the various lifts of $X$ to $\widetilde{\PP(TS)}$ and $\Sigma$.
To prove the statement of the lemma it suffices to show that any isomorphism class of objects in $\mathrm{Tw}^+(\mc C)$ can be represented by a twisted complex $(X,\delta)$ with $\delta\in\Hom^{1,0}(X,X)$, i.e. $\delta$ has even parity.

According to the classification result of \cite{hkk}, all indecomposable objects in $\mc F(S\setminus M,\nu)$ are obtained from the following construction. 
Let $X_0,\ldots,X_{n-1}$ be a sequence of objects in $\mc C$, either linearly ordered (the \textit{string} case) or cyclically ordered (the \textit{band} case), and $\delta_i\in\mathrm{Hom}^1(X_{i-1}^{\oplus r}, X_i^{\oplus r})$ or $\delta_i\in\mathrm{Hom}^1(X_i^{\oplus r},X_{i-1}^{\oplus r})$ of the form $A_i\otimes a_i$ where $A_i$ is an invertible $r\times r$-matrix ($r>0$ fixed in the band case, $r=1$ in the string case) and $a_i$ is a boundary path so that $a_i$ and $a_{i+1}$ meet the arc $X_i$ in opposite endpoints of $X_i$ (see Figure~\ref{fig_indobj}).
Under some additional conditions on the sequence, which will not be relevant here, the pair $(X,\delta)$ with $X=(X_1\oplus\cdots\oplus X_n)^{\oplus r}$ and $\delta$ with non-zero coefficients $\delta_i$ is a twisted complex.
After possibly changing the parity of $X_i$, the induced orientation on $X_i$ points away from $a_i$ and towards $a_{i+1}$, i.e. matching the orientation of the path which follows the linear or cyclic sequence $X_0,a_1,X_1,a_2,X_2,\ldots$.
According to \eqref{parity_def} we have 
\[
\pi(a_i)=|a_i|+\varepsilon_-(a_i)+\varepsilon_+(a_i)=1+1=0 \mod 2
\]
thus $\delta$ has even parity.
\end{proof}

\subsection{Deformation}
\label{subsec_def}

Recall from Subsection~\ref{subsec_wrapped} that the definition of the wrapped Fukaya category $\mc F(S\setminus M,\nu)$ is based on counting immersed polygons in the real blow-up $\widehat{S}$ of $S$ in $M$.
Instead, we could also count immersed polygons in $S$ itself, which would correspond to immersed polygons with holes (``swiss cheese slices'') in $\widehat{S}$, with a hole whenever a punctured is covered by the disk. Counting this with weight $t^k$, where $k$ is the number of holes, turns out to give a deformation of $\mc F(S\setminus M,\nu)$ over $\mathbf k[[t]]$, but the degree of $t$ depends on the structure of $\nu$ near the punctures.
To establish the relation we need the following definitions.

\begin{df}
Let $S$ be a surface with grading structure $\nu$ and $c:S^1\to S$ an immersed loop in $S$.
The \textbf{(Maslov) index}, $\mathrm{ind}(c)$, of $c$ is defined as the intersection index between the graphs of the following two sections of $c^*\PP(TS)\cong S^1\times S^1$: 1) the section $\nu\circ c$ and 2) the section given by the tangent lines to $c$.
The sign ambiguity is fixed by the following normalization: If $\nu$ is the horizontal foliation of the quadratic differential $z^kdz^2$ on $\CC$ and $c$ a counterclockwise circle around the origin, then $\mathrm{ind}(c)=k+2$ (see Figure~\ref{fig_foliation}). 
\end{df}

If $S$ is an oriented surface, $p\in S$, and a grading $\nu$ is defined away from $p$, then the \textbf{index} of $p$ is by definition the index of a small counterclockwise loop around $p$.

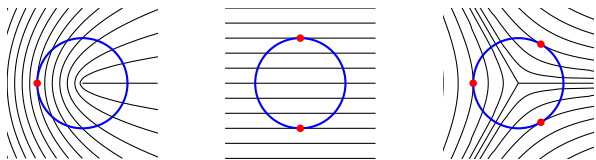
\begin{figure}
\centering
\begin{tikzpicture}[domain=-1:1]
\clip (-1,-1) rectangle (1,1);
\draw (1,0) to (0,0);
\draw[smooth,domain=-5:5,scale=.2,variable=\y] plot ({(\y*\y-.3*.3)/(2*.3)},{\y});
\foreach \i in {1,...,15}
{
	\draw[smooth,domain=-5:5,scale=.2,variable=\y] plot ({(\y*\y-\i*\i)/(2*\i)},{\y});
}
\draw[thick,blue] (0,0) circle[radius=.6];
\fill[thick,red] (-.6,0) circle[radius=.05];
\end{tikzpicture}
\qquad
\begin{tikzpicture}[domain=-1:1]
\clip (-1,-1) rectangle (1,1);
\foreach \i in {-5,...,5}
{
	\draw[smooth,domain=-5:5,scale=.2,variable=\x] plot ({\x},{\i});
}
\draw[thick,blue] (0,0) circle[radius=.6];
\fill[thick,red] (0,.6) circle[radius=.05];
\fill[thick,red] (0,-.6) circle[radius=.05];
\end{tikzpicture}
\qquad
\begin{tikzpicture}[domain=-1:1]
\clip (-1,-1) rectangle (1,1);
\foreach \i in {1,...,6}
{
	\draw[smooth,domain=1:119,scale=.2,variable=\t] plot ({\i*(1/exp(ln(sin(1.5*\t))*2/3))*cos(\t)},{\i*(1/exp(ln(sin(1.5*\t))*2/3))*sin(\t)});
}
\foreach \i in {1,...,6}
{
	\draw[smooth,domain=1:119,scale=.2,variable=\t] plot ({\i*(1/exp(ln(sin(1.5*\t))*2/3))*cos(\t)},{-\i*(1/exp(ln(sin(1.5*\t))*2/3))*sin(\t)});
}
\foreach \i in {1,...,6}
{
	\draw[smooth,domain=1:119,scale=.2,variable=\t] plot ({\i*(1/exp(ln(sin(1.5*\t))*2/3))*cos(\t+120)},{\i*(1/exp(ln(sin(1.5*\t))*2/3))*sin(\t+120)});
}
\draw (0,0) to (1,0);
\draw (0,0) to (-1,1.73);
\draw (0,0) to (-1,-1.73);
\draw[thick,blue] (0,0) circle[radius=.6];
\fill[thick,red] (-.6,0) circle[radius=.05];
\fill[thick,red] (.3,.52) circle[radius=.05];
\fill[thick,red] (.3,-.52) circle[radius=.05];
\end{tikzpicture}
\caption{Horizontal foliations of quadratic differentials with a simple pole (left), regular point (middle), and simple zero (right). The index of circle (oriented counterclockwise) is 1, 2, and 3, respectively.}
\label{fig_foliation}
\end{figure}

\begin{df}
Suppose $S$ is an oriented surface, $M\subset S$, and $\nu$ is a grading structure on $S\setminus M$.
For given $d\in\ZZ$ we say that $\nu$ is \textbf{$d$-compatible} if for each $p\in M$, some positive integer multiple of $\mathrm{ind}(p)$ is $d$.
\end{df}

The corresponding degree of $t$ is then $|t|=2-d$.
Here, we focus on the case $d=3$ (where we get 3CY categories), so $|t|=-1$, and a meromorphic quadratic differential has 3-compatible horizontal foliation in the above sense iff all its zeros and poles are simple.
Even if $\nu$ is not the horizontal foliation of a quadratic differential, we will suggestively refer to $p\in M$ with $\mathrm{ind}(p)=1$ as \textit{poles}, and $p\in M$ with $\mathrm{ind}(p)=3$ as \textit{zeros}.

\vspace{10pt}

\noindent\fbox{\begin{minipage}{\dimexpr\textwidth-2\fboxsep-2\fboxrule\relax}\textbf{Assumption:} For the rest of this section, $S$ is a compact oriented surface, $M\subset S$ a non-empty finite subset, and $\nu$ a grading structure on $S\setminus M$ which is 3-compatible.\end{minipage}}

\vspace{10pt}

Under the above assumption we will define a curved $A_\infty$-category $\mc A_\XX$ over $\mathbf k[[t]]$, where $t$ is a formal variable of bi-degree $(-1,1)$, with central fiber $\mc A_\XX\otimes_{\mathbf k[[t]]}\mathbf k=\mc F_\XX$.
Thus set
\[
\mathrm{Ob}(\mc A_\XX):=\mathrm{Ob}(\mc F_\XX)=\XX,\qquad \Hom_{\mc A_\XX}(X,Y):=\Hom_{\mc F_\XX}(X,Y)[[t]]
\]
and the $t^0$ terms of the structure maps in $\mc A_\XX$, $\mk m_{n,0}$, coincide with the structure maps $\mk m_n$ of $\mc F_\XX$.
It remains to define the $t^{>0}$ terms of the structure maps, $\mk m_{n,k}$, $k\geq 1$.
There are two types:
\begin{enumerate}
\item
$\mk m_{0,1}$ coming from punctures
\item 
$\mk m_{n,k}$, $n\geq 2$, coming from immersed disks in $S$
\end{enumerate}

\begin{remark}
The terms of the first type are seemingly more ad-hoc than the terms of the second type, which just generalize the terms in the structure maps for $\mc F_\XX$.
Without them, however, the $A_\infty$-equations would not hold.
\end{remark}

Terms of the first type are defined as follows.
If $X\in\XX$ is an arc ending in a point $p\in M$, define $c_{p,X}$ to be the path which starts and ends at $X$ and goes $3/\mathrm{ind}(p)$ times around the boundary component of $\widehat{S}$ corresponding to $p$.
The paths $c_{p,X}$ are also characterized by the properties of starting and ending at the same point and having bi-degree $(3,1)$.
The curvature terms are then defined as
\begin{equation}\label{m0_def}
\asm_{0,1}(X):=c_{q,X}-c_{p,X}
\end{equation}
where $X\in\XX$ starts at $p\in M$ and ends at $q\in M$ with respect to the orientation of $X$ coming from the grading (see Subsection~\ref{subsec_Z2}).

To define the terms of the second type, we count immersed punctured disks.

\begin{df}
\label{def_immersed}
Let $D$ be the standard $2n$-gon with edges 
\[
e_1,f_1,e_2,f_2,\ldots,e_n,f_n 
\]
in clockwise order and $k\geq 0$ open disks with boundary circles $c_1,\ldots,c_k$ removed from its interior. 
An \textbf{immersed polygon with $k$ holes} is given by an immersion $\psi:D\to \widehat{S}$ so that 
\begin{itemize}
\item 
$X_i:=\psi|_{e_i}$ is an arc in $\XX$,
\item
$a_i:=\psi|_{f_i}$ is a boundary path,
\item
$\psi(c_i)$ is a component of $\partial\widehat{S}$ corresponding to some $p\in M$ which is covered $3/\mathrm{ind}(p)$-to-1 by $c_i$,
\end{itemize}
up to reparametrization.
The \textbf{associated cyclic sequence} of boundary paths is $a_n,a_{n-1},\ldots,a_1$.
\end{df}

Immersed polygons in the above sense can equivalently be defined in a purely combinatorial way.
To see this, note that the decomposition of $\widehat{S}$ into polygons (by cutting along arcs in $\XX$) lifts along $\psi$ to a decomposition of $D$ into polygons.
One can reconstruct $D$ and $\psi$, up to reparametrization, from the following finite data:
\begin{enumerate}
\item
For each polygon $P$ in $\widehat{S}$, the finite set $E_P$ of preimages in $D$, i.e. the polygons $\widetilde{P}$ in $D$ which map to $P$.
\item
For each arc $X\in\XX$ bounding polygons $P,Q$ a bijection $\sigma_X$ from a subset of $E_P$ to a subset of $E_Q$, where $\sigma_X(\widetilde{P})=\widetilde{Q}$ records the fact that $\widetilde{P}$ is glued to $\widetilde{Q}$ in $D$ along a lift of $X$.
\end{enumerate}

We observe the following relation between the number of holes of the polygon and the total degree of the associated sequence of boundary paths.

\begin{lemma}
\label{lem_holes_degree}
Let $a_n,\ldots,a_1$ be the cyclic sequence associated with an immersed polygon with $k$ holes, then
\[
|a_1|+\ldots+|a_n|=n-k-2,\qquad\pi(a_1)+\ldots+\pi(a_n)=k\mod 2.
\]
\end{lemma}

\begin{proof}
Let $\psi:D\to \widehat{S}$ be the immersed polygon with $k$ holes.
Consider the foliation $\psi^*\nu$ on $D$.
The Maslov index of a counterclockwise loop around each of the holes is $3$, by our conditions on $\psi$.
It follows that if $\gamma$ is a counterclockwise loop in $D$ which goes once around \textit{all} the holes, then $\mathrm{ind}(\gamma)=k+2$. 
This in turn implies, by the same argument as in the case without holes, that $\|a_1\|+\ldots+\|a_n\|=-k-2$, which is the first claim.
The relation $\pi(a_1)+\ldots+\pi(a_n)=k$ follows since $\psi^*\Sigma$ has odd monodromy around each hole, thus monodromy $k\mod 2$ along $\gamma$.
\end{proof}

As a consequence, we can deduce the finiteness of the number of immersed polygons which produce a given cyclic sequence of boundary paths.

\begin{lemma}
\label{lem_immersion_finite}
Let $a_i:X_i\to X_{i+1}$, $i\in\ZZ/n$ be a cyclic sequence of boundary paths.
Then there are finitely many immersed polygons with holes which have associated cyclic sequence $a_n,a_{n-1},\ldots,a_1$.
\end{lemma}

\begin{proof}
We already know from Lemma~\ref{lem_holes_degree} that the number, $k$, of holes is fixed by the sequence.
Call a boundary path \textit{irreducible} if it starts and ends at an arc in $\XX$, but does not meet any other arcs along the way. Equivalently, it is the edge of one of the polygons cut out by $\XX$.
For a general boundary path, the \textit{length} is the number of irreducible factors under concatenation.
In particular, the sequence $a_n,\ldots,a_1$ has some total length, $L$.
If $\psi:D\to\widehat{S}$ is an immersed polygon with holes which has the given cyclic sequence of boundary paths, then $\partial D$ has some total length, $M$, in the above sense, where we ignore parts of $\partial D$ not mapped to $\partial \widehat{S}$. 
Looking at various components of the boundary we find $M\leq L+3Ck$, where $C$ is the maximal length of a boundary circle of $\widehat{S}$, thus $3C$ the maximal length of the boundary of any of the holes in $D$.
Thus, the number of polygons in $D$ has the same a-priori bound $L+3Ck$, since an $n$-gon in $D$ increases the length of the boundary by $n$.
Finally, the combinatorial description of immersed polygons in terms of finite sets and bijection (see the paragraph after Definition~\ref{def_immersed}) shows that there are only finitely many immersed disks with holes for a given upper bound on the total number of polygons.
\end{proof}

Given the previous lemma, we can define $N(a_n,\ldots,a_1)$ to be the number of immersed polygons with holes that have associated cyclic sequence $a_n,a_{n-1},\ldots,a_1$.
In the case without holes ($k=0$) we have $N(a_n,\ldots,a_1)\in\{0,1\}$ and $N(a_n,\ldots,a_1)=1$ if and only if $a_n,\ldots,a_1$ is a disk sequence in the sense of Definition~\ref{def_diskseq0}.
To see this, note that immersed polygons without holes lift to embedded polygons in the universal cover of $\widehat{S}$. 
These are then uniquely determined by their boundary.

\begin{remark}
In fact, $N(a_n,\ldots,a_1)\in\{0,1\}$ in general, though we will not need this here.
To prove this, one shows that the polygon with holes $D$ together with the immersion $\psi:D\to\widehat{S}$ can be reconstructed from the cyclic sequence $a_n,\ldots,a_1$ of boundary paths by induction on the number of polygons in $D$. 
\end{remark}

We are now ready to define $\mk m_{n,k}$ as follows.
Whenever $\psi:D\to\widehat{S}$ is an immersed $2n$-gon with $k$ holes and associated cyclic sequence of boundary paths $a_n,\ldots,a_1$, then this contributes the following terms to $\mk m_{n,k}$:
\begin{itemize}
\item 
To $\mk m_{n,k}(a_n,\ldots,a_1)$ add a term $(-1)^{k\varepsilon_-(a_1)}1_X$, where $X$ is the arc where $a_1$ starts.
\item
For any boundary path $b$ starting where $a_n$ ends: To $\mk m_{n,k}(ba_n,\ldots,a_1)$ add a term $(-1)^{k\varepsilon_-(a_1)}b$.
\item
For any boundary path $b$ ending where $a_1$ starts: To $\mk m_{n,k}(a_n,\ldots,a_1b)$ add a term $(-1)^{|b|+k\varepsilon_-(b)}b$.
\end{itemize}
(The sign $\varepsilon_-(a)\in\ZZ/2$, where $a$ is a boundary path, was defined just before \eqref{parity_def}.)

\begin{prop}
$\mc A_\XX$ as defined above is an $A_\infty$-category over $\mathbf k[[t]]$, $|t|=-1$, $\pi(t)=1$.
\end{prop}

\begin{proof}
We will show that the terms which appear in \eqref{a_infty_eta} cancel pairwise.
There are three types of non-zero terms in $\mk m_{n,k}$:
\begin{enumerate}
\item terms in $\mk m_{0,1}$ from punctures,
\item terms in $\mk m_{2,0}$ from composition of boundary paths,
\item terms in $\mk m_{n,k}$ from immersed $2n$-gons with $k$ holes.
\end{enumerate}
Since $\mk m_2$ receives contributions both from terms of the second and third type, we write $\mk m_2=\mk m_2'+\mk m_2''$ where the first (resp. second) summand collects terms of the second (resp. third type), to avoid any ambiguity. We also write $\mk m_n=\mk m_n''$ for $n\geq 3$ for notational purposes.
The possible terms in the quadratic equation \eqref{a_infty_eta} are then
\begin{itemize}
\item
1.L: terms from $\mk m_2'(\mk m_0,a)$
\item
1.R: terms from $\mk m_2'(a,\mk m_0)$
\item
2.L: terms from $\mk m_2'(\mk m_2'(a,b),c)$
\item
2.R: terms from $\mk m_2'(a,\mk m_2'(b,c))$
\item
3.L: terms from $\mk m_2'(\mk m_n''(\ldots),\_)$
\item
3.R: terms from $\mk m_2'(\_,\mk m_n''(\ldots))$
\item
4.IN/4.OUT: terms from $\mk m_n''(\ldots,\mk m_0,\ldots)$; subcase 4.IN when the cyclic boundary path $c_{p,X}$ coming from $\mk m_0$ is part of the cyclic sequence contributing a term to $\mk m_n''$ (see Figure~\ref{fig_cancel4}), subcase 4.OUT otherwise
\item
5.IN.INT/5.IN.EXT/5.OUT.L/5.OUT.R: terms from the composition $\mk m_n''(\ldots,\mk m_2'(\_,\_),\ldots)$; subcase 5.IN when the concatenation of boundary paths is inside the cyclic sequence contributing a term to $\mk m_n''$, subcase 5.OUT otherwise; 5.IN.INT and 5.IN.EXT are distinguished as follows: The two boundary paths are concatenated by $\mk m_2'$ at one of the endpoints of some arc $X\in\XX$. 
Lift this arc to the immersed polygon $D$, then the other endpoint is either on the boundary of one of the $k$ holes of $D$ (sub-subcase 5.IN.INT, see Figure~\ref{fig_cancel4}) or on the \textit{exterior boundary} of $D$ (sub-subcase 5.IN.EXT, see Figure~\ref{fig_cancel3}).
Subcase 5.OUT.L is $\mk m_n''(\mk m_2'(\_,\_),\ldots)$ while 5.OUT.R is $\mk m_n''(\ldots,\mk m_2'(\_,\_))$.
\item
6.IN/6.OUT: terms from $\mk m_n''(\ldots,\mk m_r''(\ldots),\ldots)$; subcase 6.IN when the term coming from the inner $\mk m_r''$ is part of the cyclic sequence contributing a term to $\mk m_n''$ (see Figure~\ref{fig_cancel3}), subcase 6.OUT otherwise
\end{itemize}

We claim that the following types of terms cancel with terms of a different type:
\begin{gather*}
\mathrm{1.L\leftrightarrow 1.R,\qquad 2.L\leftrightarrow 2.R,\qquad 3.L\leftrightarrow 5.OUT.R},\\
 \mathrm{3.R\leftrightarrow 5.OUT.L,\qquad 4.IN\leftrightarrow 5.IN.INT,\qquad 5.IN.EXT\leftrightarrow 6.IN}
\end{gather*}
and the following types of terms cancel with terms of the same type:
\[
\mathrm{4.OUT, \qquad 6.OUT}
\]
The remainder of the proof consists of carefully checking each of these cancellations.

\paragraph*{\textbf{1.L cancels 1.R:}}
As a special case of \eqref{a_infty_eta} we need to check that
\[
\mk m_{2,0}(a,\mk m_{0,1})+(-1)^{\pi(a)}\mk m_{2,0}(\mk m_{0,1},a)=0.
\]
If $a$ is a boundary path from an arc $X$ to an arc $Y$ then the 1.R term in the first summand above is of the form $\pm ac_{p,X}$ while the 1.L term in the second summand above is of the form $\pm c_{p,Y}a$.
These are the same boundary path. 
Moreover, the first term comes with a sign $\varepsilon_-(a)+1+1$ (where $\varepsilon_-(a)+1$ comes from $\mk m_0$ and $1$ comes from $\mk m_2$), while the second comes with a sign $\varepsilon_+(a)+1+|a|+\pi(a)$ (where $\varepsilon_+(a)+1$ comes from $\mk m_0$, $|a|$ comes from $\mk m_2$, and $\pi(a)$ comes from the above equation) and these are opposite by \eqref{parity_def}, so the terms cancel.

\paragraph*{\textbf{2.L cancels 2.R:}}
Up to signs, this is just associativity of concatenation of boundary paths. Signs cancel since we know this already for the wrapped Fukaya category.

\paragraph*{\textbf{3.L cancels 5.OUT.R:}}
Here we have an immersed $2n$-gon with $k$ holes $\psi:D\to\widehat{S}$ with sequence of boundary paths $a_n,\ldots,a_1$, and boundary paths $b$ and $c$ so that the concatenation $a_1bc$ is another boundary path.
This contributes a term $\pm bc$ to
\[
\mk m_{2,0}(\mk m_{n,k}(a_n,\ldots,a_1b),c)\qquad \text{and}\qquad \mk m_{n,k}(a_n,\ldots,a_2,\mk m_{2,0}(a_1b,c)).
\]
The first term comes with a sign
\[
|b|+k\varepsilon_-(b)+|c|+(k+1)(|c|+1)+k(\pi(c)+1)
\]
which is opposite to the sign 
\[
|c|+|b|+|c|+k\varepsilon_-(c)
\]
of the second term.

\paragraph*{\textbf{3.R cancels 5.OUT.L:}}
This is similar to the previous case except that now $cba_n$ is a boundary path instead of $a_1bc$.
We get contributions $\pm cb$ to 
\[
\mk m_{n,k}(\mk m_{2,0}(c,ba_n),\ldots,a_1)\qquad \text{and}\qquad \mk m_{2,0}(c,\mk m_{n,k}(ba_n,\ldots,a_1)).
\]
The first term comes with a sign
\[
|b|+|a_n|+k\varepsilon_-(a_1)+|a_1|+\ldots+|a_{n-1}|+n-1
\]
which is opposite to the sign
\[
|b|+k+k\varepsilon_-(a_1)
\]
of the second term.

\paragraph*{\textbf{4.IN cancels 5.IN.INT:}}
The main observation is that if $a_n,\ldots,a_{i+1},c,a_i,\ldots,a_1$ is the cyclic sequence of boundary paths associated with an immersed polygon with holes $\psi:D\to \widehat{S}$, and $c=c_{p,X}$ is a boundary path winding once around a component of $\partial\widehat{S}$ (thus giving a term of type 4.IN), then we get another $\psi':D'\to\widehat{S}$ where two edges of $D$ mapping to $X$ have been glued to produce $D'$ with one more hole than $D$ (see Figure~\ref{fig_cancel4}). Then $\psi'$ has associated sequence of boundary paths $a_n,\ldots,a_{i+1}a_i,\ldots,a_1$ where $a_{i+1}$ and $a_i$ have been concatenated (thus giving a term of type 5.IN.INT).

\begin{figure}
\centering
\includegraphics[scale=.3]{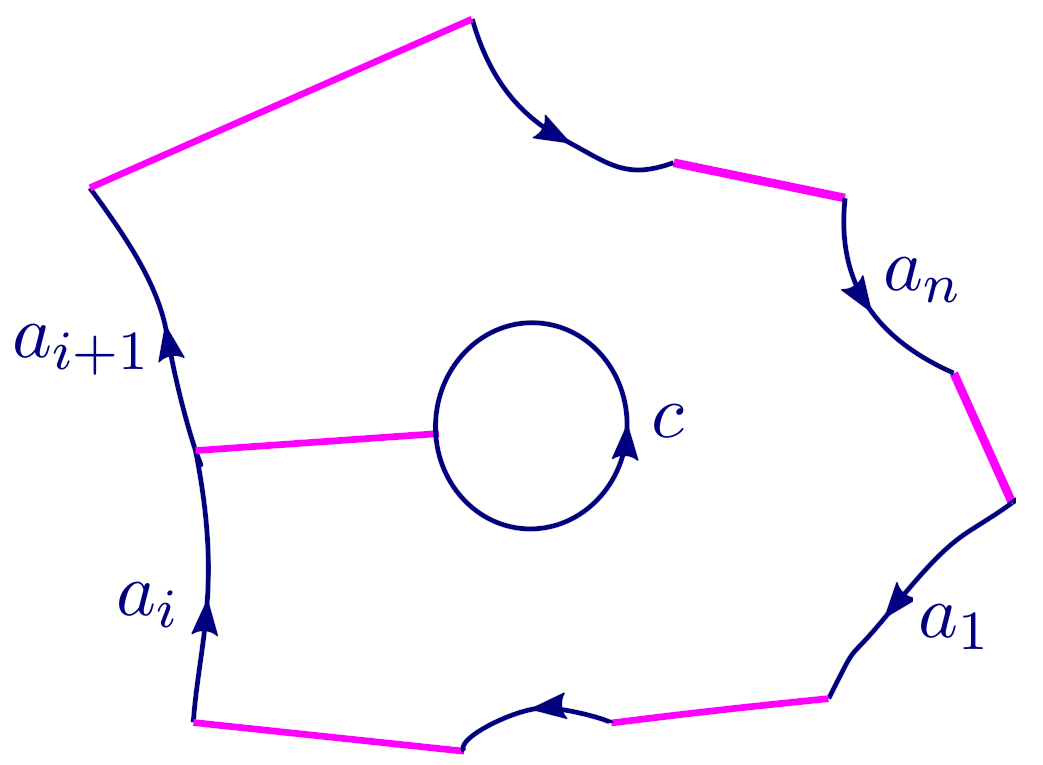}
\caption{Contribution to terms of type 4.IN and 5.IN.INT.}
\label{fig_cancel4}
\end{figure}

The 4.IN term contributes $\pm b$ to
\[
\mk m_{n+1,k-1}(a_n,\ldots,a_{i+1},\mk m_{0,1},a_i,\ldots,a_1b)
\]
and 5.IN.INT term also contributes $\pm b$ to
\[
\mk m_{n-1,k}(a_n,\ldots,a_{i+2},\mk m_{2,0}(a_{i+1},a_i),a_{i-1},\ldots,a_1b)
\] 
and similarly with $b$ on the left end instead of the right end.
The first term comes with a sign
\[
\underbrace{\pi(a_1)+\ldots+\pi(a_i)}_{\mathclap{|a_1|+\ldots+|a_i|+\varepsilon_-(a_1)+\varepsilon_+(a_i)+i-1}}+1+\varepsilon_+(a_i)+(k-1)\varepsilon_-(a_1)
\]
which is opposite to the sign 
\[
|a_i|+\ldots+|a_i|+i-1+|a_i|+k\varepsilon_-(a_1)
\]
of the second term, so the two cancel.

\paragraph*{\textbf{5.IN.EXT cancels 6.IN:}}
The main observation is that if $\psi:D\to\widehat{S}$ is an immersed $2n$-gon with $r$ holes and sequence of boundary paths $a_n,\ldots,a_1$, $\psi':D'\to\widehat{S}$ an immersed $2k$-gon with $s$ holes and sequence of boundary paths $b_k,\ldots,b_1$ and the arc $X$ where $a_{i-1}$ ends and $a_i$ starts is the same as the arc where $b_k$ ends and $b_1$ starts, but with opposite induced orientation from $\partial D$ and $\partial D'$ (see Figure~\ref{fig_cancel3}), then we can glue $D$ and $D'$ along this arc to obtain $\psi'':D''\to\widehat{S}$ which has $r+s$ holes and associated sequence of boundary paths 
\[
a_n,\ldots,a_{i+1},a_ib_k,b_{k-1},\ldots,b_2,b_1a_{i-1},a_{i-2},\ldots,a_1.
\]

\begin{figure}
\centering
\includegraphics[scale=.3]{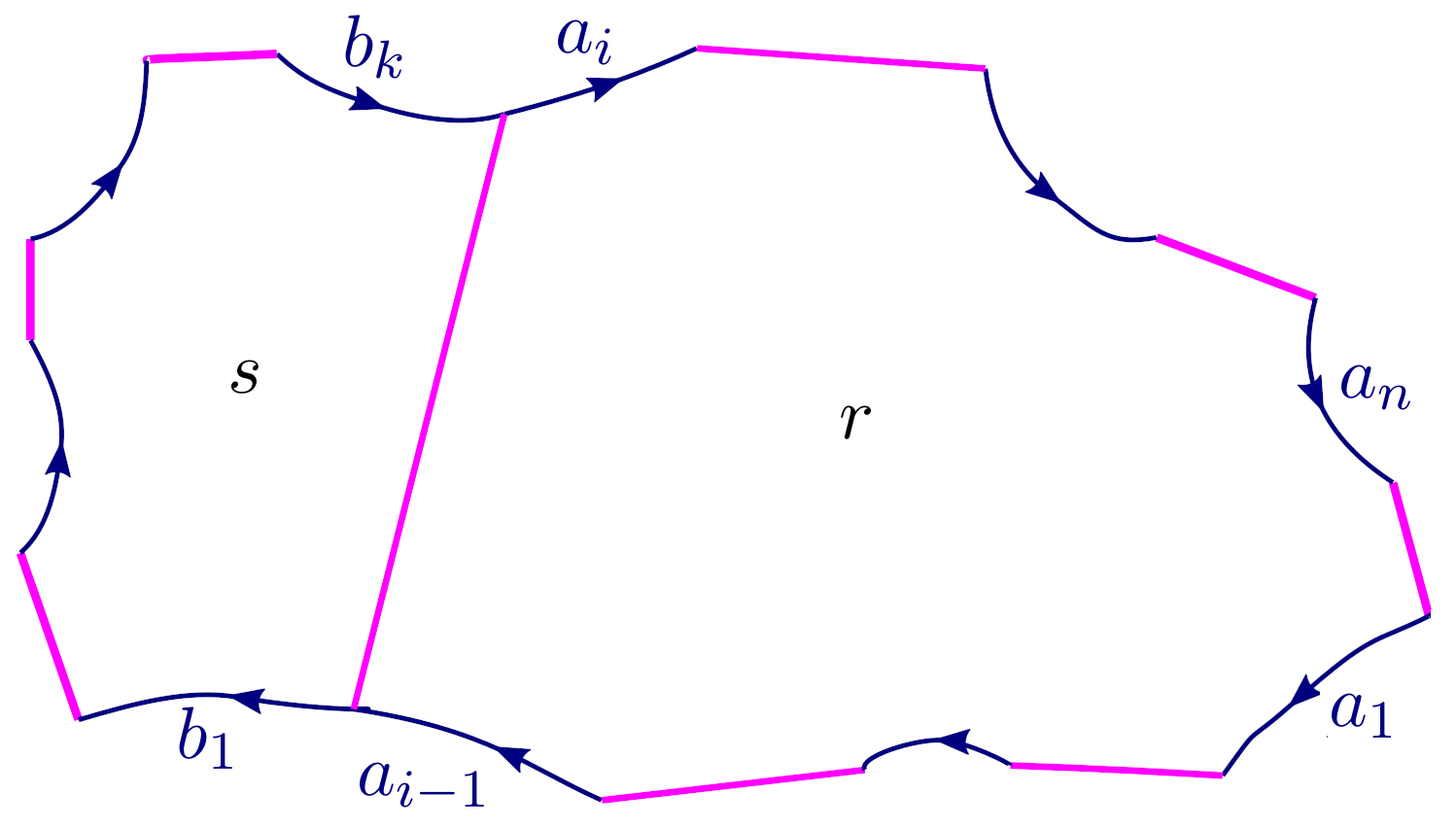}
\caption{Contribution to terms of type 5.IN.EXT and 6.IN.}
\label{fig_cancel3}
\end{figure}

The immersed polygons $D$ and $D'$ contribute a term $\pm c$ of type 6.IN to 
\[
\mk m_{n,r}(a_n,\ldots,a_{i+1},\mk m_{k,s}(a_ib_k,\ldots,b_1),a_{i-1},\ldots,a_1c)
\]
while $D''$ contributes a term $\pm c$ of type 5.IN.EXT to
\[
\mk m_{n+k-2,r+s}(a_n,\ldots,a_ib_k,\ldots,\mk m_{2,0}(b_1,a_{i-1}),\ldots,a_1c)
\]
and similarly with $c$ at the other end.
The first comes with a sign
\begin{align*}
&(s+1)(|c|+|a_1|+\ldots+|a_{i-1}|+i-1)+s(\underbrace{\pi(c)+\pi(a_1)+\ldots+\pi(a_{i-1})+1}_{\mathclap{|c|+|a_1|+\ldots+|a_{i-1}|+i-1+\varepsilon_-(c)+\varepsilon_+(a_{i-1})}}) \\
&+|c|+s\varepsilon_-(b_1)+r\varepsilon_-(c)
\end{align*}
which is opposite to the sign
\[
|c|+|a_1|+\ldots+|a_{i-2}|+i-2+|a_{i-1}|+|c|+(r+s)\varepsilon_-(c)
\]
of the second term.

\paragraph*{\textbf{4.OUT cancels itself:}}
Here we have an immersed $2n$-gon with $k$ holes $\psi:D\to\widehat{S}$ giving a cyclic sequence $a_n,\ldots,a_1$ and the boundary path $b$ complementary to $a_1$ in sense that the concatenated paths are $ba_1=c_{p,X}$ and $a_1b=c_{p,Y}$ for some $p\in M$ and arcs $X,Y$.
This gives a contribution of $\pm b$ to 
\[
\mk m_{n,k}(\mk m_{0,1},a_n,\ldots,a_2)\qquad\text{and}\qquad \mk m_{n,k}(a_n,\ldots,a_2,\mk m_{0,1}).
\]
The first term comes with a sign
\[
\varepsilon_-(a_1)+1+k\varepsilon_-(a_2)+\underbrace{\pi(a_2)+\ldots+\pi(a_n)}_{\pi(a_1)+k}+1
\]
which is opposite to the sign
\[
\varepsilon_+(a_1)+1+|b|+k\varepsilon_-(b)+1
\]
of the second term, so the two cancel.

\paragraph*{\textbf{6.OUT cancels itself:}}
Here we have an immersed $2n$-gon with $r$ holes $\psi:D\to\widehat{S}$ with sequence of boundary paths $a_n,\ldots,a_1$, and an immersed $2k$-gon with $s$ holes $\psi':D'\to\widehat{S}$ with sequence of boundary paths $c_k,\ldots,c_1$ and a boundary path $b$ which starts where $c_k$ ends and ends where $a_1$ starts, i.e. so that the  concatenation $a_1bc_k$ is a boundary path.
This contributes a term $\pm b$ to 
\begin{gather*}
\mk m_{n,r}(a_n,\ldots,a_2,\mk m_{k,s}(a_1bc_k,c_{k-1},\ldots,c_1))\\
\text{and}\qquad \asm_{k,s}(\mk m_{n,r}(a_n,\ldots,a_2,a_1bc_k),c_{k-1},\ldots,c_1)
\end{gather*}
where the first comes with a sign
\[
s+s\varepsilon_-(c_1)+|b|+r\varepsilon_-(b)
\]
which is opposite to the sign 
\begin{gather*}
r(\underbrace{\|c_1\|+\ldots+\|c_{k-1}\|+\pi(c_1)+\ldots+\pi(c_{k-1})+1}_{\varepsilon_-(c_1)+\varepsilon_+ (c_{k-1})})\\ 
+\underbrace{|c_1|+\ldots+|c_{k-1}|+k-1}_{|c_k|+s+1} +r\varepsilon_-(c_k)+s\varepsilon_-(c_1)+|b|+|c_k|
\end{gather*}
of the second term.
\end{proof}

\subsection{Change of arc system}
\label{subsec_changearc}

In the previous subsection we defined an $A_\infty$-category, $\mc A_\XX$, over $\mathbf k[[t]]$ depending on a surface $S$ with punctures $M\subset S$, grading structure $\nu$, and arc system $\XX$.
This category is curved and we get an uncurved $A_\infty$-category by passing to torsion modules (i.e. objects in the triangulated completion of the image of the functor from $\mc F_\XX$, see Subsection~\ref{subsubsec_basechange}).
The purpose of this subsection is to show that $\mathrm{Tors}(\mc A_\XX)$ is independent of the choice of arc system $\XX$, as in the case of the wrapped Fukaya category.

\subsubsection*{Change of grading}
First we want to relate $\mc A_\XX$ and $\mc A_{\XX'}$ if the arc system $\XX$ and $\XX'$ differ only in the choice of grading.
For this it is convenient to consider a larger category, containing both as full subcategories.
Namely, consider the category $\mc A_{\XX,\mathrm{shifts}}$ which has an object for each arc in $\XX$ together with a choice of grading, i.e. a lift to both $\widetilde{\PP(TS)}$ and $\Sigma$.
Thus, $\mc A_{\XX,\mathrm{shifts}}$ has a $\ZZ\times\ZZ/2$-torsor of objects for every arc $X\in\XX$.
The definition of morphisms and structure maps works as in the case of $\mc A_\XX$, the only difference being that instead of formally adding identity maps of objects one adds (to the basis of morphism) maps 
\[
\sigma_{X,m,n}:X\to X[m,n]
\]
of bi-degree $(m,n)$, generalizing $1_X=\sigma_{X,0,0}$, and satisfying
\begin{align*}
\mk m_{2,0}\left(\sigma_{Y,m,n},a\right)&:=(-1)^{|a|+m\varepsilon_+(a)}a \\
\mk m_{2,0}\left(a,\sigma_{X,m,n}\right)&:=(-1)^{m(\varepsilon_-(a)+1)}a \\
\mk m_{2,0}\left(\sigma_{X,m,n},\sigma_{X,r,s}\right)&:=(-1)^r\sigma_{X,m+r,n+s}
\end{align*}
where $a$ is a boundary path from $X$ to $Y$.
A straightforward checking of signs shows that this gives a curved $A_\infty$-category over $\mathbf k[[t]]$.
The category $\mc A_{\XX,\mathrm{shifts}}$ is then, by construction, closed under shifts of cohomological degree and parity in the sense that the abstract shift of any object is representable in that category.
This also shows that $\mc A_{\XX,\mathrm{shifts}}$ is isomorphic to the formal completion of $\mc A_\XX$ under shifts in bi-degree.

\begin{prop}
The inclusion functor $\mc A_\XX\to\mc A_{\XX,\mathrm{shifts}}$ induces a quasi-equivalence
$\mathrm{Tors}(\mc A_\XX)\to\mathrm{Tors}(\mc A_{\XX,\mathrm{shifts}})$ of $A_\infty$-categories over $\mathbf k$.
\end{prop}

\begin{proof}
This follows from the equivalence of $\mathrm{Add}(\mc A_\XX)$ and $\mathrm{Add}(\mc A_{\XX,\mathrm{shifts}})$.
\end{proof}

\subsubsection*{Adding a single arc}

\begin{prop}
Suppose an arc system $\XX'$ is obtained from an arc system $\XX$ by adding a single arc $X$ (bisecting some polygon).
Then the inclusion functor $\mc A_\XX\to\mc A_{\XX'}$ induces a quasi-equivalence of categories of torsion objects.
\end{prop}

\begin{proof}
We know from~\cite{hkk} that $X$, as an object of $\mc F_{\XX'}$, is isomorphic to an object in $\mathrm{Tw}^+(\mc F_\XX)$.
Hence, $G(X)\in\mathrm{Tors}(\mc A_{\XX'})$ is isomorphic to an object in $\mathrm{Tors}(\mc A_\XX)$, thus the two categories are quasi-equivalent. 
\end{proof}

\begin{coro}
Up to canonical equivalence, the $A_\infty$-category $\mathrm{Tors}(\mc A_\XX)$ is independent of the choice of graded arc system $\XX$.
\end{coro}

\begin{df}
Given a punctured surface $(S,M)$ with $3$-compatible grading $\nu$ and a choice of ground field $\mathbf k$, let $\mc C(S,\nu):=\mathrm{Tors}(\mc A_\XX)$. 
This is a $\ZZ$-graded $A_\infty$-category over $\mathbf k$, independent of the arc system $\XX$ by the above corollary.
\end{df}

In the special case where $\nu$ is the horizontal foliation of a quadratic differential with simple zeros and $n$ simple poles, and $g$ is the genus of $S$, then we write $\mc C_{g,n}$ for the category $\mc C(S,\nu)$. 
We need to show that this category is well-defined, up to quasi-equivalence.

\begin{df}
A 3-compatible grading on a punctured surface $(S,M)$ is \textbf{geometric} if it is of the form $\mathrm{hor}(\varphi)$ for some quadratic differential $\varphi$ on $S$, meromorphic with respect to some complex structure, with $\mathrm{Zeros}(\varphi)\cup\mathrm{Poles}(\varphi)=M$.
All poles and zeros of $\varphi$ are necessarily simple by 3-compatibility. 
\end{df}

\begin{prop}
\label{prop_geomgrading}
All geometric gradings on decorated surface $(S,M)$ belong to the same component of the space of sections $\Gamma(S\setminus M;\PP(TS))$.
\end{prop}

\begin{proof}
Let $g=g(S)$, $n$ such that $|M|=4g-4+2n$, and $\mathcal M_{g,n}$ the Riemann moduli space (orbifold) of genus $g$ surfaces with $n$ marked points.
The cotangent space $T^*_X\mathcal M_{g,n}$ is identified with the space of quadratic differentials on $X$ with at most simple poles at the marked points (see e.g.~\cite{gardiner}).
Quadratic differentials away from the codimension one discriminant locus $\Delta\subset T^*\mathcal M_{g,n}$ have simple poles at the marked points and simple zeros.
It follows that the space of pairs $(J,\varphi)$ where $J$ is a complex structure on $S$ and $\varphi$ a quadratic differential as in the definition of a geometric grading, is connected.
(This argument fails, and its conclusion is generally false, for lower dimensional strata of $T^*\mathcal M_{g,n}$.)
\end{proof}

\subsection{3CY structure}
\label{subsec_3cy}

\subsubsection*{Properness}
A $\mathbf k$-linear triangulated category is \textit{proper} (over $\mathbf k$) if it has a generator, $E$, with $\dim_{\mathbf k}\mathrm{Ext}^\bullet(E,E)<\infty$.
Our first task in this subsection is to show that $\mc C(S,\nu)$ is proper.
In fact, we can find an explicit basis of $\mathrm{Ext}^\bullet$ in terms of boundary paths.

\begin{prop}
\label{prop_3cybasis}
The categories $\mc C(S,\nu)$ are proper.
Moreover, for an arc system $\XX$ on $S$, a basis of $\mathrm{Ext}^\bullet_{\mc A(\XX)}(GX,GY)$ is given by:
\begin{enumerate}
\item
Boundary paths, $b$, from $X$ to $Y$ which are ``short'' in the following sense: If $b$ is on a component of $\partial\widehat{S}$ mapping to a point $p\in M$, then $b$ should travel strictly less than $3/\mathrm{ind}(p)$ times ($=1$ for a zero, $=3$ for a pole) around the boundary circle of $\widehat{S}$.
Equivalently, $b$ is a proper factor of some $c_{p,X}$.
\item
If $X=Y$, an element of degree 3 corresponding to both $c_{p,X}$ and $c_{q,X}$, where $p,q$ are the endpoints of $X$.
\end{enumerate}
\end{prop}

\begin{proof}
The category $\mc C(S,\nu)$ is generated, by construction, by $E:=\bigoplus_{X\in\XX}GX$, so it suffices to show that $\dim_{\mathbf k}\mathrm{Ext}^\bullet_{\mc C(S,\nu)}(GX,GY)<\infty$ for any $X,Y\in\XX$.
By adjunction and the formula for the cone over the counit (Proposition~\ref{prop_adjunction}), $\mathrm{Ext}^\bullet_{\mathrm{Tw}_{\ZZ/2}(\mc A_\XX)}(GX,GY)$ is the cohomology of the complex
\begin{equation}\label{2stepcx}
\Hom_{\mc F_\XX}(X[2,1],Y)\oplus\Hom_{\mc F_\XX}(X,Y),\qquad (a,b)\mapsto(0,-\mk m_{2,0}(a,\mk m_{0,1})).
\end{equation}
If we consider the differential as a map from the first summand above to the second one, then its kernel is trivial and its image consists of multiples of $c_{p,X}-c_{q,X}$, in particular all paths which contain some $c_{p,X}$ as a proper factor.
It follows that the cokernel, which coincides with the cohomology of the complex, has the basis described in the statement of the proposition.
\end{proof}

The following proposition gives an explicit description of the endomorphism algebras of non-closed arcs.

\begin{prop}
\label{prop_arcext}
Let $X$ be an arc, identified with the object $GX\in\mc C(S,\nu)$.
\begin{enumerate}
\item
If $X$ connects distinct zeros, then 
\[
\Ext^\bullet(X,X)\cong H^\bullet(S^3;\mathbf k)
\] 
i.e. $X$ is a spherical object.
\item
If $X$ connects a zero $p$ with a pole $q$, then 
\[
\Ext^\bullet(X,X)\cong\mathbf k[x]/x^4
\]
where $x$ has bi-degree $(1,1)$ and corresponds to the cyclic path going once around $q$. 
\item
If $X$ is an arc connecting two distinct poles, $p$ and $q$, then 
\[
\Ext^\bullet(X,X)\cong \mathbf k[x,y]/(xy,x^3-y^3)
\]
where $x$ and $y$ have bi-degree $(1,1)$ and correspond to the cyclic paths around $p$ and $q$, respectively. 
\end{enumerate}
In all cases, the isomorphism is one of $A_\infty$-algebras, i.e. $\Ext^\bullet(X,X)$ is formal (without higher order terms).
Commutativity is understood in the $\ZZ\times\ZZ/2$-graded sense.
\end{prop}

\begin{proof}
Choose an arc system $\XX$ containing $X$.
We have already computed $\Ext^\bullet(X,X)$ as graded vector space.
The curved $A_\infty$-algebra has underlying algebra
\[
\Hom_{\mc A_{\XX}}(X,X)=\mathbf k[t,x,y]/(xy)
\]
were $x$ and $y$ have bi-degree $(1,1)$ or $(3,1)$, depending on the types of endpoints, and $t$ has bi-degree $(-1,1)$ as usual.

We claim that there are no immersed polygons, possibly with holes, which could contribute to the $A_\infty$-structure on $\Hom_{\mc A_{\XX}}(X,X)$.
To see this, suppose first that all $p\in M$ have index 3 (i.e. are zeros). 
Then $g(S)\geq 2$ by the Poincar\'e--Hopf index formula.
Suppose $\psi:D\to \widehat{S}$ is an immersed polygon which contributes to $\Hom_{\mc A_{\XX}}(X,X)$, then $\partial D$ maps to $X\cup \partial\widehat{S}\subset \widehat{S}$, so $\psi$ restricts to a topological covering over $\widehat{S}\setminus X$. But this contradicts $g(D)=0$.
In the general case, if there are some $p\in M$ with $\mathrm{ind}(p)=1$ (i.e. poles), we argue as follows.
Construct a $3:1$ cover, $S'$, of $\widehat{S}$ which restricts to a connected $3:1$ covering $S^1\to S^1$ over any boundary component of $\widehat{S}$ corresponding to a pole and to a trivial $3:1$ covering of $S^1$ over any boundary component corresponding to a zero. 
Then then lift of $\nu$ to $S'$ has index 3 around any boundary component of $S'$, so $g(S')\geq 2$.
Furthermore, any immersed polygon in our sense lifts to $S'$, so we have essentially reduced to the previous case, except that $X$ is replaced by a triple of arcs (the preimage of $X$ in $S'$) which may meet in one or both endpoints, which does not change the argument.

We conclude that $\Hom_{\mc A_{\XX}}(X,X)$ is just a graded algebra and that the quotient map from $\Hom_{C(S,\nu)}(GX,GX)$ to $\Ext^\bullet_{C(S,\nu)}(GX,GX)$, constructed in the proof of Proposition~\ref{prop_3cybasis}, is a quasi-isomorphism to the graded algebra
\[
\mathbf k[x,y]/(xy,x^{3/|x|}-y^{3/|y|})
\]
which specializes to the ones in the proposition, depending on the types of the endpoints.
\end{proof}

\begin{prop}
\label{prop_loopext}
Let $X\in \mathrm{Tw}_{\ZZ/2}(\mc F_\XX)$ be a spherical object, i.e. 
\[
\mathrm{Ext}^\bullet(X,X)=H^\bullet(S^1;\mathbf k),
\]
then 
\begin{equation}
\label{extS1S2}
\mathrm{Ext}^\bullet(GX,GX)\cong H^\bullet(S^1\times S^2;\mathbf k)
\end{equation}
as graded algebras.
Assume moreover that $\mathrm{Ext}^1(X,X)$ is even with respect to the additional $\ZZ/2$-grading, then $\mathrm{Ext}^\bullet(GX,GX)$ is formal, i.e. the above isomorphism is one of $A_\infty$-algebras.

In particular, these assumptions are satisfied if $X$ is an embedded closed curve, equipped with grading and rank one local system of $\mathbf k$-vector spaces.
\end{prop}

\begin{proof}
By Proposition~\ref{prop_adjunction}, $\mathrm{Ext}^\bullet(GX,GX)=\mathrm{Ext}^\bullet(FGX,X)$ and there is a long exact sequence
\[
\ldots \longrightarrow \mathrm{Ext}^i(X,X)\longrightarrow\mathrm{Ext}^i(FGX,X)\longrightarrow\mathrm{Ext}^i(X[2,1],X)\longrightarrow\ldots
\]
which already implies \eqref{extS1S2} on the level of graded vector spaces.
To see that the isomorphism is one of algebras note that $\mathrm{Ext}^2(GX,GX)$ is odd with respect the the additional $\ZZ/2$-grading, since this comes from the identity in $\mathrm{Ext}^0(X,X)$ shifted by an odd amount. 
In particular, the square of any element in $\mathrm{Ext}^1(GX,GX)$ must vanish.
Under the evenness assumption on $\mathrm{Ext}^1(X,X)$, $\mk m_n(\alpha,\ldots,\alpha)=0$ for any $\alpha\in\mathrm{Ext}^1(GX,GX)$ by the same reasoning.
By the classification of 3CY $A_\infty$-algebras in~\cite{ks} in terms of quivers with potential, this shows that $\mathrm{Ext}^\bullet(GX,GX)$ is classified by the one-loop quiver with vanishing potential, thus formal.

Suppose that $X$ is an embedded loop. 
It is explained in~\cite[Subsection 4.1]{hkk} how to associate to a loop, together with a choice of monodromy $\lambda\in\mathbf k^\times$, an object in $\mc F(S\setminus M,\nu)$, and that this object is spherical. 
Essentially, $X$ is represented in $\mc F(S\setminus M,\nu)$ by a twisted complex built from two arcs $Y,Z$ which intersect at both endpoints, so $\dim\mathrm{Hom}^1(Y,Z)=2$. 
The arcs $Y,Z$ are chosen so that their concatenation is isotopic to $X$, i.e. they are two ``halves'' of $X$. 
Concatenation is represented algebraically by forming the twisted complex $(Y\oplus Z;(\lambda,1))$ where $(\lambda,1)\in\mathrm{Hom}^1(Y,Z)$ is the twisting cochain and we use the natural basis coming from the pair of intersection points of $Y$ and $Z$. 
One can check that $\mathrm{Hom}^1(Y,Z)$ is even with respect to the additional $\ZZ/2$-degree, hence so is $\mathrm{Ext}^1$. 
\end{proof}

\subsubsection*{3CY structure}
An $n$-dimensional Calabi--Yau structure on a proper $A_\infty$-category $\mc A$ over $\mathbf k$ is given by a chain map
\[
\tau:CC_\bullet(\mc A)\to \mathbf k[-n]
\]
where $CC_\bullet(\mc A):=CC_\bullet(\mc A,\mc A)[u^{\pm 1}]/uCC_\bullet(\mc A,\mc A)[u]$ is the cyclic complex with differential $b+uB$ where $b$ is the Hochschild differential and $B$ is the Connes differential, such that the maps 
\[
\mc \Hom_{\mc A}^\bullet(X,Y)\otimes \Hom_{\mc A}^\bullet(Y,X)\to\mathbf k[-n], \qquad a\otimes b\mapsto\tau(\asm_2(a,b))
\] 
induce non-degenerate pairings
\begin{equation} \label{cy_pairing}
\mathrm{Ext}^i(X,Y)\otimes \mathrm{Ext}^{n-i}(Y,X)\to\mathbf k
\end{equation}
for all $i\in\ZZ$.
We refer to \cite{ks_ainfty,cho_lee} for details and the equivalence of this definition with others.

Our strategy for constructing a 3CY structure on $\mc C(S,\nu)$ is to first define a map $\tau_0:CC_\bullet(\mc F_\XX)\to\mathbf k[-3]$, then compose it with the map 
\[
CC_\bullet(\mathrm{Tors}(\mc A_\XX))\to CC_\bullet(\mathrm{Tw}^+(\mc F_\XX))\cong CC_\bullet(\mc F_\XX),
\]
where the first map is induced by the functor $F:\mathrm{Tors}(\mc A_\XX)\to \mathrm{Tw}^+(\mc F_\XX)$ as in Subsection~\ref{subsubsec_basechange} and the second map comes from Morita equivalence of $\mc F_\XX$ with $\mathrm{Tw}^+(\mc F_\XX)$, to obtain a map $\tau:CC_\bullet(\mathrm{Tors}(\mc A_\XX))\to\mathbf k[-3]$ for which non-degeneracy can then be checked.

The map $\tau_0$ has a very simple form and vanishes on all of $CC_\bullet(\mc F_\XX)$ except for the summand $\bigoplus_{X\in\XX} \Hom^\bullet(X,X)$ where it is defined on basis elements by $\tau(a):=1$ if $a=c_{p,X}$ for some $p\in M$, and $\tau(a):=0$ otherwise.

\begin{lemma}
The map $\tau_0$ defined above vanishes on the image of the Hochschild differential, hence is a chain map.
\end{lemma}

\begin{proof}
Using the explicit formula for the Hochschild differential (see e.g. \cite{seidel_hochschild,cho_lee}) this means that $\tau_0$ should vanish on elements of the form
\[
\sum_{i=1}^n(-1)^{(\|a_i\|+\ldots+\|a_1\|)(\|a_n\|+\ldots+\|a_{i+1}\|)}\asm_n(a_i,\ldots,a_1,a_n,\ldots,a_{i+1})
\]
where $a_1,\ldots,a_n$ is a cyclic sequence of morphisms.

The first case to check is when $a_1,a_2$ are boundary paths such that both compositions $a_1a_2$ and $a_2a_1$ are non-zero and $|a_1|+|a_2|=3$. 
In this case
\[
\asm_2(a_1,a_2)+(-1)^{\|a_1\|\|a_2\|}\asm_2(a_2,b_2)=\pm(a_1a_2-a_2a_1)
\]
for which $\tau_0(a_1a_2-a_2a_1)=1-1=0$.
The second case to check is when $a_1,\ldots,a_n$ is a disk sequence and $b$ is a complementary boundary path to $a_1$ in the sense that $\tau_0$ is non-zero on $ba_1$ and $a_1b$. Such $b$ is possibly a constant path (identity) and unique if it exists.
Then
\begin{align*}
\asm_n(a_n,\ldots, a_2,a_1ba_1)&+(-1)^{(\|a_1ba_1\|)(\|a_2\|+\ldots+\|a_n\|)}\asm_n(a_1ba_1,a_n,\ldots,a_2)\\
=-ba_1&+a_1b
\end{align*}
on which $\tau_0$ again vanishes.

Besides the two cases above there are no other terms of the form $\asm_n(a_n,\ldots,a_1)$ which produce a cyclic path of degree three and where $a_1,\ldots,a_n$ is a cyclic sequence. This proves the lemma.
\end{proof}

\begin{prop}
The functional $\tau$ defined above gives $\mc C(S,\nu)$ the structure of a 3CY category.
\end{prop}

\begin{proof}
It remains to show that the induced pairing \eqref{cy_pairing} is non-degenerate.
This follows from the explicit description of a basis for $\mathrm{Ext}^\bullet$ in terms of boundary paths.
Since each basis element has a unique complementary basis element with which it pairs non-trivially.
For a boundary path $a$ this is the boundary path $b$ with $ba=c_{p,X}$ for some $p\in M$ of $X\in\XX$, and for $1_X$ this is $c_{p,X}=c_{q,X}$ where $p,q\in M$ are the endpoints of $X$. 
\end{proof}

\section{Stability conditions}
\label{sec_stability}

In this section we prove our first main result, identifying a component of the space of stability conditions $\mathrm{Stab}(\mc C_{g,n})$ with a moduli space of quadratic differentials (see Subsection~\ref{subsec_spaces}).
This relies heavily on the corresponding theorem in~\cite{hkk} for the wrapped Fukaya category. 
What remains to be done is to transfer the stability conditions from the wrapped Fukaya categories to the 3CY categories defined in the previous section.
We show that such a transfer along an adjunction is possible in a somewhat more general (though still quite special) context in Subsection~\ref{subsec_transfer}.
The preliminary Subsection~\ref{subsec_stability} collects some definitions and basic results around stability conditions on triangulated categories.

\subsection{Definition and moduli space}
\label{subsec_stability}

The notion of a stability condition on a triangulated category is due to Bridgeland~\cite{bridgeland07}, inspired by properties of D-branes in string theory and semistable vector bundles on algebraic curves.
The definition we use here is not exactly Bridgeland's, but includes some modifications introduced by Kontsevich--Soibelman~\cite{ks} which have nowadays become fairly standard.

\begin{df}
Let $\mc C$ be an essentially small triangulated category, $\Gamma$ a finitely generated abelian group, and $\mathrm{cl}:K_0(\mc C)\to\Gamma$ an additive map.
A \textbf{stability condition} on $\mc C$ (more precisely the triple $(\mc C,\Gamma,\mathrm{cl})$) is given by 
\begin{enumerate}
\item an additive map $Z:\Gamma\to\CC$, the \textit{central charge}, and
\item for each $\phi\in\RR$ a full additive subcategory $\mc C_\phi\subset\mc C$, the \textit{semistable objects of phase $\phi$}
\end{enumerate}
satisfying the following properties:
\begin{itemize}
\item $\mc C_{\phi+\pi}=\mc C_\phi[1]$,
\item If $\phi_1<\phi_2$ and $E_i\in\mc C_{\phi_i}$ then $\Hom(E_2,E_1)=0$.
\item For each $E\in\mc C$ there is a tower of exact triangles
\[
\begin{tikzcd}[column sep=tiny]
0=E_0 \arrow{rr} & & E_1\arrow{rr}\arrow{dl} & & E_2\arrow{rr}\arrow{dl} & &\cdots \arrow{rr} &  & E_{n-1} \arrow{rr} & & E_n\cong E \arrow{dl} \\
& A_1 \arrow[dashed]{ul} & & A_2 \arrow[dashed]{ul} & & & &  & & A_n \arrow[dashed]{ul}
\end{tikzcd}
\]
with $0\neq A_i\in\mc C_{\phi_i}$, $\phi_1>\phi_2>\ldots>\phi_n$, the \textit{Harder--Narasimhan filtration}.
\item If $E\in\mc C_{\phi}$, $E\neq 0$, then $Z(E):=Z(\mathrm{cl}(E))\in\RR_{>0}e^{i\phi}$.
\item The \textit{support property} holds: There is a norm $\|.\|$ on $\Gamma\otimes_{\ZZ}\RR$ and $C>0$ such that
\[
\|\gamma\|\leq C|Z(\gamma)|
\]
for any $\gamma\in\Gamma$ such that there is a semistable object $E$ with $\gamma=\mathrm{cl}(E)$. 
\end{itemize}
\end{df}

Denote by $\mathrm{Stab}(\mc C)=\mathrm{Stab}(\mc C,\mathrm{cl})$ the set of stability condition for fixed $\mc C$, $\mathrm{cl}:K_0(\mc C)\to \Gamma$. 
Bridgeland defines a natural topology on $\mathrm{Stab}(\mc C)$ such that the map
\[
\mathrm{Stab}(\mc C,\mathrm{cl})\to\Hom(\Gamma,\CC),\qquad \left((\mc C_\phi)_{\phi\in\RR},Z\right)\mapsto Z
\]
becomes a local homeomorphism. 
Thus, $\mathrm{Stab}(\mc C)$ is locally modeled on the complex vector space $\Hom(\Gamma,\CC)$ and inherits the structure of a complex manifold.

\subsection{Transfer along adjunctions}
\label{subsec_transfer}

Stability conditions have notoriously bad functoriality properties:
Given an exact functor $G:\mc C\to \mc D$ between triangulated categories and a stability condition $\sigma$ on $\mc C$ there is in general no natural push-forward stability condition $G_*\sigma$ on $\mc D$, unless further conditions on $G$ and/or $\sigma$ are imposed.
Examples of such conditions are investigated in this subsection, where it turns out --- not only --- that $G_*\sigma$ is defined for any stability condition, but also that $G_*:\mathrm{Stab}(\mc C)\to \mathrm{Stab}(\mc D)$ identifies components of the domain with certain components of the target, see Theorem~\ref{thm_stab_transfer} below.

This subsection assumes somewhat more familiarity with stability conditions on triangulated categories (mainly from Bridgeland's original paper~\cite{bridgeland07}) than the rest of the paper. 

\begin{theorem}
\label{thm_stab_transfer}
Let $\mc C$ and $\mc D$ be triangulated categories, $G:\mc C\to \mc D$ an exact functor with left adjoint $F:\mc D\to \mc C$, and $d\geq 3$ such that 
\begin{equation}\label{tfg_exact_seq}
\mathrm{Cone}(\varepsilon_X:FGX\to X)\cong X[d]
\end{equation}
for $X\in\mc C$, where $\varepsilon_X:FGX\to X$ are the counit morphisms of the adjunction. 
Suppose also that the image of $G$ generates $\mc D$ under direct sums, shifts, and cones.
If $\mc C$ admits a bounded t-structure then $G$ induces an isomorphism $K_0(G):K_0(\mc C)\to K_0(\mc D)$.
Furthermore, if $\mathrm{cl}:K_0(\mc C)\to\Gamma$ is fixed, then there is a map
\begin{equation}\label{induced_map_on_stab}
G_*:\mathrm{Stab}(\mc C,\mathrm{cl})\longrightarrow\mathrm{Stab}\left(\mc D,\mathrm{cl}\circ K_0(G)^{-1}\right)
\end{equation}
biholomorphic onto its image, which is a union of connected components.
It is defined by sending $\sigma=(Z,(\mc C_\phi)_{\phi\in\RR})\in \mathrm{Stab}(\mc C,\mathrm{cl})$ to the stability condition $G_*(\sigma)$ with the same central charge $Z$ and so that the semistable objects of phase $\phi$ are $G(\mc C_\phi)$.
\end{theorem}

The proof of this theorem will take up the rest of this subsection.
Here is a sketch of the main steps:
\begin{itemize}
\item
By a result of Bridgeland, stability conditions are equivalently given by the heart of a bounded t-structure and a central charge satisfying certain properties. 
\item
Show that bounded t-structures can be pushed forward along $G$ (Lemma~\ref{lem_tstr_transfer}).
This uses an alternative characterization of hearts of bounded t-structures (Lemma~\ref{lem_tstr_char}), which might be of independent interest.
\item 
The previous step allows us to construct a map $G_*:\mathrm{Stab}(\mc C)\to \mathrm{Stab}(\mc D)$.
We use the metric on the space of stability conditions to show continuity and injectivity of this map.
\item
To show that the image of $G_*$ is closed, we use completeness of the space of stability conditions, a result of J.~Woolf.
\end{itemize}

We recall the definition of a heart of a bounded t-structure on a triangulated category~\cite{bridgeland07}.
The point is that while any t-structure $(\mc C_{\geq 0},\mc C_{\leq 0})$ in the sense of~\cite{bbd} has a heart $\mc A:=\mc C_{\geq 0}\cap \mc C_{\leq 0}[-1]$, the t-structure is in general not determined by $\mc A$, though it is in the bounded case, leading to the following equivalent definition.

\begin{df}
A full subcategory, $\mc A$, of a triangulated category $\mc C$ is the \textbf{heart of a bounded t-structure} if $\mathrm{Ext}^{<0}(A,B)=0$ for $A,B\in\mc A$, and for any $E\in\mc C$ there is a tower of exact triangles
\[
\begin{tikzcd}[column sep=tiny]
0=E_0 \arrow{rr} & & E_1\arrow{rr}\arrow{dl} & & E_2\arrow{rr}\arrow{dl} & &\cdots \arrow{rr} &  & E_{n-1} \arrow{rr} & & E_n\cong E \arrow{dl} \\
& A_1[k_1] \arrow[dashed]{ul} & & A_2[k_2] \arrow[dashed]{ul} & & & &  & & A_n[k_n] \arrow[dashed]{ul}
\end{tikzcd}
\]
where $k_1>k_2>\ldots>k_n$ are integers and $A_1,\ldots,A_n\in\mc A$.
\end{df}

The following characterization of bounded t-structures will be useful.

\begin{lemma}
\label{lem_tstr_char}
Let $\mc C$ be a triangulated category and $\mc A\subset\mc C$ a full additive subcategory. 
Then $\mc A$ is the heart of a bounded t-structure on $\mc C$ if and only if
\begin{enumerate}
\item
$\mathrm{Ext}^{<0}(A,B)=0$ for $A,B\in\mc A$,
\item
$\mc A$ is closed under extensions,
\item
every morphism, $f$, in $\mc A$ is \textit{admissible}, i.e. there is a triangle 
\[
K[1]\longrightarrow\mathrm{Cone}(f)\longrightarrow C\xrightarrow{+1}
\]
with $K,C\in\mc A$,
\item
the closure of $\mc A$ under shifts and cones is all of $\mc C$.
\end{enumerate}
\end{lemma}

For the proof of Lemma~\ref{lem_tstr_char} we will need the following basic lemma to modify filtrations in triangulated categories.
\begin{lemma}
\label{lem_filt_switch}
Let
\[
\begin{tikzcd}[column sep=tiny]
E_0 \arrow{rr} & & E_1\arrow{rr}\arrow{dl} & & E_2\arrow{dl} \\
& A_1 \arrow[dashed]{ul} & & A_2 \arrow[dashed]{ul} 
\end{tikzcd}
\]
be a pair of triangles in a triangulated category $\mc C$, $C\cong \mathrm{Cone}(A_2[-1]\to A_1)$, and $B_1\to C\to B_2\xrightarrow{+1}$ a triangle.
Then there is an $F_1\in \mc C$ and triangles
\[
\begin{tikzcd}[column sep=tiny]
E_0 \arrow{rr} & & F_1\arrow{rr}\arrow{dl} & & E_2\arrow{dl} \\
& B_1 \arrow[dashed]{ul} & & B_2 \arrow[dashed]{ul}.
\end{tikzcd}
\]
\end{lemma}

\begin{proof}
The octahedral axiom gives a triangle $E_2\to C\to E_0[1]\xrightarrow{+1}$.
Define $F_1$ as the cone of the composition $B_1[-1]\to C[-1]\to E_0$.
A second application of the octahedral axiom gives the triangle $B_2[-1]\to F_1\to E_2\xrightarrow{+1}$.
\end{proof}

\begin{proof}[Proof of Lemma~\ref{lem_tstr_char}]
According to \cite[Proposition 1.2.4]{bbd}, the first and third condition above imply that $\mc A$ is an abelian category whose short exact sequences are in one-to-one correspondence  with exact triangles in $\mc C$ with all vertices in $\mc A$.
Let $E\in\mc C$, then by the fourth assumption there is a tower of triangles
\[
\begin{tikzcd}[column sep=tiny]
0=E_0 \arrow{rr} & & E_1\arrow{rr}\arrow{dl} & & E_2\arrow{rr}\arrow{dl} & &\cdots \arrow{rr} &  & E_{n-1} \arrow{rr} & & E_n\cong E \arrow{dl} \\
& A_1[k_1] \arrow[dashed]{ul} & & A_2[k_2] \arrow[dashed]{ul} & & & &  & & A_n[k_n] \arrow[dashed]{ul}
\end{tikzcd}
\]
with $A_i\in\mc A$, $k_i\in\ZZ$.
If the sequence of $k_i$ is not strictly decreasing, then there is an $i$ with $1\leq i<n$ such that $k_i\leq k_{i+1}$.
We claim that there is a modified tower such that the number of pairs $(i,j)$ with $1\leq i<j\leq n$ and $k_i\leq k_j$ is strictly smaller.
An inductive argument then implies that there is a tower of extensions as above such that $k_i$ are strictly decreasing, showing that $\mc A$ is the heart of a bounded t-structure.

\underline{Case $k_i=k_{i+1}=:k$}.
The composite morphism $A_{i+1}[k]\to E_i\to A_i[k]$ of degree one determines an extension $A_i\to B\to A_{i+1}\xrightarrow{+1}$, and $B\in\mc A$ by the second assumption of the lemma. 
Applying the octahedral axiom one sees that the tower of extensions can be replaced by a shorter one with $E_i$ removed and new triangle $E_{i+1}\to B[k]\to E_{i-1}\xrightarrow{+1}$.

\underline{Case $k_i=k_{i+1}-1=:k$}.
The composite morphism $A_{i+1}[k+1]\to E_i\to A_i[k]$ has degree zero as a morphism $f:A_{i+1}\to A_i$, i.e. is a morphism in $\mc A$.
By the third assumption of the lemma that is a triangle 
\[
K[1]\longrightarrow\mathrm{Cone}(f)\longrightarrow C\xrightarrow{+1}
\]
with $K,C\in\mc A$, thus by Lemma~\ref{lem_filt_switch} there is a modified tower with $A_i[k]$, $A_{i+1}[k+1]$, $E_i$, replaced by $K[k+1]$, $C[k]$, and some $E_i'$, respectively, i.e. $k_i$ and $k_{i+1}$ are switched.

\underline{Case $k_i<k_{i+1}-1$}.
The composite morphism $A_{i+1}[k]\to E_i\to A_i[k]$ has degree $k_i-k_{i+1}+1<0$ as a morphism from $A_{i+1}$ to $A_i$, so must vanish by the first assumption of the lemma. 
By Lemma~\ref{lem_filt_switch} there is a modified tower with $A_i[k_i]$ and $A_{i+1}[k_{i+1}]$ exchanging places.
\end{proof}

The next lemma shows that bounded t-structures can be pushed forward along a certain class of functors.

\begin{lemma}
\label{lem_tstr_transfer}
Let $G:\mc C\to\mc D$ be an exact functor of triangulated categories with the property that if $X,Y\in\mc C$ with $\mathrm{Ext}^{<0}(X,Y)=0$, then $G$ gives an isomorphism $\mathrm{Ext}^i(X,Y)\to\mathrm{Ext}^i(GX,GY)$ for any $i\leq 1$, and that the image of $G$ generates $\mc D$ under shifts and cones.
Then if $\mc A\subset \mc C$ is the heart of a bounded t-structure, so is $G(\mc A)\subset \mc D$, and $\mc A\cong G(\mc A)$ as abelian categories.
\end{lemma}

\begin{proof}
We will check that the full subcategory $G(\mc A)\subset \mc D$ satisfies the conditions of Lemma~\ref{lem_tstr_char}.
For (1) note that $X,Y\in\mc A$ implies $\mathrm{Ext}^{<0}(X,Y)=0$ so also $\mathrm{Ext}^{<0}(GX,GY)=0$ by assumption.
For (2) suppose that $X,Y\in\mc A$ and $GX\to E\to GY\xrightarrow{+1}$ is a triangle in $\mc D$.
By assumption, $G$ induces an isomorphism $\mathrm{Ext}^1(Y,X)\cong \mathrm{Ext}^1(GY,GX)$, so there is a lift of the degree one morphism $GY\to GX$ which can be completed to a triangle whose image in $\mc D$ is isomorphic to the original triangle.
Hence, $G(\mc A)$ is closed under extensions.
For (3) suppose that $X,Y\in\mc A$ and $f\in\Hom(GX,GY)$. 
By assumption, $f$ lifts to an morphism $\tilde{f}:X\to Y$, and since morphisms in $\mc A$ are admissible there is an exact triangle
\[
K[1]\longrightarrow\mathrm{Cone}(\tilde{f})\longrightarrow C\xrightarrow{+1}
\]
whose image under $G$ shows that $f$ is admissible, since $G(\mathrm{Cone}(\tilde f))\cong \mathrm{Cone}(f)$ by exactness of $G$.
For (4) note that the triangulated closure of $\mc A$ is all of $\mc C$, since it is the heart of a bounded t-structure.
Hence the closure of $F(\mc A)$ in $\mc D$ contains $F(\mc C)$, whose closure is by assumption $\mc D$.
\end{proof}

The following result, essentially due to Bridgeland~\cite{bridgeland07}, gives an equivalent definition of stability conditions in terms of bounded t-structures.

\begin{prop}
\label{prop_stab_tstr}
Let $\left (Z,(\mc C_\phi)_\phi\right)$ be a stability condition on a triangulated category $\mc C$ with homomorphism $\mathrm{cl}:K_0(\mc C)\to\Gamma$, then the full subcategory $\mc C_{(0,\pi]}\subset \mc C$ of objects with semistable factors of phase contained in the interval $(0,\pi]$ is abelian and the heart of a bounded t-structure.

Conversely, suppose we have an additive map $Z:\Gamma\to\CC$ and the heart of a bounded t-structure $\mc A\subset \mc C$ such that:
\begin{itemize}
\item $Z(E)\in\RR_{>0}e^{i(0,\pi]}$ for $0\neq E\in\mc A$. Thus the phase $\phi(E):=\mathrm{Arg}(Z(E))\in (0,\pi]$ is well-defined for $0\neq E\in\mc A$ and $E$ is by definition semistable if 
\[
0\neq A\subset E\implies \phi(A)\leq\phi(E)
\]
where $A\in\mc A$ is any non-zero subobject of $E$.
\item If $0\neq E\in\mc A$ then there is a filtration
\[
0=E_0\subset E_1\subset \ldots\subset E_{n-1}\subset E_n\cong E
\]
such that $A_i:=E_i/E_{i-1}$ is semistable for $1\leq i\leq n$ with $\phi(A_1)>\phi(A_2)>\ldots>\phi(A_n)$.
\item $Z$ satisfies the support property.
\end{itemize}
Then defining $\mc C_\phi$ for $\phi\in (0,\pi]$ to be semistable objects with $\phi(E)=\phi$ and the zero object, which fixes $\mc C_\phi$ for all $\phi\in\RR$ by $\mc C_\phi[k]=\mc C_{\phi+k\pi}$, gives a stability condition on $\mc C$.
\end{prop}

\begin{proof}[Proof of Theorem~\ref{thm_stab_transfer}]
The assumptions of Lemma~\ref{lem_tstr_transfer} are satisfied since~\eqref{tfg_exact_seq} gives a long exact sequence
\[
\ldots\to\mathrm{Ext}^i(X,Y)\longrightarrow \underbrace{\mathrm{Ext}^i(FGX,Y)}_{\mathrm{Ext}^i(GX,GY)}\longrightarrow\underbrace{\mathrm{Ext}^i(X[d-1],Y)}_{\mathrm{Ext}^{i-d+1}(X,Y)}\to\ldots.
\] 
Furthermore, if $\mc A\subset \mc C$ is the heart of a bounded t-structure (which exists by assumption), then there is a commutative square of maps between Grothendieck groups
\[
\begin{tikzcd}
K_0(\mc A) \arrow{r}\arrow{d} & K_0(G(\mc A)) \arrow{d} \\
K_0(\mc C) \arrow{r} & K_0(\mc D) 
\end{tikzcd}
\]
where the vertical maps are induced by inclusion functors and are isomorphisms (this is a general property of bounded t-structures) and the top horizontal arrow is an isomorphism since $\mc A\cong G(\mc A)$ as abelian categories, thus the bottom horizontal arrow is also an isomorphism.
Thus, using the definition of stability conditions in terms of t-structures (Proposition~\ref{prop_stab_tstr}) we get an induced map $G_*$ on sets of stability conditions as in \eqref{induced_map_on_stab}.

Next we want to check continuity of $G_*$.
The topology on $\mathrm{Stab}$ is the subspace topology of the product topology on the product of the set of slicings and $\Hom(\Gamma,\CC)$. 
The topology on the latter is the standard Euclidean topology, while the topology on the former comes from a metric on the space of slicings. 
This metric can be written as 
\[
d(\mc P,\mc Q)=\inf\left\{\varepsilon\in [0,\infty)\mid \mc Q_\phi\subset \mc P_{[\phi-\varepsilon,\phi+\varepsilon]}\forall \phi\in\RR \right\}
\]
according to~\cite[Lemma 6.1]{bridgeland07}.
If we start with stability conditions $\sigma$ on $\mc C$ with slicing $\mc P$, then $G_*(\sigma)$ has  a slicing $G_*\mc P$ in $\mc D$ satisfying 
\[
\left(G_*\mc P\right)_{[\phi-\varepsilon,\phi+\varepsilon]}= G_*\left(\mc P_{[\phi-\varepsilon,\phi+\varepsilon]}\right)
\]
whenever $\varepsilon<\frac{\pi}{2}$. 
This follows, since everything is contained in the heart of a bounded t-structure of the form $\mc D_{(\phi,\phi+\pi]}$.
In particular if $\sigma_1,\sigma_2\in\mathrm{Stab}(\mc C)$ have slicing $\mc P,\mc Q$, respectively, with $d(\mc P,\mc Q)<\varepsilon< \frac{\pi}{2}$, then 
\[
\left(G_*\mc Q\right)_\phi=G_*(\mc Q_\phi)\subset G_*\left(\mc P_{[\phi-\varepsilon,\phi+\varepsilon]}\right)=\left(G_*\mc P\right)_{[\phi-\varepsilon,\phi+\varepsilon]}
\]
thus $d(G_*\mc P,G_*\mc Q)\leq d(\mc P,\mc Q)$, which shows continuity of $G_*$, since the map on $\Hom(\Gamma,\CC)$ is just the identity.

Since the complex structure on $\mathrm{Stab}$ is lifted from the one on $\Hom(\Gamma,\CC)$, and $G_*$ is the identity on the latter, $G_*$ is a holomorphic local homeomorphism.

It remains to show that $G_*$ is injective and that the image of $G_*$ is closed.
For both we use the metric, defined by Bridgeland,
\[
d(\sigma_1,\sigma_2):=\sup_{0\neq E\in\mc C}\max\left\{|\phi^-_{\sigma_1}(E)-\phi^-_{\sigma_2}(E)|,|\phi^+_{\sigma_1}(E)-\phi^+_{\sigma_2}(E)|,\left|\log\frac{m_{\sigma_1}(E)}{m_{\sigma_2}(E)}\right|\right\}
\]
on $\mathrm{Stab}$.
Here, $m(E):=|Z(A_1)|+\ldots+|Z(A_n)|$, $\phi^+(E):=\phi_1$, and $\phi^-(E):=\phi_n$, where $A_1,\ldots,A_n$ are the semistable components of $E$ with phases $\phi_1>\ldots >\phi_n$.
(Strictly speaking this is a \textit{pseudometric}, i.e. distances can be $+\infty$.)
Since $G_*$ preserves Harder--Narasimhan filtrations we find that
\begin{equation}
\label{Gstar_expanding}
d(G_*\sigma_1,G_*\sigma_2)\geq d(\sigma_1,\sigma_2),\qquad \sigma_1,\sigma_2\in\mathrm{Stab}(\mc C).
\end{equation}
This gives injectivity.

To show that the image of $G_*$ is closed, we use the fact, due to Woolf~\cite{woolf12}, that the above metric is complete. 
(A priori, the finiteness condition in \cite{woolf12} is different from the support property, however the discussion in \cite[Appendix B]{bayer_macri} shows that the two coincide in our case.)
Now, if $G_*(\sigma_i)\to \sigma$ in $\mathrm{Stab}(\mc D)$ then, since $G_*(\sigma_i)$ form a Cauchy sequence, the $\sigma_i$ themselves form a Cauchy sequence in $\mathrm{Stab}(\mc C)$ by~\eqref{Gstar_expanding}, which converges by completeness.
By continuity, this implies that the image of $G_*$ is closed.

Combining all of the above, $G_*$ must be biholomorphic onto a union of connected components.
\end{proof}

\subsection{Proof of the first main theorem}
\label{subsec_spaces}

In this subsection we apply Theorem~\ref{thm_stab_transfer} from the previous subsection to establish the first main result of this paper, identifying a component of $\mathrm{Stab}(\mc C_{g,n})$ with a space of quadratic differentials with $n$ simple poles on genus $g$ surfaces.

Fix $g\geq 0$ and $n\geq 0$ with $n\geq 4$ if $g=0$ and $n\geq 2$ if $g=1$.
Choose a compact Riemann surface $S$ of genus $g$ and a quadratic differential $\varphi$ with simple zeros and exactly $n$ poles, all of which are simple.
Let $M\subset S$ be the set of zeros and poles of $\varphi$ and $\nu=\mathrm{hor}(\varphi)$ be the horizontal foliation of $\varphi$ on $S\setminus M$.
For any other choice $(S',\varphi')$ there is a diffeomorphism of punctured surfaces $f:S\to S'$ which sends $\nu$ to a foliation homotopic to $\mathrm{hor}(\varphi')$, so $(S,M,\nu)$ is, up to equivalence, uniquely determined by $g$ and $n$ (Proposition~\ref{prop_geomgrading}).

Consider the finitely generated abelian group
\[
\Gamma_{S,\nu}:=H_1(S,M;\ZZ\otimes_{\ZZ/2}\Sigma_\nu)
\]
where $\Sigma_\nu\to S\setminus M$ is the double cover given by the two orientations of $\nu(p)\subset T_pS$, as in Subsection~\ref{subsec_Z2}, and $\ZZ\otimes_{\ZZ/2}\Sigma_\nu$ is the local system of abelian groups locally isomorphic to $\ZZ$ but globally twisted by $\Sigma_\nu$.
Since $\nu=\mathrm{hor}(\varphi)$ we can equivalently define $\Sigma_\nu$ to be the double cover where the abelian differential $\sqrt{\varphi}$ is well-defined.

The moduli space $\mc Q_{S,\nu}$ was defined in~\cite{hkk} and in the introduction.
(In the introduction the notation $\mc Q_{g,n}$ was used. As discussed above, $(S,\nu)$ is determined up to isomorphism by $(g,n)$, so the two may be used interchangeably to some extent. However, $\mc Q_{g,n}$ is only defined up to non-unique isomorphism.)
If $(J,\varphi,h)$ represents a point in $\mc Q_{S,\nu}$ then there is a canonical isomorphism
\[
\Gamma_{S,\nu}=H_1(S,M;\ZZ\otimes_{\ZZ/2}\Sigma_{\mathrm{hor}(\varphi)})
\]
since we can use $h$ to identify the two double covers.
Integration along paths gives an additive map
\[
Z:\Gamma_{S,\nu}\to\CC,\qquad \gamma\mapsto\int_\gamma\sqrt{\varphi}
\]
and letting the point in $\mc Q_{S,\nu}$ vary, we get a \textit{period map}
\[
\mc Q_{S,\nu}\to \mathrm{Hom}\left(\Gamma_{S,\nu},\CC\right)
\]
which can be shown to be a local homeomorphism, see~\cite{hkk}, thus providing $\mc Q_{S,\nu}$ with the structure of a complex manifold.

Up to isomorphism, $\mc Q_{S,\nu}$ depends only on $g$ and $n$.
We note that while certain moduli spaces of quadratic differentials/half-translation surfaces have several connected components~\cite{lanneau}, all $\mc Q_{S,\nu}$ are connected, essentially since this is the generic stratum of $T^*\mc M_{g,n}$, c.f.~Proposition~\ref{prop_geomgrading}.

A special case of the main result of~\cite{hkk} is the following theorem.

\begin{theorem}[\cite{hkk}]
For $g,n\geq 0$ choose  $(S,M,\nu)$ as above and let $\mc F(S\setminus M,\nu)$ be the wrapped Fukaya category (over any field). 
Then there is a canonical isomorphism of abelian groups $K_0(\mc F(S\setminus M,\nu))=\Gamma_{S,\nu}$ and a continuous map 
\[
\mc Q_{S,\nu}\to \mathrm{Stab}(\mc F(S\setminus M,\nu))
\]
which is a homeomorphism onto a connected component of its target and makes the diagram
\[
\begin{tikzcd}
\mc Q_{S,\nu}\arrow{d}\arrow{r} & \mathrm{Stab}(\mc F(S\setminus M,\nu)) \arrow{d} \\
\mathrm{Hom}(\Gamma_{S,\nu},\CC) \arrow{r} & \mathrm{Hom}(K_0(\mc F(S\setminus M,\nu)),\CC)  
\end{tikzcd}
\]
commute.
Moreover, given $(J,\varphi,h)\in \mc Q_{S,\nu}$, the semistable objects of phase $\phi\in\RR$ of the corresponding stability conditions are obtained as extensions of those objects which are represented by saddle connections and closed geodesics of phase $\phi$.
\end{theorem}

Combining the above theorem with Theorem~\ref{thm_stab_transfer} applied to the functor $G:\mc F(S\setminus M,\nu)\to\mc C(S,\nu)$ gives our first main result.
This functor satisfies the conditions of Theorem~\ref{thm_stab_transfer} by Proposition~\ref{prop_adjunction} and Lemma~\ref{lem_inv_trivial}.

\begin{theorem}
\label{thm_stab_C}
For $g,n\geq 0$ choose the punctured surface with grading $(S,M,\nu)$ as above.
Then there is a canonical isomorphism of abelian groups $K_0(\mc C(S,\nu))=\Gamma_{S,\nu}$ and a continuous map 
\[
\mc Q_{S,\nu}\to \mathrm{Stab}(\mc C(S,\nu))
\]
which is a homeomorphism onto a connected component of its target and makes the diagram
\[
\begin{tikzcd}
\mc Q_{S,\nu}\arrow{d}\arrow{r} & \mathrm{Stab}(\mc C(S,\nu)) \arrow{d} \\
\mathrm{Hom}(\Gamma_{S,\nu},\CC) \arrow{r} & \mathrm{Hom}(K_0(\mc C(S,\nu)),\CC)  
\end{tikzcd}
\]
commute.
Moreover, given $(J,\varphi,h)\in \mc Q_{S,\nu}$, the semistable objects of phase $\phi\in\RR$ of the corresponding stability conditions are obtained as extensions of those objects which are represented by saddle connections and closed geodesics of phase $\phi$.
\end{theorem}

\section{DT invariants}
\label{sec_dt}

The goal of this section is to prove our second main result, Theorem~\ref{thm2_intro}, which expresses the DT invariants of the 3CY categories with stability conditions constructed from quadratic differentials in the previous sections, in terms of counts of saddle connections and flat cylinders.
On the one hand, the computations themselves are not too difficult thanks to the fact that at most one-parameter families of stable objects appear. 
On the other hand, the general machinery of DT invariants is fairly involved, and thus this section is, unfortunately, significantly less self-contained than previous ones.
In particular, while we briefly review some of the parts of Kontsevich--Soibelman's seminal work~\cite{ks} in Subsections~\ref{subsec_dt_intro} and~\ref{subsec_orientation}, some familiarity with its content is essential in order to follow all the proofs.

Subsections~\ref{subsec_pairing} and~\ref{subsec_generic} are more elementary. 
The first describes the skew-symmetric Euler pairing on $K_0(\mc C(S,\nu))$ in geometric terms, while the second investigates the behavior of saddle connections and closed geodesics under the genericity condition on $\varphi$.
The goal of Section~\ref{subsec_orientation} is to construct orientation data on $\mc C(S,\nu)$, which is necessary for the definition of the motivic DT invariants.
Actually, we only construct \textit{partial} orientation data, but this turns out to be good enough for our purposes.
Finally, Section~\ref{subsec_dtcomp} completes the proof of Theorem~\ref{thm2_intro}.
The computation is based on the knowledge of the DT invariants for a certain small list of quivers with potential. 
Some parts of the calculation are relegated to Appendix~\ref{sec_qdilog}.

\subsection{Overview of motivic Donaldson--Thomas invariants}
\label{subsec_dt_intro}

In their landmark paper~\cite{ks}, Kontsevich--Soibelman proposed a way of ``counting'' semistable objects in a 3-d Calabi--Yau triangulated $A_\infty$-category, motivated by analogies with mirror symmetry.
The invariants depend on the following three pieces of data, analogous to a 3-d Calabi--Yau variety, spin structure, and K\"ahler class, respectively.

\begin{enumerate}
\item An ind-constructible, 3CY $A_\infty$-category, $\mathcal C$, over a field of characteristic zero,
\item Orientation data on $\mathcal C$: a square-root of the determinant line bundle $\mathrm{Det}(\mathrm{Ext}^\bullet_{\mc C}(X,X))$, see also Subsection~\ref{subsec_orientation}, and
\item A stability condition on $\mathcal C$ (a variant of the definition in Subsection~\ref{subsec_stability} adapted to the ind-constructible setting).
\end{enumerate}

Part of the data of a stability condition is a map $\mathrm{cl}:K_0(\mc C)\to \Gamma$ to a finite rank lattice with anti-symmetric pairing so that $\mathrm{cl}$ sends the Euler pairing $(X,Y)\mapsto\chi \Ext^\bullet(X,Y)$ to this pairing on $\Gamma$.
For $\gamma\in\Gamma$ with $Z(\gamma)\neq 0$ one defines an invariant $A(\gamma)_{mot}$ in a certain ring of motives which counts semistable objects of phase $\phi=\mathrm{Arg}(Z(\gamma))$ in the class $\gamma$, in a suitable sense.
More precisely,
\[
A(\gamma)_{mot}=\int_{\mc C_{\phi,\gamma}}w(E)
\]
where $\mc C_{\phi,\gamma}$ is the constructible stack of semistable objects with phase $\phi$ and class $\gamma$, the integral sign is pushforward of motives to the point, and $w$ is the motivic weight (measure) given by
\[
w(E)=\mathbb{L}^{\frac{1}{2}\left(\sum_{i\leq 1}(-1)^i\dim\mathrm{Ext}^i(E,E)-\mathrm{rk}Q_E\right)}(1-MF(E))(1-MF_0(Q_E)).
\]
In the above formula $E$ is an object in $\mc C$, $\mathbb{L}$ is the motive of $\mathbb A^1$, $MF(E)$ is the Milnor fiber of the potential 
\[
W_E^{min}(\alpha):=\sum_{n\geq 3}\langle \mk m_{n-1}(\alpha,\ldots,\alpha),\alpha\rangle
\]
where $\alpha\in\mathrm{Ext}^1(E,E)$ and $\mk m_n$ are the $A_\infty$-operations of the minimal model of $\mc C$ (where $\mk m_1=0$), and the $Q_E$ terms come from orientation data on $\mc C$.

\begin{remark}
Some informal remarks which may shed light on the above formula:
Critical points of $W_E^{\min}$ correspond to objects of $\mc C$ in a formal neighborhood of $E$. 
The Milnor fiber provides a way of counting critical points ``with multiplicity'', compensating for the fact that the potential is typically not Morse--Bott.
Passing to the $A_\infty$-minimal model removes quadratic terms from the potential (which would come from $\mk m_1$). 
However, the motivic Milnor fiber is not completely invariant under removing quadratic terms, a problem which is fixed by orientation data.
\end{remark}

To simplify matters, it is convenient to pass from motives to rational functions in $\sqrt{q}$ by applying the (twisted) Serre polynomial $S$, which sends $\sqrt{\LL}$ to $\sqrt{q}$. 
Thus let $A(\gamma)_q:=S(A(\gamma)_{mot})$.
Assume that there is a unique primitive class $\gamma\in\Gamma$ with given phase $\mathrm{Arg}(Z(\gamma))=\phi$, which holds for generic stability condition.
Suppose further that there is a factorization
\[
\sum_{n\geq 0}A(n\gamma)_qx^n=\prod_{n=1}^\infty\prod_{k\in\ZZ}\mathbf E\left((-q^{1/2})^kx^n\right)^{(-1)^k\Omega_{k}(n\gamma)}
\]
where $\mathbf E$ is the quantum dilogarithm as in Appendix~\ref{sec_qdilog} and $\Omega_k(n\gamma)\in\ZZ$ with $\Omega_k(n\gamma)=0$ for all but finitely many $k\in\ZZ$, depending on $n$.
(This is \textit{admissibility} in the sense of~\cite{ks_coha}, except that $q^{1/2}$ is replaced by $q^{-1/2}$.) 
The Laurent polynomials in $q^{1/2}$
\[
\Omega(\gamma):=\sum_{k\in\ZZ}\Omega_k(\gamma)q^{k/2}
\]
are then the \textit{(quantum/refined) DT invariants} and the \textit{(numerical) DT invariants} are $\Omega(\gamma)_{num}:=\Omega(\gamma)(-1)$, i.e. setting $q^{1/2}=-1$.

The ambitious scope of~\cite{ks} meant that the some crucial steps in the construction of the invariants were left as conjectures and some parts only sketched.
Many of those conjectures were later resolved, at least partly. 
The \textit{integral identity} was proven for algebraically closed fields by Thuong~\cite{thuong}.
The integrality conjecture was proven in the case of quivers with potential by Davison--Meinhardt~\cite{davison_meinhardt20}.
Orientation data was constructed in the case of derived categories of Calabi--Yau threefolds by Joyce--Upmeier~\cite{joyce_upmeier}.

A special case which is nowadays very thoroughly understood is that of 3CY categories constructed from quivers with (polynomial) potential, see for instance~\cite{davison_meinhardt15}.
While the 3CY categories $\mc C_{g,n}$ constructed in this paper are \textit{not} of this form, it turns out that the subcategories of semistable objects of a given slope do have a description in terms of a certain finite list of simple quivers (see Subsection~\ref{subsec_dtcomp}), so the construction of the DT invariants does not depend on any of the above mentioned conjectures.

\subsection{Skew-symmetric pairing}
\label{subsec_pairing}

Recall from Subsection~\ref{subsec_spaces} the isomorphism of abelian groups
\begin{equation} \label{K0_H1_iso}
K_0(\mc C(S,\nu))\cong H_1(S,M;\ZZ\otimes_{\ZZ/2} \Sigma_\nu)=:\Gamma_{S,\nu}
\end{equation}
where $\Sigma_\nu\to S\setminus M$ is the double cover of orientations of $\nu$.
Since $\mc C(S,\nu)$ is proper over $\mathbf k$, the Grothendieck group $K_0(\mc C(S,\nu))$ has an \textit{Euler pairing}
\[
(X,Y)\mapsto \sum_i(-1)^i\dim_{\mathbf k}\Ext^i(X,Y)
\]
which is skew-symmetric since $\mc C(S,\nu)$ is odd Calabi--Yau.
The goal of this subsection is to show that this pairing corresponds to a natural geometric pairing on $H_1$ under the isomorphism~\eqref{K0_H1_iso}, which is moreover non-degenerate up to torsion.

For brevity, write $\tau:=\ZZ\otimes_{\ZZ/2}\Sigma_\nu$ for the twisted local system.
First, in order to apply Lefschetz duality (Poincar\'e duality for manifolds with boundary), we replace $S$ by its real blow up $\widehat{S}$ as in Section~\ref{sec_3cysurf}.
Then
\[
H_1(S,M;\tau)\cong H_1(\widehat{S},\partial\widehat{S};\tau)
\]
and Lefschetz duality provides an isomorphism
\begin{equation}
\label{PD_iso}
H_1(S;\tau)\to H^1(S,\partial S;\tau)
\end{equation}
since $\tau\otimes_{\ZZ}\tau\cong \ZZ$.
Furthermore, the universal coefficient theorem gives a pairing
\begin{equation}\label{ho_coho_pairing}
H^1(\widehat{S},\partial\widehat{S};\tau)\otimes H_1(\widehat{S},\partial\widehat{S};\tau)\to\ZZ
\end{equation}
which is non-degenerate up to torsion.
Also, as part of the long exact sequence of the pair $(\widehat{S},\partial\widehat{S})$ we get
\[
0\to H_1(\widehat{S};\tau)\to H_1(\widehat{S},\partial \widehat{S};\tau)\to H_0(\widehat{S};\tau)
\]
where $H_0(\widehat{S};\tau)$ is a product of $\ZZ/2$'s, since $\tau$ has non-trivial monodromy around each component of $\partial\widehat{S}$.
Thus, the first map can be inverted on $2H_1(\widehat{S},\partial \widehat{S};\tau)$ and we consider the composite map
\[
H_1(\widehat{S},\partial \widehat{S};\tau)\xrightarrow{2\cdot} 2H_1(\widehat{S},\partial \widehat{S};\tau) \to H_1(\widehat{S};\tau)\to H^1(\widehat{S},\partial \widehat{S};\tau)
\]
where the third map is~\eqref{PD_iso}.
Combining this map with the pairing~\eqref{ho_coho_pairing} we finally obtain a pairing
\[
H_1(\widehat{S},\partial\widehat{S};\tau)\otimes H_1(\widehat{S},\partial\widehat{S};\tau)\to\ZZ
\]
which is non-degenerate up to torsion.

\begin{prop}
The pairing on $H_1(S,M;\ZZ\otimes_{\ZZ/2}\Sigma_\nu)=H_1(\widehat{S},\partial\widehat{S};\tau)$ constructed above coincides with the Euler pairing on $K_0(\mc C(S,\nu))$ under the isomorphism~\eqref{K0_H1_iso}.
\end{prop}

\begin{proof}
Fix an arc system $\XX$ on $S$ and a choice of grading of each arc.
The classes of $G(X)\in\mc C(S,\nu)$, $X\in\XX$, generate $K_0(\mc C(S,\nu))$ by construction.
On the other hand, each $X\in\XX$ comes with an orientation determined by the grading (c.f. Subsection~\ref{subsec_Z2}) and thus belongs to a well-defined class in $H_1(S,M;\ZZ\otimes_{\ZZ/2}\Sigma_\nu)$, which is also the image of $G(X)$ under the isomorphism~\eqref{K0_H1_iso}.
To check the statement of the proposition, it thus suffices to compute the value of both pairings on pairs of arcs in $\XX$.

\begin{figure}
\centering
\includegraphics[scale=1]{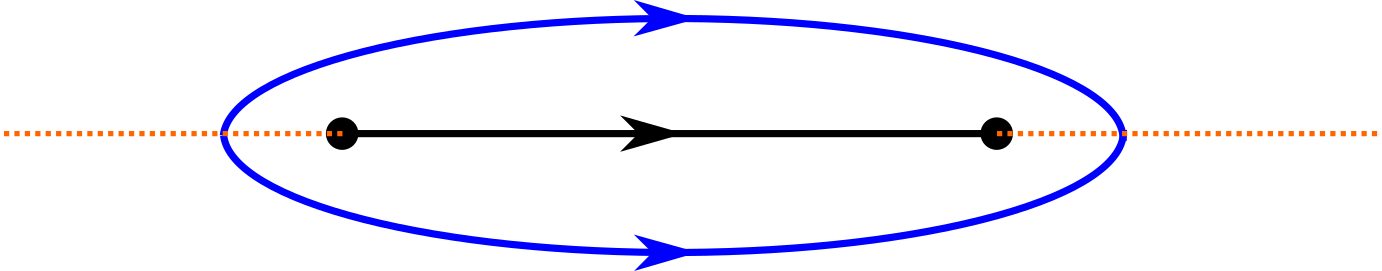}
\caption{Lifting an arc to a class in $H_1(S\setminus M;\ZZ\otimes_{\ZZ/2}\Sigma_\nu)$. Branch cuts for $\Sigma_\nu$ are indicated as dotted lines. The choice of sheet of $\Sigma_\nu$ along the curves should be constant away from the branch cuts. Since the monodromy of $\Sigma_\nu$ around the endpoints is non-trivial, in order to obtain a cycle, the orientation switches as one crosses a branch cut.}
\label{fig_arclift}
\end{figure}

For a pair of arcs $X,Y\in\XX$ write $\langle GX,GY\rangle$ for the pairing on $K_0(\mc C(S,\nu))$ and $\langle X,Y\rangle$ for the pairing on $H_1(S,M;\ZZ\otimes_{\ZZ/2}\Sigma_\nu)$.
The latter can be computed as follows.
The image of $X$ under the map
\[
H_1(S,M;\ZZ\otimes_{\ZZ/2}\Sigma_\nu)\to H_1(S;\ZZ\otimes_{\ZZ/2}\Sigma_\nu)
\]
used in the definition of the pairing is represented by a closed loop, $\widetilde{X}$, which goes once around $X$, see Figure~\ref{fig_arclift}.
This is because the image of $\widetilde{X}$ in $H_1(S,M;\ZZ\otimes_{\ZZ/2}\Sigma_\nu)$ under the natural map is \textit{twice} the class of $X$, as one can see from the figure.
The pairing $\langle X,Y\rangle$ is then just the intersection number between $\widetilde X$ and $Y$.
In particular, $\langle X,X\rangle=0$ whenever $X$ has distinct endpoints, since $\widetilde X$ and $X$ do not intersect, but also $\langle X,X\rangle=0$ if the endpoints of $X$ coincide, since the two contributions to the intersection number cancel, as one can easily check.

For a pair of distinct arcs $X,Y$, both pairings are a sum over pairs consisting of an end of $X$ and an end of $Y$ which meet in the same point in $M$.
For $\langle X,Y\rangle$ this follows from the description as intersection number of $\widetilde X$ and $Y$, and for $\langle GX,GY\rangle$ this follows from the description of $\mathrm{Ext}^\bullet(GX,GY)$ in terms of boundary paths, Proposition~\ref{prop_3cybasis}.

\begin{figure}
\centering
\includegraphics[scale=1]{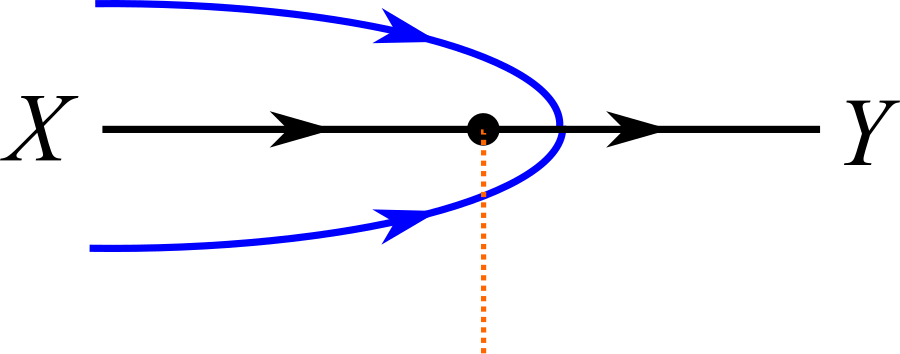}
\caption{Two arcs meeting at an endpoint. The branch cut for $\Sigma_\nu$ is indicated as dotted line. The choice of sheet of $\Sigma_\nu$ along the curves should be constant away from the branch cuts. The contribution to both pairings is $+1$.}
\label{fig_arcint}
\end{figure}

Suppose first, that the point where an end of $X$ meets an end of $Y$ is a zero. 
Then the contribution to both pairings is $\pm 1$. To check the sign, consider the situation of Figure~\ref{fig_arcint}, where the contribution to $\langle X,Y\rangle$ is $+1$, and there is a boundary path of even degree from $X$ to $Y$, likewise giving a contribution of $+1$ to $\langle GX,GY\rangle$.
If the arcs meet in a pole instead of zero, then there are additional boundary paths from $X$ to $Y$ to consider, but these cancel when taking the alternating sum of $\dim_{\mathbf k}\mathrm{Ext}^i(X,Y)$.
\end{proof}

\subsection{Genericity condition}
\label{subsec_generic}

In this subsection we study the geometric consequences of imposing a certain genericity condition, considered in~\cite{ks} in the general context of stability structures, on a point in $\mc Q_{S,\nu}$. 
We fix throughout a punctured surface $(S,M)$ with grading $\nu$ coming from a quadratic differential with simple zeros and poles.
The following definition already appeared in the introduction, but we repeat it here for convenience.
\begin{df}
A point $(J,\varphi,h)$ in $\mc Q_{S,\nu}$ is \textbf{generic} if given any pair $\gamma_1,\gamma_2\in\Gamma_{S,\nu}$ such that
\begin{enumerate}
\item
$\gamma_1$ and $\gamma_2$ are classes of saddle connections or flat cylinders, and
\item
$Z(\gamma_1)\in\RR_{>0} Z(\gamma_2)$,
\end{enumerate}
then $\gamma_1\in\QQ\gamma_2$.
\end{df}
The name is justified: The set of non-generic points is contained in the union of countably many hypersurfaces $W_L$, indexed by rank 2 sublattices $L\subset\Gamma$, where $Z(L)\subset\CC$ is contained in a line.
Thus, the set of generic points is generic in the usual sense of measure theory or point set topology.

Given a representative $(J,\varphi,h)$ of a point in $\mc Q_{S,\nu}$ we consider the subset $C$ of $S$ which is the closure of the set of finite-length horizontal trajectories, i.e. the union of those closed geodesic loops and saddle connections which have horizontal slope.
The following proposition shows that under the genericity condition the components of $C$ come in just a few basic types.

\begin{prop}
\label{prop_gencomp}
Let $C\subset S$ be as above, then each connected component of $C$ is either a single saddle connection or a flat cylinder with each boundary of one of the following types:
\begin{enumerate}
\item a saddle connection starting and ending at the same zero,
\item a saddle connection between two poles (necessarily distinct), or
\item a pair of saddle connections joining a pair of zeros, with a flat torus attached on the other side of the boundary. 
\end{enumerate}
(See Figure~\ref{fig_cylend}.)
\end{prop}

\begin{figure}[!ht]
\centering
\includegraphics[scale=.9]{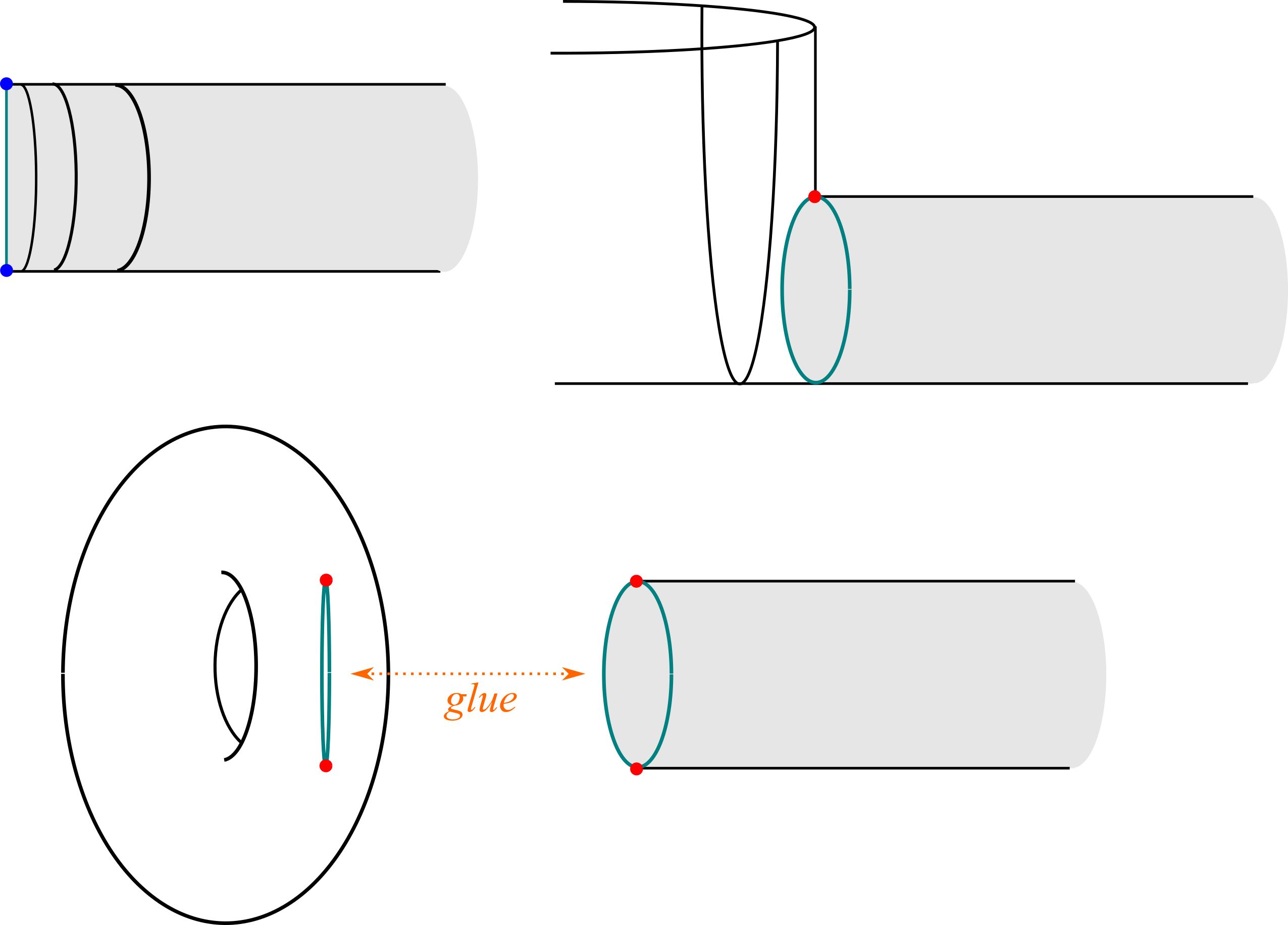}
\caption{Types of boundaries of a flat cylinder (gray). Blue dots are poles and red dots are zeros. \textit{Top left:} Cylinder ending in a saddle connection between poles. \textit{Top right:} Cylinder ending in a closed saddle connection. \textit{Bottom:} Cylinder ending in a pair of saddle connections between zeros, attached to a flat torus. }
\label{fig_cylend}
\end{figure}

\begin{proof}
Let $A$ and $B$ be distinct horizontal saddle connections.
The corresponding classes in $\Gamma_{S,\nu}$, denoted by the same letters, are well-defined up to sign, which depends on a choice of orientation and $\sqrt{\varphi}$ along $A$ and $B$. 
We can fix the sign by requiring that $Z(A)>0$ and $Z(B)>0$.

Suppose $A$ and $B$ meet in a conical point $p$, which is necessarily a zero as there is only a single horizontal leaf ending at a simple pole.
Let $q$ be the other endpoint of $A$, or $q=p$ if $A$ is closed, and similarly $r$ the other endpoint of $B$.
There are several cases to consider.
\begin{enumerate}
\item 
Case $p$, $q$, $r$ all distinct: In this case, $\langle A,B\rangle=\pm 1$, so $A$ and $B$ are $\QQ$-linearly independent, contradicting genericity.
\item
Case $q=r$: First, $p=q=r$ is impossible as there are only three horizontal leaves ending at a simple zero.
Also, $q$ cannot be a pole.
Thus we have two zeros connected by a pair of saddle connections.
Since $A$ and $B$ meet in a pair of points, we must have $\langle A,B\rangle\in\{-2,0,2\}$.
\begin{enumerate}
\item 
Subcase $\langle A,B\rangle = \pm 2$: Again, $A$ and $B$ are $\QQ$-linearly independent, contradicting genericity.
\item
Subcase $\langle A,B\rangle = 0$: In this case there is a flat cylinder, $Y$, which abuts in $A\cup B$.
To see this, note that the concatenation of $A$ and $B$ is a geodesic loop which meets two conical points, and after translating it slightly it becomes a geodesic loop which does not meet any conical points.
\begin{enumerate}
\item Sub$^2$case $Y$ non-separating:  There must be another conical point, $s\in S$, which is distinct from $p$ and $q$. Choose an embedded path $\alpha$ in $S$ which starts at $p$, avoids $Y$, and ends at $s$. 
Then $\langle \alpha,A\rangle=-\langle \alpha,B\rangle$. If $A$ and $B$ were $\QQ$-linearly dependent, they would thus necessarily be negatives of each other, but this contradicts $Z(A),Z(B)\in\RR_{>0}$, so this configuration contradicts genericity.
\item Sub$^2$case $Y$ separating: Let $W\subset S$ be the region of the surface on the other side of $A\cup B$, i.e. not containing $Y$. 
If $W$ contains another conical point distinct from $p$ and $q$, then we can use the same argument as in the previous sub$^2$case to conclude that the configuration is not generic.
Otherwise, if $W$ contains no other conical points, then by detaching $Y$ and gluing $A$ to $B$ we obtain a flat surface without conical points, which must be a flat $T^2$. Thus $W$ is a flat torus with the rest of $S$ attached along $A\cup B$.
Also, $A$ and $B$ have the same class in $\Gamma_{S,\nu}$ in this case.
\end{enumerate}
\end{enumerate}
\item Case $p=q\neq r$: In this case $A$ is closed, so there is a flat cylinder, $Y$, which abuts in $A$.
\begin{enumerate}
\item Subcase $r$ is a pole: There is another flat cylinder, $Y'$, which abuts in the cyclic sequence $ABB$ of saddle connections. Furthermore, $Y\neq Y'$, as the circumferences, $Z(A)$ and $Z(A)+2Z(B)$, are not equal.
Choose an embedded path $\alpha$ in $Y'$ which starts at $r$ and ends at a conical point on the other boundary of $Y'$.
Then $\langle \alpha,B\rangle=\pm 1$ but $\langle\alpha,A\rangle=0$, so $A$ and $B$ are $\QQ$-linearly independent and this contradicts genericity.
\item Subcase $r$ is a zero: 
\begin{enumerate}
\item Sub$^2$case $Y$ is non-separating: There is a conical point, $s$, not equal to $p$ or $r$. 
Choose an embedded path $\alpha$ in $S$ which connects $s$ to $r$ and avoids $Y$, then $\langle \alpha,A\rangle=0$ but $\langle \alpha,B\rangle=\pm 1$, so $A$ and $B$ are $\QQ$-linearly independent, contradicting genericity.
\item Sub$^2$case $Y$ is separating: Let $W$ be the region of $S$ attached to $Y$ along $A$. 
If $W$ contains conical points besides $p$ and $r$, then we can use the same argument as in the previous sub$^2$case to derive a contradiction. 
If $W$ contains no other conical points, then we get a relation $A+2B=0$ in $\Gamma_{S,\nu}$, contradicting $Z(A),Z(B)\in\RR_{>0}$.
\end{enumerate}
\end{enumerate}
\end{enumerate}
Note that the only case above which did not lead to a contradiction is that of a pair of saddle connections joining zeros and forming the boundary of a flat cylinder along which a flat $T^2$ is attached.

Consider now one of the boundaries of a flat cylinder $Y$ in $S$ foliated by closed horizontal leaves.
As one approaches the boundary, the closed leaves break into a cyclic sequence of horizontal saddle connections and conical points, possibly with some of them repeating.
If the boundary consists of a single saddle connection, then it must be either closed, connecting a zero to itself, which is the first case in the proposition, or a saddle connection between a pair of poles, which is the second case in the proposition.
If the boundary consists of multiple saddle connections, then by the above analysis the only possible configuration is a pair of saddle connections connecting a pair of zeros, which is the third case in the proposition.
\end{proof}

\subsection{Orientation data}
\label{subsec_orientation}

We recall, briefly, the definition of orientation data from~\cite{ks}.
Let $\mc C$ be an ind-constructible $\mathbf k$-linear 3CY category.
There is a super line bundle, $D$, on the ind-constructible stack $\mathrm{Ob}(\mc C)$ with fiber $D_X=\mathrm{Det}(\Ext^\bullet(X,X))$ over $X\in\mathrm{Ob}(\mc C)$.
Moreover, on the ind-constructible stack of triangles $X\to Y\to Z\to $ in $\mc C$ there is a canonical isomorphism
\begin{equation}\label{triangle_det_iso}
D_Y\cong D_X\otimes D_Z\otimes \mathrm{Det}(\Ext^\bullet(X,Z))^{\otimes 2}
\end{equation}
of super line bundles.
\textit{Orientation data} on $\mc C$ is given by a super line bundle $\sqrt{D}$ together with a choice of isomorphism $\sqrt{D}^{\otimes 2}\xrightarrow{\cong}D$ such that there exists an isomorphism of super line bundles on the ind-constructible stack of triangles \[
\sqrt{D}_Y\cong \sqrt{D}_X\otimes\sqrt{D}_Z\otimes\mathrm{Det}(\Ext^{\bullet}(X,Z))
\]
whose square is~\eqref{triangle_det_iso}.

We note that the definition of the DT invariants requires only the choice of  $\sqrt{D}$ with $\sqrt{D}^{\otimes 2}\cong D$ --- and only for semistable objects. 
The compatibility with triangles is needed to show multiplicativity of the integration map from the Hall algebra to the quantum torus, and ultimately the wall-crossing formula.

A candidate for $\sqrt{D}$ seems, at first, to be the constructible super line bundle $D_{\leq 1}$ with fiber
\[
\left(D_{\leq 1}\right)_X=\mathrm{Det}\left(\Ext^{\leq 1}(X,X)\right)
\]
which comes with a natural isomorphism $\left(D_{\leq 1}\right)^{\otimes 2}\cong D$ thanks to the 3CY structure.
However, $D_{\leq 1}$ fails to be compatible with triangles in general.
The obstruction is the possible non-triviality of the super line bundle with fiber 
\begin{equation}\label{leq1_obstruction}
\left(D_{\leq 1}\right)_Y\otimes \left(D_{\leq 1}\right)_X^{-1}\otimes \left(D_{\leq 1}\right)_Z^{-1}\otimes\mathrm{Det}\left(\Ext^{\bullet}(X,Z)\right)^{-1}
\end{equation}
over a triangle $X\to Y\to Z\to $.
Giving orientation data for $\mc C$ is then equivalent to tensoring $D_{\leq 1}$ by a suitable square root of the trivial line bundle which cancels the obstruction, see~\cite{ks} for more details.

Suppose that $\mc C$ is actually $\ZZ\times \ZZ/2$-graded and that the 3CY pairing is odd with respect to the additional $\ZZ/2$-grading. 
This is the case for the category $\mc C=\mathrm{Tw}_{\ZZ/2}(\mc A_\XX)$ constructed in Section~\ref{sec_3cysurf}, from which $\mc C(S,\nu)$ is obtained by passing to the triangulated completion of the full subcategory which is the image of the functor $G$ from $\mc F(S\setminus M,\nu)$.
In general, if $V$ is a $\ZZ\times\ZZ/2$-graded vector space with non-degenerate pairing $V\otimes V\to\mathbf{k}[-3,1]$, then there are canonical isomorphisms
\begin{align*}
\mathrm{Det}\left(V^{\bullet,0}\right)&\cong\mathrm{Det}\left(V^{\leq 1,0}\right)\otimes \mathrm{Det}\left(V^{\geq 2,0}\right) \\
&\cong \mathrm{Det}\left(V^{\leq 1,0}\right)\otimes \mathrm{Det}\left(V^{\leq 1,1}\right) \cong \mathrm{Det}\left(V^{\leq 1,\bullet}\right)
\end{align*}
where duality induces the isomorphism $V^{\geq 2,0}\cong \left(V^{\leq 1,1}\right)^\vee[-3]$ and the additional $\ZZ/2$-grading is not used in the definition of $\mathrm{Det}$.
Thus, if $D_+$ is the super line bundle with fiber $\left(D_+\right)_X=\mathrm{Det}\left(\mathrm{Ext}^{\bullet,0}\right)$, then $D_+\cong D_{\leq 1}$.
Moreover, the obstruction~\eqref{leq1_obstruction} vanishes on triangles classified by $\delta\in\Ext^{1,0}(Z,X)$ as follows from the existence of the canonical isomorphisms
\[
\mathrm{Det}\left(\Ext^{\bullet,\bullet}(X,Z)\right)\cong \mathrm{Det}\left(\Ext^{\bullet,0}(X,Z)\right)\otimes \mathrm{Det}\left(\Ext^{\bullet,0}(Z,X)\right)
\]
coming from the duality pairing, and
\[
\left(D_+\right)_Y\cong \left(D_+\right)_X\otimes \left(D_+\right)_Z\otimes \mathrm{Det}\left(\Ext^{\bullet,0}(X,Z)\right)\otimes \mathrm{Det}\left(\Ext^{\bullet,0}(Z,X)\right)
\]
coming from compatibility of the long exact sequences with the additional $\ZZ\times \ZZ/2$-grading.

It is not clear, first, whether $D_+\cong D_{\leq 1}$ is compatible with triangles classified by $\delta\in \Ext^{1,\bullet}(Z,X)$ of mixed parity, and second whether the orientation data extends to the triangulated completion $\mathrm{Tw}^+(\mc C)$, where we forget the additional $\ZZ/2$ grading on $\mc C$.
In this sense we have only \textit{partial orientation data} on $\mathrm{Tw}^+(\mc C)$, in particular on the categories $\mc C(S,\nu)$.
However, it turns out that this is good enough for our purposes:

\begin{prop}
Let $\sigma\in\mathrm{Stab}(\mc C(S,\nu))$ lie in the distinguished component identified with $\mc Q_{S,\nu}$ by Theorem~\ref{thm_stab_C}.
Then all $\sigma$-semistable objects lift to the $\ZZ\times\ZZ/2$-graded category $\mathrm{Tw}_{\ZZ/2}(\mc A_\XX)$, as do Harder--Narasimhan filtrations of objects in the essential image of the functor from $\mc F(S\setminus M,\nu)$, in particular all objects in the heart of the t-structure given by $\sigma$.
\end{prop}

The proposition implies that the motivic DT invariants for the given component of $\mathrm{Stab}(\mc C_{g,n})$ are well-defined and that the wall-crossing formula holds as one moves in this component.

\begin{proof}
The first claim follows since semistable objects lift to  $\mc F(S\setminus M,\nu)$ by construction and all objects in that category lift to the $\ZZ\times\ZZ/2$-graded category, see Lemma~\ref{lem_inv_trivial}.
The second claim follows from the explicit description of Harder--Narasimhan filtrations in~\cite[Section 5.5]{hkk} and the proof of Lemma~\ref{lem_inv_trivial}.
Indeed, any indecomposable $X\in \mc F(S\setminus M,\nu)$ is given by an immersed curve $\gamma$ with local system by~\cite{hkk}.
The Harder--Narasimhan filtration of $X$ corresponds to a presentation of $\gamma$ as a concatenation of (geodesic) segments.
Thus, if one chooses an arbitrary orientation on $\gamma$ and the induced orientation on the segments, one sees, as in the proof of Lemma~\ref{lem_inv_trivial}, that $X$ is an upper-triangular deformation of its semistable factors by some $\delta\in\Hom^{1,0}$. 
\end{proof}

\subsection{Proof of the second main theorem}
\label{subsec_dtcomp}

Let $S$ be a compact Riemann surface and $\varphi$ a quadratic differential with simple zeros and/or simple poles on $S$.
Also fix an algebraically closed field $\mathbf k$ with $\mathrm{char}(\mathbf k)=0$.
In Section~\ref{sec_3cysurf} we constructed a proper 3CY category, $\mc C(S,\mathrm{hor}(\varphi))$, and in Section~\ref{sec_stability} a stability condition on it, whose stable objects are, roughly speaking, the finite length geodesics on $(S,\varphi)$, regarded as a flat surface with conical singularities.
In the previous subsection we also constructed (partial) orientation data on $\mc C(S,\mathrm{hor}(\varphi))$.
Thus, the DT invariants of  $\mc C(S,\mathrm{hor}(\varphi))$ are defined.
We compute them in this subsection in terms of numbers of saddle connections and closed cylinders for any $\varphi$ which is generic in the sense of Subsection~\ref{subsec_generic}.
This is our second main theorem, which we repeat here for convenience.

\begin{theorem}
\label{thm2}
Fix a generic pair $(S,\varphi)$ as above, then the refined DT invariants of the corresponding stability condition on $\mc C(S,\mathrm{hor}(\varphi))$ are given by
\[
\Omega(\gamma)=N_{++}(\gamma)+2N_{+-}(\gamma)+4N_{--}(\gamma)+\left(q^{1/2}+q^{-1/2}\right)N_c(\gamma)
\]
where $\gamma\in\Gamma_{S,\mathrm{hor}(\varphi)}$ and $N_{++}(\gamma)$ (resp. $N_{+-}(\gamma)$, resp. $N_{--}(\gamma)$) is the number of saddle connections with class $\gamma$ between distinct zeros (resp. a zero and a pole, resp. distinct poles) of $\varphi$ and $N_c(\gamma)$ is the number of flat cylinders with class $\gamma$.
\end{theorem}

The strategy of the proof is to show that, for a given phase $\phi$, the subcategory of $\mc C(S,\mathrm{hor}(\varphi))$ of semistable objects of phase $\phi$ is, thanks to the genericity condition, a direct sum of (abelian) 3CY categories constructed from certain simple quivers with potential.
By their motivic nature, DT invariants are additive with respect to such direct sums, so it suffices to know the DT invariants of these quivers.

\subsubsection*{3CY categories from quivers with potential}
Suppose $Q$ is a quiver and $\mathbf kQ$ its path algebra. 
A \textit{potential (for $Q$)} is an element $W\in \mathbf kQ/[\mathbf kQ,\mathbf kQ]$, i.e. a linear combination of cyclic paths.

\begin{remark}
It is also possible to consider potentials which are infinite linear combinations of longer and longer paths by replacing $\mathbf kQ$ with its completion with respect to the length filtration, but we will not need this.
\end{remark}

Starting from any quiver with potential, Ginzburg~\cite{ginzburg_cy} constructs a DG-algebra $D(Q,W)$, concentrated in non-positive degrees. 
The associative algebra $H^0(D(Q,W))$ is a quotient of the path algebra of $Q$ by certain relations coming from the potential (its partial first derivatives).
From $D(Q,W)$ we obtain an $\mathrm{Ext}$-finite triangulated 3CY category, denoted $\mc C(Q,W)$, as the subcategory of the derived category of DG-modules over $D(Q,W)$ of modules with finite-dimensional cohomology.
Since $D(Q,W)$ vanishes in positive degrees, $\mc C(Q,W)$ carries a natural bounded t-structure with heart $\mc A(Q,W)\subset \mc C(Q,W)$ equivalent to the abelian category of finite-dimensional modules over $H^0(D(Q,W))$.
We can easily construct stability conditions on $\mc C(Q,W)$ as follows: Fix a number $z_i$ in the upper half-plane for each vertex $i\in Q_0$ of the quiver. 
Then there is a stability condition with heart $\mc A(Q,W)$ and central charge which sends the 1-dimensional $H^0(D(Q,W))$-module concentrated at the $i$-th vertex to $z_i\in\CC$ and depends only on the dimension vector of the module.
In the special case where all $z_i$ have the same argument, all objects in $\mc A(Q,W)$ become semistable, so the DT invariants no longer depend on the $z_i$'s, but only the pair $(Q,W)$.
(There is a natural choice of orientation data on $\mc C(Q,W)$ as explained in~\cite{ks}.)

There is an alternative (Koszul-dual) construction of a 3CY category starting from a quiver with potential described in~\cite{ks}. 
This yields in general a full subcategory of $\mc C(Q,W)$ --- the subcategory generated by nilpotent representations, or equivalently representations over the completed version of $D(Q,W)$. 
The relation with Ginzburg's construction goes as follows: If $A(Q,W)$ is the $A_\infty$-algebra constructed in~\cite{ks}, then the completed version of $D(Q,W)$ is obtained as the topological dual of the reduced cobar coalgebra of $A(Q,W)$ (the tensor coalgebra on the kernel of the augmentation, shifted down by 1).

As an example, consider the quiver $Q$ with a single vertex, a single loop $x$, and potential $W=x^3$.
The algebra defined by Kontsevich--Soibelman has basis $1,x,x^*,y$ of degrees $0,1,2,3$, respectively, relation $xx^*=x^*x=y$ and, coming from the potential, the relation $x^2=x^*$.
This is just the algebra $\mathbf k\langle x\rangle/x^4$ which appears in Proposition~\ref{prop_arcext}.
Taking the dual of the reduced cobar coalgebra yields the DG-algebra $D(Q,W)$ generated by $\alpha,\alpha^*,\beta$ of degrees $0,-1,-2$, respectively, with differential $d\alpha=0$, $d\alpha^*=\alpha^2$, $d\beta=\alpha\alpha^*-\alpha^*\alpha$. 
In particular, $H^0(D(Q,W))=\mathbf k[\alpha]/\alpha^2$.

\begin{prop}
\label{prop_littlequivers}
The table below gives the DT invariants of some quivers with potential.
\begin{center}
\begin{tabular}{c|c|c}
Quiver & Potential & Counting series \\
\hline
\begin{tikzcd} \bullet \end{tikzcd} & $0$ & $\mathbf E(x)$ \\
\begin{tikzcd} \bullet\arrow[loop right] \end{tikzcd} & $0$ & $\mathbf E(-q^{1/2}x)^{-1}$ \\
\begin{tikzcd} \bullet\arrow[loop right,"x"] \end{tikzcd} & $x^3$ & $\mathbf E(x)^2$ \\
\begin{tikzcd} \bullet\arrow[loop left,"x"]\arrow[loop right,"y"] \end{tikzcd} & $x^3+y^3$ & $\mathbf E(x)^4\mathbf E(-q^{1/2}x^2)^{-1}$ \\
\begin{tikzcd} \bullet\arrow[r,bend left] & \bullet\arrow[l,bend left]\end{tikzcd} & $0$ & $\mathbf E(x)^2\mathbf E(-q^{1/2}x^2)^{-1}$ 
\end{tabular}
\end{center}
\end{prop}

\begin{proof}
The first corresponds to the 3CY category generated by a single spherical object, see~\cite[Section 6.4]{ks}. 
The one loop case is similar and well-known. 

The motivic DT invariants of the one-loop quiver with potential were computed in~\cite{davison_meinhardt}.
In particular, for the potential $x^k$, $k\geq 3$, the result is $\Omega=\LL^{-1/2}(1-[\mu_k])$.
Here $[\mu_k]$ is the equivariant motive of $k$-th roots of unity, which the twisted Serre polynomial sends to $1-(k-1)q^{1/2}$.
Thus, the refined DT invariant is simply $k-1$.

For the fourth quiver in the table we use that the one-loop quiver with potential $x^3$ has counting series $\mathbf E(x)^2$ by the above mentioned result, which may be written as
\[
\mathbf E(x)^2=\sum_{n=0}^{\infty}\frac{B_nq^{n^2/2}}{\chi_q(GL_n)}x^n.
\]
where
\[
B_n:=\sum_{k=0}^n\binom{n}{k}_q.
\]
It follows from the Thom-Sebastiani theorem that the counting series for the two-loop quiver with potential $x^3+y^3$ is then
\[
\sum_{n=0}^{\infty}\frac{B_n^2q^{n^2/2}}{\chi_q(GL_n)}x^n=\mathbf E(x)^4\mathbf E(-q^{1/2}x^2)^{-1}
\]
where the equality is Lemma~\ref{lem_w33dt} in the appendix.

For the last quiver in the table a direct computation as for the previous one is possible, however we will instead prove the factorization using the wall-crossing formula.
Write a quiver representation 
\[
\begin{tikzcd}
V_1\arrow[r,bend left,"T_1"] & V_2\arrow[l,bend left,"T_2"]
\end{tikzcd}
\]
as a quadruple $(V_1,T_1,V_2,T_2)$.
Choose $z_1,z_2\in\CC$, with $\mathrm{Im}(z_i)>0$ and $\mathrm{Arg}(z_1)<\mathrm{Arg}(z_2)$, giving a central charge $Z(\dim V_1,\dim V_2):=\dim V_1z_1+\dim V_2z_2$ and a resulting notion of slope stability.
There are three phases in $(0,\pi)$ with non-zero semistable objects.
\begin{enumerate}
\item
$(V,0,0,0)$, semistable of phase $\mathrm{Arg}(z_1)$ 
\item
$(V_1,T_1,V_2,T_2)$ with $T_2$ an isomorphism, semistable of phase $\mathrm{Arg}(z_1+z_2)$
\item 
$(0,0,V,0)$, semistable of phase $\mathrm{Arg}(z_2)$ 
\end{enumerate}
There is a fully faithful functor from the category of representations of the one loop quiver to the category of representations of the present quiver given by $(V,T)\mapsto (V,T,V,1_V)$, and the image is essentially the full subcategory of semistable objects of phase $\mathrm{Arg}(z_1+z_2)$.
It follows from the results for the first two quivers in the table and the wall-crossing formula~\cite{ks}, that the counting series is $\mathbf E(x)\mathbf E(-q^{1/2}x^2)^{-1}\mathbf E(x)$ as claimed.
\end{proof}

We will also need a variant of the previous Proposition where we count nilpotent representations only.
The relation between the two turns out to be very simple.

\begin{prop}
\label{prop_quivers_nil}
For the 2nd, 4th, and 5th quiver with potential listed in Proposition~\ref{prop_littlequivers}, the abelian category $\mc A(Q,W)$ is a direct sum of the full subcategory $\mc A(Q,W)_{\mathrm{nil}}$ of nilpotent representations and a category equivalent to the category of finite dimensional $\mathbf k[x^{\pm 1}]$-modules.
The counting series for  $\mc A(Q,W)_{\mathrm{nil}}$ is obtained from the one for $\mc A(Q,W)$ by replacing the $\mathbf E(-q^{1/2}x^2)^{-1}$ factor by $\mathbf E(-q^{-1/2}x^2)^{-1}$.
\end{prop}

\begin{proof}
For the second quiver the first statement is easy to see geometrically, since $\mc A(Q,W)$ is just the category of torsion sheaves on $\AA^1_{\mathbf k}$, which is a direct sum of the category of sheaves supported at the origin (corresponding to nilpotent representations) and the category category of sheaves supported on $\AA^1_{\mathbf k}\setminus \{0\}$, which is equivalent to the category of finite dimensional $\mathbf k[x^{\pm 1}]$-modules, $\mathrm{Mod}_{\mathrm{fd}}(\mathbf k[x^{\pm 1}])$.

For the fourth quiver, the 2-loop quiver with potential $x^3+y^3$, we are considering modules over the algebra $\mathbf k\langle x,y\rangle/(x^2,y^2)$, which is just a vector space $V$ with two linear endomorphisms $A,B:V\to V$ with $A^2=B^2=0$. 
There is a functor from $\mathrm{Mod}_{\mathrm{fd}}(\mathbf k[x^{\pm 1}])$ which sends a vector space $W$ with automorphism $T:W\to W$ to the triple $(V,A,B)$ with $V:=W\oplus W$, $A(a,b):=(T(b),0)$, $B(a,b):=(0,a)$.
We leave it as a linear algebra exercise to check that this is fully faithful and its image a complement to the nilpotent modules.

For the fifth quiver, the 2-cycle, there is again a geometric interpretation: $\mc A(Q,W)$ is equivalent to the category of $\ZZ/2$-equivariant torsion sheaves on $\AA^1_{\mathbf k}$ with $\ZZ/2$ acting by $x\mapsto -x$, so the same reasoning as in the first example works.

To show the second claim, note that counting series of $\mathrm{Mod}_{\mathrm{fd}}(\mathbf k[x^{\pm 1}])$ is $\mathbf E(-q^{1/2}x)^{-1}\mathbf E(-q^{-1/2}x)$ corresponding to $\Omega=q^{1/2}-q^{-1/2}=q^{-1/2}(q-1)$, where $q-1$ is the Serre polynomial of $\AA^1_{\mathbf k}\setminus 0$.
Thus, the counting series of the complementary category, which consists of nilpotent representations, is obtained by replacing the $\mathbf E(-q^{1/2}x)^{-1}$ factor (found in the corresponding rows of the table in Proposition~\ref{prop_littlequivers}) by $\mathbf E(-q^{-1/2}x)^{-1}$.
\end{proof}

\begin{proof}[Proof of Theorem~\ref{thm2}]
Fix a phase $\phi\in\RR$ and consider, as in Subsection~\ref{subsec_generic}, the set $C\subset S$, depending only on $\phi\in \RR/\ZZ\pi$, which is the union of saddle connections and closed loops of phase $\phi$.
The category $\mc C(S,\nu)_\phi$ of semistable objects of phase $\phi$ is a finite direct sum of categories corresponding to connected components of $C$, hence the DT invariants are a sum of contributions from each such component.
If $C_0\subset C$ is a connected component, then by the genericity assumption and Proposition~\ref{prop_gencomp} it is either a single saddle connection or the closure of a flat cylinder.

Suppose first that $C_0$ is a saddle connection. 
If $C_0$ connects two distinct zeros, then it corresponds to a spherical object in $\mc C(S,\nu)$ by Proposition~\ref{prop_arcext}, thus contributes $1$ to $\Omega(\gamma)$, where $\gamma$ is the class of $C_0$, by Proposition~\ref{prop_littlequivers}.
If $C_0$ connects a zero to itself, then it bounds a flat cylinder of the same phase, contradicting the assumption that $C_0$ consists of a single saddle connection.
If $C_0$ connects a pole to a zero, then the corresponding object $X\in\mc C(S,\nu)$ generates a 3CY category isomorphic to the one coming from the one-loop quiver with potential $x^3$ by Proposition~\ref{prop_arcext}. 
Thus the contribution to the $\Omega(\gamma)$ is $2$ by Proposition~\ref{prop_littlequivers}.
If $C_0$ connects two distinct poles, then it bounds a flat cylinder of the same phase, contradicting the assumption that $C_0$ consists of a single saddle connection.

Suppose now that $C_0$ the closure of a flat cylinder and let $\gamma\in\Gamma$ be its homology class.
We will show that there is a contribution of $2q^{-1/2}$ to $\Omega(\gamma)$ coming from the two boundaries of $C_0$ and a contribution of $q^{1/2}-q^{-1/2}$ coming from the interior of the cylinder, thus a total contribution of $q^{1/2}+q^{-1/2}$, in agreement with the formula in the theorem.

First we consider the boundary of $C_0$.
By Proposition~\ref{prop_gencomp} there are three ways in which the cylinder abuts in a union of saddle connections at each end.
First, consider the case of a closed saddle connection, $X$, with both endpoints at the same zero.
The category generated by $X$ is equivalent to the category of nilpotent representations of the one loop quiver without potential.
The contribution to the refined DT invariant $\Omega(\gamma)$ is $q^{-1/2}$ by Proposition~\ref{prop_quivers_nil}.
Second, if the cylinder abuts in a saddle connection, $X$, between two distinct poles then the contribution to the DT invariant $\Omega(\gamma)$ is again $q^{-1/2}$ but there is also a contribution of $4$ to the DT invariant $\Omega(\gamma/2)$ where $\gamma/2$ is the class of $X$, counting the saddle connection itself.
Finally, if the cylinder abuts in a pair of saddle connections between a pair of zeros, then the contribution to the DT invariant $\Omega(\gamma)$ is again $q^{-1/2}$ and there is also a contribution of $2$ to the DT invariant $\Omega(\gamma/2)$, where $\gamma/2$ is the class of either saddle connection, counting the two saddle connections.
Hence, for all cylinders there is a contribution of $2q^{-1/2}$ to $\Omega(\gamma)$ coming from the two boundaries.

Finally, we consider the interior of $C_0$, which is a cylinder without its boundary, foliated by a one-parameter family of closed geodesics.
As shown in \cite{hkk} and sketched in the proof of Proposition~\ref{prop_loopext}, the isotopy class of these closed geodesics corresponds to a $\mathbf k^\times$-family of spherical objects in $\mc F(S\setminus M,\nu)$. 
The image under $G$ is a family of objects $A_\lambda\in \mc C(S,\nu)$, $\lambda\in\mathbf k^\times$, with $\mathrm{Ext}^\bullet(A_\lambda,A_\lambda)\cong H^\bullet(S^1\times S^2;\mathbf k)$ by Proposition~\ref{prop_loopext}.
Moreover, the $A_\lambda$ are all semistable by Theorem~\ref{thm_stab_C}.
We claim that the full $A_\infty$-subcategory of $\mc C_{S,\nu}$ whose objects are the $A_\lambda$'s is quasi-equivalent to the full subcategory of $\mathbf k[x^{\pm 1},x^*]$-modules, $|x|=0$, $|x^*|=-1$, whose objects are the one-dimensional simple modules in the heart of the canonical t-structure.
Indeed, up to isomorphism there is a $\mathbf k^\times$-family $B_\lambda$ of such simple modules and $\mathrm{Ext}^\bullet(B_\lambda,B_\lambda)\cong H^\bullet(S^1\times S^2;\mathbf k)$ by Koszul duality, where $x$ corresponds to the fundamental class of $S^1$ and $x^*$ to the fundamental class of $S^2$. 
(The $\mathrm{Ext}$-algebra depends only the formal neighborhood of $x=\lambda$, so is unaffected by localizing away from $x=0$.)
Note that $\mathbf k[x,x^*]$ is quasi-isomorphic to the Ginzburg algebra constructed from the one-loop quiver with trivial potential, $\mathbf k\langle x,x^*,t\rangle$, $t=-2$, $dt=[x,x^*]$, which has DT invariant $\Omega=q^{1/2}$ by Proposition~\ref{prop_littlequivers}.
Localizing away from $x=0$ has the effect of deleting the nilpotent representations from the category, so by Proposition~\ref{prop_quivers_nil} we need to subtract $q^{-1/2}$ to get the DT-invariant for $\mathbf k[x^{\pm 1},x^*]$, which is $q^{1/2}-q^{-1/2}$.
\end{proof}

\appendix

\section{Quantum dilogarithm}
\label{sec_qdilog}

This section contains, first, a dictionary for various ways of expressing the quantum dilogarithm found in the literature, and second, a proof of an identity needed Subsection~\ref{subsec_dtcomp}.

The quantum dilogarithm, using the normalization of \cite{ks}, is
\[
\mathbf E(x):=\sum_{n=0}^\infty\frac{(-1)^nq^{n/2}x^n}{(1-q)(1-q^2)\cdots(1-q^n)}
\]
as a formal power series in $x$.
This has an interpretation as ``counting'' objects in the category of finite-dimensional vector spaces, since the coefficient of $x^n$ is $q^{n^2/2}/\chi_q(GL_n)$ where $\chi_q(GL_n)$ is the Serre polynomial of the general linear group. (The unexpected factor $q^{n^2/2}$ is a feature of motivic DT theory.)
The series $\mathbf E(x)$ is essentially the same as the $q$-exponential defined as
\[
\exp_q(x):=\sum_{n=0}^{\infty}\frac{x^n}{[n]_q!}=\mathbf E\left((q^{1/2}-q^{-1/2})x\right)
\]
where $[n]_q!$ is the $q$-factorial, and the (infinite) $q$-Pochhammer symbol defined as
\[
(x;q)_\infty:=\prod_{n=0}^\infty(1-q^nx)=\sum_{n=0}^\infty\frac{(-1)^nq^{n(n-1)/2}x^n}{(1-q)(1-q^2)\cdots (1-q^n)}=\mathbf E\left(-q^{-1/2}x\right)^{-1}.
\]
Various permutations of these naming conventions can be found in the literature.

For $f_n\in\ZZ[t^{\pm}]$, $n\geq 1$, a sequence on Laurent polynomials with integer coefficients define
\begin{equation}
\label{sym_formula}
\mathrm{Sym}\left(\sum_{n\geq 1}\frac{f_n(q^{1/2})x^n}{q^{1/2}-q^{-1/2}}\right):=\exp\left(\sum_{n\geq 1}\sum_{k\mid n}\frac{(-1)^{k-1}f_{n/k}(-(-q^{1/2})^k)x^n}{k(q^{k/2}-q^{-k/2})}\right).
\end{equation}
This has a more conceptual definition (and generalization) in the language of (complete) $\lambda$-rings where 
\[
\mathrm{Sym}(x)=1-\lambda^1(-x)+\lambda^2(-x)-\lambda^3(-x)\pm\ldots
\]
and \eqref{sym_formula} is the special case of this for the $\lambda$-ring $\ZZ((-q^{1/2}))[[x]]$ where $-q^{1/2}$ and $x$ are adjoined as line elements, i.e. $\lambda^n(-q^{1/2})=\lambda^n(x)=0$ for $n>1$.
The relation of \eqref{sym_formula} to the quantum-dilogarithm is
\[
\mathrm{Sym}\left(\sum_{n\geq 1}\frac{f_n(q^{1/2})x^n}{q^{1/2}-q^{-1/2}}\right)=\prod_{n=1}^\infty\prod_{k\in\ZZ}\mathbf E\left((-q^{1/2})^kx^n\right)^{(-1)^kc_{n,k}}
\]
where $f(t)=\sum_{k\in\ZZ}c_{n,k}t^k$, $c_{n,k}\in\ZZ$.

\begin{lemma}
\label{lem_w33dt}
\begin{equation}
\sum_{n=0}^{\infty}\left(\sum_{k=0}^n\binom{n}{k}_q\right)^2\frac{q^{n^2/2}x^n}{\chi_q(GL_n)}=\mathrm{Sym}\left(\frac{4x+q^{1/2}x^2}{q^{1/2}-q^{-1/2}}\right)=\mathbf E(x)^4\mathbf E(-q^{1/2}x^2)^{-1}
\end{equation}
\end{lemma}

\begin{proof}
In terms of $q$-multinomial coefficients (counting flags of given type) we have
\[
E(x)^4=\sum_{n=0}^\infty\left(\frac{q^{n^2/2}}{\chi_q(GL_n)}\sum_{i_1+i_2+i_3+i_4=n}\binom{n}{i_1,i_2,i_3,i_4}_q\right)x^n.
\]
Furthermore,
\[
E\left(-q^{1/2}x^2\right)^{-1}=\sum_{n=0}^\infty\frac{q^{n^2}}{\chi_q(GL_n)}x^{2n}
\]
from its interpretation as counting representations of the one-loop quiver.
Multiplying the two and applying 
\begin{equation}\label{qbinom_gl}
q^{k(n-k)}\chi_q(GL_k)\chi_q(GL_{n-k})\binom{n}{k}_q=\chi_q(GL_n)
\end{equation}
we get
\begin{align*}
E(x)^4E\left(-q^{1/2}x^2\right)^{-1}&= \\
\sum_{n=0}^{\infty}\frac{q^{n^2/2}}{\chi_q(GL_n)}& \sum_{i_1+\ldots+i_4+2j=n}\chi_q(GL_j)\binom{n}{i_1,\ldots,i_4,j,j}_qx^n.
\end{align*}
Thus it remains to show the equality
\begin{equation}\label{twosubspace}
A_n^2=\sum_{i_1+\ldots+i_4+2j=n}\chi_q(GL_j)\binom{n}{i_1,\ldots,i_4,j,j}_q
\end{equation}
of polynomials in $q$.
It suffices to give a combinatorial proof assuming $q$ is a prime power.

The left-hand side of \eqref{twosubspace} counts pairs of subspaces in $\mathbb F_q^n$.
A pair of subspaces $A,B\subset \mathbb F_q^n$ is equivalently determined by the flag
\[
0\subset A\cap B\subset A\subset A+B\subset F_q^n
\]
together with a complement of $A\subset A+B$.
The number of such complements is the same as the number of linear maps $f:(A+B)/A\to A/(A\cap B)$.
Such a linear map provides a refined flag of length six
\[
0\subset A\cap B \subset \mathrm{Im}(f)\subset A \subset \mathrm{Ker}(f) \subset A+B\subset F_q^n
\]
with two of the subquotients, $(A+B)/\mathrm{Ker}(f)$ and $\mathrm{Im}(f)/(A\cap B)$, naturally identified.
Note however that this --- flags of length six with two of the subquotients identified --- is precisely what the right-hand side of \eqref{twosubspace} counts.
Hence \eqref{twosubspace} and the lemma are proven.
\end{proof}

\bibliographystyle{alpha}
\bibliography{surf_cy3}

\Addresses

\end{document}